\documentclass[11pt]{article}

\usepackage{amssymb,graphics,amsmath,amsthm,amsopn,amstext,amsfonts,color,fullpage,fancybox,
hyperref,bm}
\usepackage{graphicx}
\usepackage{verbatim}
\usepackage{algorithm}
\usepackage{algorithmicx}
\usepackage{algpseudocode}
\usepackage{longtable}
\usepackage{multirow}
\usepackage{dsfont}
\usepackage{mathrsfs}
\usepackage{caption}
\usepackage{subcaption}
\usepackage{setspace}
\usepackage{color,cite}

\newcommand{\gap}{\vspace{0.1in}}
\newcommand{\epc}{\hspace{1pc}}
\newcommand{\wh}{\widehat}
\newcommand{\wt}{\widetilde}
\newcommand{\thalf}{\frac{1}{2}}

\definecolor{UniBlue}{RGB}{254,107,1}

\newtheorem{theorem}{Theorem}
\newtheorem{proposition}[theorem]{Proposition}
\newtheorem{lemma}[theorem]{Lemma}
\newtheorem{corollary}[theorem]{Corollary}
\newtheorem{definition}[theorem]{Definition}

\newtheorem{remark}[theorem]{Remark}

\title{Quasi-difference-convexity: Modernization of \\
Quasi-differentiable Optimization\footnote{Both authors 
are affiliated with the Daniel J.\ Epstein Department of Industrial and Systems 
Engineering, University of Southern
California, Los Angeles, U.S.A.\ 90089. This work was based on research supported by the
U.S.\ Air Force Office of Scientific Research under grant FA9550-22-1-0045.
{\tt Emails: jongship@usc.edu and yulinpen@usc.edu}.  The authors gratefully acknowledge several
fruitful discussions with Xinyao Zhang, Kim-Chuan Toh, Junyi Liu, Shaoning Han, and Ying Cui 
(in reverse order of their last names) in the initial phase of and during this work.}}

\author{Jong-Shi Pang \and Yulin Peng}

\date{Original July 2025}

\begin{document}

\maketitle

\begin{abstract}
\noindent Quasi-differentiable functions were introduced by Pshenichnyi in a 1969 
monograph written 
in Russian and translated in an English version in 1971.  This class of nonsmooth functions 
was studied extensively in two decades since but has not received much attention in 
today's wide optimization literature.  This regrettable omission is in spite of the fact 
that many functions in modern day applications of optimization can be shown to be 
quasi-differentiable.  In essence, a quasi-differentiable function is one whose directional 
derivative at an arbitrary reference vector, as a function of the direction, is the 
difference of two positively homogenous, convex functions.  Thus, to bring quasi-differentiable 
functions closer to the class of difference-of-convex functions that has received fast growing 
attention in recent years in connection with many applied subjects, we propose to rename 
quasi-differentiable functions as quasi-difference-convex (quasi-dc) functions.    
Besides modernizing and advancing this class of nonconvex and nondifferentiable functions, 
our research aims to put together a unified treatment 
of iterative convex-programming based descent algorithms for solving a broad class of 
composite quasi-dc programs and to establish their subsequential convergence, sequential 
convergence, and rates of convergence; the latter two topics are in line with the modern 
focus of such analysis for convex programs and some extensions and are departures from 
the sole emphasis of subsequential convergence in the traditional studies 
of quasi-differentiable optimization.  Through this research, we have gained significant
new insights and understanding, advanced the fundamentals, and broadened the applications 
of this neglected yet pervasive class of nonconvex and nondifferentiable functions 
and their optimization.
\end{abstract}

\section{Introduction}

As documented in the monograph \cite{CuiPang2021}, there has been an 
explosion of interests in recent years in the study of minimization 
of nonconvex and nondifferentiable functions, abbreviated as ``non''-functions.  
Among other things, 
this monograph has provided a categorization of many classes of such functions
and highlighted their fundamental properties and
clarified their connections (Chapter~4).  This extensive summary of the fundamentals
of nonconvex and nondifferentiable functions is accompanied by 
a unified computational approach for solving the resulting minimization problems,
even with constraints defined by functions of the same kind.  The basic idea of this 
approach is surrogation (Chapter~7) which upper bounds the given functions by
(typically) convex majorants that define a sequence of convex subprograms solvable
by well established methods.  With focus on generality
and breadth, the monograph has paid little attention to the class of what may
be called {\sl simple ``non''-functions}, which are given by the sums of 
convex (not differentiable) functions and differentiable (not convex) functions.
The advantage of this class of simple ``non''-functions is that full advantage
can be taken of the rich theory and methods of convex programming
and differentiable optimization each having a long history of fruitful developments.  
The summation operation greatly facilitates
the merger of these two individual subareas and allows the toolsets therein to
be profitably combined, resulting in particular
sharp theories and efficient first-order methods for solving this class of 
simple ``non''-optimization problems; a small sample of the literature of this area
most relevant to our present work
includes \cite{Beck23,AttouchBolte09,AttouchBolteSvaiter13,LiPong18}.

\gap

A very recent development on the latter problems is the incorporation of two simple
``non''-functions in a fraction, leading to a fractional program that is the
main topic of two recent studies \cite{BotDaoLi22,BotLiTao24}.  Starting from
the pioneering work of Dinkelbach \cite{Dinkelbach67} in 1967, fractional programs
have their own long history with rich applications.  Traditionally, a fractional program
refers to the minimization of a single ratio.  As seen from the two cited references
\cite{BotDaoLi22,BotLiTao24}, the area of (single-ratio) fractional programming 
continues to flourish to date; yet, not because of the lack of attempts
(to be surveyed later), there are technical challenges to directly and 
rigorously extend the original Dinkelbach method to fractional programs 
with sums of multiple ratios, particularly those involving nondifferentiable 
functions.  With an eye toward applications, our work originated from an interest 
in such an extension.  As it turns out,
this initial interest leads us to realize that a fraction whose numerator
and denominator are both simple ``non''-functions belongs to the classical
family of {\sl quasi-differentiable} functions introduced by Pshenichnyi 
in a 1969 monograph written in Russian which was translated
in 1971 into the English version \cite{Pshenichnyi71}.  (This was pointed out in
\cite{DemyanovPolyakovaRubinov86}.)  In turn, the class
of nondifferentiable functions has much to do with the class of difference-of-convex
(dc) functions \cite{Tuy87,Tuy98} that has significantly broader appeal and sharper
theory \cite{Hartman59,Hiriart-Urruty85,Shapiro86,ShapiroYomdin87} as well as a very 
successful algorithm (known as the Difference-of-Convex Algorithm
abbreviated as the DCA) \cite{PhamLeThi97}.  Partly due to the extensive efforts of
the authors \cite{LeThiPham18,LeThiPham15,LeThiPham14}, the DCA has been applied to
many domains.  Yet because of its generality, the limit points computed by the
algorithm are ``critical solutions'' of a dc program, which are a very relaxed kind
of stationary solutions.   Enhancements of the DCA for problems with special structures
will be reviewed below.

\section{Literature Review and Goals} \label{sec:literature}

Formally, a directionally differentiable function 
$f : {\cal O} \subseteq \mathbb{R}^n \to \mathbb{R}$ defined on the open set
${\cal O}$ is {\sl quasi-differentiable} at a vector 
$\bar{x} \in {\cal O}$ \cite{Pshenichnyi71} if there exist compact convex sets 
$\overline{\partial} f(\bar{x})$ and $\underline{\partial} f(\bar{x})$ such that 
for every direction $v \in \mathbb{R}^n$, 
\[ \begin{array}{lllll}
f^{\, \prime}(\bar{x};v) & \triangleq & \displaystyle{
\lim_{\tau \downarrow 0}
} \, \displaystyle{
\frac{f(\bar{x} + \tau v) - f(\bar{x})}{\tau}
} & \equiv & \mbox{the directional derivative of $f$ at $\bar{x}$ along $v$} 
\\ [0.1in]
& = & \displaystyle{
\max_{a \in \overline{\partial} f(\bar{x})}
} \, a^{\top}v - \displaystyle{
\max_{b \in \underline{\partial} f(\bar{x})}
} \, b^{\top}v & \equiv & \mbox{a difference-of-convex function in $v$ given $\bar{x}$};
\end{array} \]
in short, $f$ is quasi-differentiable at $\bar{x}$ if the directional derivative
$f^{\, \prime}(\bar{x};\bullet)$ is difference-of-convex (dc) \cite{Tuy87,Tuy98}
with the two convex functions being positively homogeneous (i.e., they
are support functions of 
compact convex sets. Seemingly obvious, there is a little subtlety in this 
equivalence that deserves to be highlighted; see 
Proposition~\ref{pr:quasi-dc is quasi-diff}.) 
While there have been several focused studies of quasi-differentiable functions 
and their minimization, including an early article \cite{DemyanovPolyakova79}, 
a survey chapter \cite{Demyanov08}, and two edited volumes 
\cite{DemyanovDixon86,DemyanovRubinov00}, this subject has never been a main
topic of interest in the field of mathematical optimization.  There is nevertheless
a recent extension \cite{Ishikuka88} to ``minmax quasi-differentiable'' functions 
that have close connections to ``generalized bilevel programs''.  A possible reason 
for the lack of attention of quasi-differentiable functions and their optimization
may be attributed to the persistent abstraction of research about these functions 
that deal mostly with the upper and lower sets $\overline{\partial} f(\bar{x})$ 
and $\underline{\partial} f(\bar{x})$, respectively, and pay little attention to 
applied realizations of the theory and methods, such as in the special case of 
fractional programs that themselves have vast applications.  Nevertheless these
early works are foundational and are the building block of
some descent algorithms for the minimization of certain quasi-differentiable functions; 
see e.g.\ \cite{Bagirov00,DemyanovGSDixon86,KDDixon86,Kiewel84}.

\gap

In contrast, the class of dc functions,
which is a subclass of the quasi-differentiable functions (because the
directional derivative of a dc function at a fixed reference vector must 
be dc in the direction), has received extensive attention and occupies 
a prominent place in today's vast literature on the subject of 
nonconvex functions \cite{CuiPang2021}.  Besides its pervasiveness in applications, 
as evidenced by the references documented in the surveys 
\cite{LeThiPham18,LeThiPham15,LeThiPham14}
and also in \cite{NouiehedPangRaza18} that has added many functions
occurring in statistics to the class, the basic difference-of-convex algorithm (DCA)
is very simple conceptually and easy to be applied and has aided the widespread
use of the dc metheodology.  Two noteworthy recent additions to the literature on 
the DCA are \cite{LeThiHuynhPham18,LeThiHuynhPham24}.  Among these references,
the former establishes the sequential 
convergence of the algorithm and its rate for problems with subanalytic data; the tool is 
a nonsmooth form of the Lojasiewicz inequality applied to problems where one of the 
two dc components is differentiable with locally Lipschitz derivatives; the latter
reference extends the DCA to a dc function composite with a differentiable
function; such a composite function is an extension of an amenable function that
we will review later. 
A related sequential convergence analysis of the DCA has appeared 
in \cite[Subsection~8.4.2]{CuiPang2021} without explicitly involving the 
Kurdyka-Lojaziewicz (KL) theory \cite{Kurdyka98,Lojasiewicz59,Lojasiewicz64}
commonly used to address
such convergence issues for subanalytic functions 
\cite{AttouchBolte09,AttouchBolteSvaiter13}.  Among the class of dc function, the
subclass of difference-of-finite-max functions play an important role in the
design of descent algorithms for computing sharp stationary 
solutions \cite{PangRazaAlvarado17}; see also the follow-up papers 
\cite{LuZhouSun19,LuZhou19}.  Specifically, $f : {\cal O} \to \mathbb{R}^n$
is a difference-of-finite-max-convex-differentiable function on the open set ${\cal O}$ 
if
\begin{equation} \label{eq:max-max}
f(x) \, \equiv \displaystyle{
\max_{1 \leq i \leq I}
} \, g_i(x) - \displaystyle{
\max_{1 \leq j \leq J}
} \ h_j(x)
\end{equation}
for some positive integers $I$ and $J$ (which may be taken to be equal without loss
of generality) and some convex and continuously differentiable functions $g_i$ and
$h_j$.  Included in this subclass of dc functions are the piecewise affine (PA) functions
where the functions $g_i$ and $h_j$ are all affine, a fact whose proof can be found 
in \cite{Scholtes02}.   The DCA has been enhanced via of the classical idea of subgradient 
bundling.  This leads to the family of bundle methods of nonsmooth dc programs for
which there is a growing literature 
\cite{deOliveira19,JBKMakela17,KanzowNeder24,vanADJMdeOliveiraSwaminathan21}.

\gap

Being a subclass of the
dc functions, the class of weakly convex functions, which are
originally defined in \cite{Nurminski78} and
subsequently expanded in \cite{Uryasev88}, 
has been particularly
popular in the recent nonsmooth optimization literature; this popularity is partly 
due to the easy and favorable dc decompositions of weakly convex functions 
that facilitate the design of practical algorithms supported by sharp analysis.
These favorable properties are particularly appealing for modern
applications in machine learning, statistics, and diverse engineering domains.
Extending the class of dc functions, the class of implicitly convex-concave
functions was identified in \cite[Subsection~4.4.6]{CuiPang2021} and shown
to be a subclass of the quasi-differentiable functions; see Proposition~4.4.26
in the reference.  Specifically, a function $\theta : {\cal O} \to \mathbb{R}$ is
{\sl implicitly convex-concave} on the open convex set ${\cal O}$ if there exists 
a bivariate function $\psi : {\cal O} \times {\cal O} \to \mathbb{R}$ such that  
$\psi(\bullet;z)$ is convex on ${\cal O}$ for
every $z \in {\cal O}$ and $\psi(x;\bullet)$ is concave on ${\cal O}$ 
for every $x \in {\cal O}$ and $f(x) = \psi(x;x)$.  Implicitly convex-concave
functions arise as value functions of parametric optimization problems
with parameters in both the objective functions and constraints satisfying
constraint qualifications; in turn
such value functions are pervasive in optimistic and pessimistic versions
of bilevel programs and in two-stage stochastic programs; see \cite{CuiPang2021}.  

\gap

In summary, the master class of quasi-differentiable functions has been overshadowed
by its many subclasses all having connections to difference of convexity.  For this
reason, we will use the term {\sl quasi-difference-convexity} as a modern substitute
for quasi-differentiability in order to tie these functions closer
to applications and computations and draw the attention of potential users
of dc functions to the broader family of quasi-dc functions.

\subsection{Fractional programs}

In \cite{Dinkelbach67},  Dinkelbach introduced the single-ratio fractional program:
\begin{equation} \label{eq:Dinkelbach single ratio}
\displaystyle{
\operatornamewithlimits{\mbox{\bf minimize}}_{x \in X \subseteq \mathbb{R}^n}
} \ \displaystyle{
\frac{n(x)}{d(x)}
},
\end{equation}
where the denominator is a positive function.  Dinkelbach's algorithm for
the above single-ratio problem iteratively solves
subproblems each of the form:
\begin{equation} \label{eq:Dinkelbach original subproblem}
\displaystyle{
\operatornamewithlimits{\mbox{\bf minimize}}_{x \in X}
} \ n(x) - \bar{\psi} \, d(x), \epc
\mbox{where } \ \bar{\psi} \, \triangleq \, 
\displaystyle{
\frac{n(\bar{x})}{d(\bar{x})}
}
\end{equation}
with $\bar{x}$ being a given iterate.  Dinkelbach's original paper
has neither specified how each subproblem should be solved in practice nor 
required a globally optimal solution to be computed; yet, for practical
purposes, this ``scalarization'' approach essentially requires the 
numerator $n(x)$ to be convex and nonnegative and the denominator 
$d(x)$ to be concave so that a solution
for each subproblem (\ref{eq:Dinkelbach original subproblem}) can be readily
computed to satisfy the conditions for convergence.  (For one thing, the
sign restriction of the numerator and denominator functions ensures
the convexity of the scalarized objective function 
of (\ref{eq:Dinkelbach original subproblem}).)
Obviously, 
such convexity/concavity requirements, along with the single-ratio setting, 
are rather restrictive.  Nevertheless, Dinkelbach's pioneering work
has led to a large body of extended studies of fractional 
programs in broader settings.  Today, there is an extensive 
literature on this topic, which for the most part pertains to
several extended classes.

\gap 

\noindent (a) Generalized fractional programs were introduced and studied in
\cite{CrouzeixFerlandSchaible85,CrouzeixFerland91}:
\begin{equation} \label{eq:max of single ratios}
\displaystyle{
\operatornamewithlimits{\mbox{\bf minimize}}_{x \in X}
} \ \displaystyle{
\max_{1 \leq j \leq J}
} \ \left\{ \, \displaystyle{
\frac{n_j(x)}{d_j(x)}
} \, \right\}.
\end{equation}
The special class of fractional quadratic functions (i.e., all the
functions $n_j$ and $d_j$ are quadratic functions) have found applications 
in signal and information processing \cite{GSMO19} and energy efficiency 
in wireless networks \cite{ZapponeJorswieck14}.

\gap

\noindent (b) Sums-of-ratios problems 
\begin{equation} \label{eq:sum of single ratios}
\displaystyle{
\operatornamewithlimits{\mbox{\bf minimize}}_{x \in X}
} \ \displaystyle{
\sum_{k=1}^K 
} \ \displaystyle{
\frac{n_k(x)}{d_k(x)}
} 
\end{equation}
have diverse applications in signal 
processing \cite{ShenYu18-I,ShenYu18-II,ZapponeJorswieck14} and
in multiclass clustering \cite{LiWangMerchant23,ShiMalik00,YuShi03}; 
there have been many suggestions to 
extend the original Dinkelbach algorithm; a sample of references
includes \cite{AlmogyLevin71,Benson02-a,
Benson02-b,FreundJarre01,GruzdevaStrekalovsky18,JargalDarkijav22,Korramabadi21,ShenWangWu23}.
Some of these works propose global optimization methods under special assumptions
on the numerator and denominator functions.  

\gap

(c) As already mentioned \cite{BotDaoLi22,BotLiTao24,LeThiHuynhPham18}, 
single ratios with simple non-functions
whose split convexity-plus-differentiability structure can be profitably exploited
have been the subject of some additional papers \cite{LiShenZhangZhou22,ZhouZhangLi24}.

\gap

In the recent monograph \cite[Section~7.2.4]{CuiPang2021}, 
the authors have considered the generalized fractional 
program (\ref{eq:max of single ratios}),
where the functions $n_j$ and $d_j$ are nonconvex and 
nondifferentiable but can be ``surrogated'' 
at an arbitrary $\bar{x} \in X$ by
$\wh{n}_j(\bullet;\bar{x})$ and $\wh{d}_j(\bullet;\bar{x})$
respectively.  These surrogated functions are such that  
the resulting {\sl surrogated Dinkelbach subprogram}
\begin{equation} \label{eq:Dinkelbach subproblem}
\displaystyle{
\operatornamewithlimits{\mbox{\bf minimize}}_{x \in X}
} \ \displaystyle{
\max_{1 \leq j \leq J}
} \ \wh{\psi}_j(x;\bar{x}),
\end{equation}
where
\[
\wh{\psi}_j(x;\bar{x}) \, \triangleq \, \wh{n}_j(x;\bar{x}) - 
\left[ \, \displaystyle{
\max_{1 \leq j^{\, \prime} \leq J}
} \, \displaystyle{
\frac{n_{j^{\prime}}(\bar{x})}{d_{j^{\prime}}(\bar{x})}
} \, \right] \, \wh{d}_j(x;\bar{x}),
\epc j \, = \, 1, \cdots, J
\]
is globally solvable.  An instance where such
surrogated functions can easily be constructed is the
classical Dinkelbach setting where 
each $n_j$ is convex and nonnegative and $d_j$ is concave and positive; 
in this case, we may take
$\wh{n}_j(\bullet;\bar{x}) \equiv n_j$ and
$\wh{d}_j(\bullet;\bar{x}) \equiv d_j$.  
Under a key directional derivative consistency condition
on the surrogated functions, Algorithm~7.2.8 in
\cite{CuiPang2021}, which is based on 
a combination of the above surrogation and 
the classical Dinkelbach algorithm 
for (\ref{eq:Dinkelbach single ratio}), is shown to compute a 
directional stationary solution of 
(\ref{eq:max of single ratios}); see Proposition~7.2.6
in the cited reference.  The idea of surrogation will play a major role in this
paper.

\subsection{Goals and organization of this paper}

In a nutshell, inspired by the fusion of several old and contemporary research topics: 
the historical class 
of quasi-differentiable functions,
the classical iterative descent methods and their modernization,
the renewed interest of fractional programs, the popularization of dc programming,
the explosive literature on first-order methods for the minimization of 
simple non-functions and their sequential convergence analysis via the KL theory, 
and more broadly,
the expansive recent research of modern nonconvex and nondifferentiable 
optimization, this paper has two major goals:

\gap

$\bullet $ to revisit, modernize, and advance the old subject of quasi-differentiable 
optimization via the class of
quasi-difference-convex functions and the manifestation of their breadth in 
contemporary applications of nonsmooth and nonconvex optimization, focusing on 
composite functions involving the class of simple non-functions;

\gap

$\bullet $ to put together a unified treatment of iterative convex-programming
based descent algorithms for solving a broad class of composite quasi-dc programs,
and to establish their
subsequential convergence, sequential convergence, and rates of convergence; the
latter two topics are in line with the modern analysis of these algorithms 
and are departures from the sole emphasis of subsequential convergence in 
their traditional studies.

\gap

We begin in the next section with the formal definitions of the quasi-functions, 
reviewing and extending their properties, and ending the discussion with
a general result showing the preservation of the quasi-difference-convexity
property and with an equivalent restatement of the ``inf-stationarity'' introduced
in the quasi-differentiable calculus literature 
\cite{DemyanovPolyakova79,DemyanovRubinov80} in terms of the more transparent
directional stationarity condition.  In Section~\ref{sec:quasi-dc for minimization},
we set the stage for our journey in algorithmic development for a major class of 
quasi-dc minimization problems.  In this vein, our work brings the early references 
\cite{Bagirov00,DemyanovGSDixon86,KDDixon86}, which have established only the subsequential
convergence of the descent algorithms described therein, to align with the modern interests 
of sequential convergence and rate analysis using the popular tool of the KL theory.  
Recent references on the latter topics can be found 
in \cite{LeThiHuynhPham18,LeThiHuynhPham24} for dc programming 
(see also \cite[Subsection~8.4.2]{CuiPang2021}) and 
\cite{BotDaoLi22,BotLiTao24} for fractional programs of simple non-functions.  Details
of the algorithmic design and description as well as their convergence analysis 
are covered in the last
two Sections~\ref{sec:iterative algorithms} and \ref{sec:sequential convergence}.

\gap

{\bf A word about the contributions of the present work.}  Readers
familiar with the recent works of the author (Pang) will find many of the ideas employed
here already appeared in the monograph \cite{CuiPang2021} and the subsequent references
\cite{CuiLiuPang2022,CuiLiuPang2023}.  Nevertheless, as mentioned already, an initial
impetus of the present work stemmed from a very different topic, namely, that of 
fractional programming.  This motivation is reinforced by the recent works 
\cite{BotDaoLi22,BotLiTao24,LeThiHuynhPham24} focused on special classes of composite
non-problems and the seemingly lack of attention to the unifying class of less
well-known quasi-differentiable
functions.  So the idea of the present paper was born with the goals stated above.  The end
product is a mix of old and new ideas and results applied to a class of composite nonconvex 
and nondifferentiable optimization problems and its specializations to be treated by 
classical methods enhanced by 
renewed perspectives and modern tools for the benefits of broad applications. 

\section{Quasi-difference-convex Functions: Old and New Results} 
\label{sec:quasi-dc functions}

We say that a Bouligand differentiable (i.e., locally Lipschitz continuous and
directionally differentiable) function $f : {\cal O} \to \mathbb{R}$ defined
on the open convex set ${\cal O} \subseteq \mathbb{R}^n$ is 
{\sl quasi-difference-convex} (quasi-dc) at $\bar{x}$ if
$f^{\, \prime}(\bar{x};\bullet) : \mathbb{R}^n \to \mathbb{R}$ is the difference of
two convex functions (on $\mathbb{R}^n$).  These properties
are said to hold on a domain if they hold at every point in the domain.
As a historical remark, we note that the term {\sl pseudosmoothness}
has been used for this concept in an unpublished 
manuscript \cite[Definition~3.1]{ShapiroYomdin87}; this reference has also 
used the terminology of difference-convex-homogeneous for a function that is 
the difference of two convex positively homogeneous function.  Since throughout this paper, 
we will work only with 
locally Lipschitz continuous functions, we require all quasi-dc functions to be locally 
Lipchitz.  It is well known that positively homogeneous convex functions are closely 
tied to sublinear
functions (i.e., functions that are positively homogeneous and subadditive); for
this reason, quasi-difference-convex functions have been 
call {\sl difference-sublinear} functions, a terminology we do not use in this paper.

\gap

It was mentioned in \cite{DemyanovPolyakovaRubinov86}, without proof, that an important role
of quasi-difference-convexity is that the directional 
derivative $f^{\, \prime}(\bar{x};\bullet)$ of every Bouligand differentiable function 
$f$ can be
approximated to within any prescribed accuracy by the difference of two positively homogeneous
convex functions. In this vein, a recent result in \cite{Royset2020} shows that every upper
semicontinuous extended-valued function is the limit
of a ``hypo-converging'' sequence of piecewise affine functions of the 
difference-of-max type.  There are many results on 
positively homogeneous functions that are applicable to the directional derivatives 
of a quasi-dc function; see \cite{GorokhovikTrafimovich16,GorokhovikTrafimovich18}.
The following simple result formally proves that quasi-differentiability is equivalent
to quasi-difference-convexity.  While seemingly obvious, a subtlety in this equivalence
is that in the definition of quasi-difference-convexity, positive homogeneity of the
two convex functions in the dc decomposition of $f^{\, \prime}(\bar{x};\bullet)$ is not
explicitly stated.  The result below essentially shows that this explicit requirement
is implicit in the definition of quasi-difference-convexity of the directional derivative.

\begin{proposition} \label{pr:quasi-dc is quasi-diff} \rm 
Let $f : {\cal O} \to \mathbb{R}$ be Bouligand differentiable on the open convex set
${\cal O} \subseteq \mathbb{R}^n$.  Then $f$ is quasi-dc at $\bar{x} \in {\cal O}$
if and only if it is quasi-differentiable there.
\end{proposition}



\begin{proof} `It suffices to prove the ``only if'' statement.  Suppose that 
$\wh{f} \triangleq f^{\, \prime}(\bar{x};\bullet) = g - h$ has a dc decomposition.
For any scalar $\alpha > 0$, we have
\[
\alpha \, \wh{f}(v) \, = \, \wh{f}(\alpha v) \, = \, g(\alpha v) - h(\alpha v),
\]
which upon dividing by $\alpha$ and using the fact that $\wh{f}(0) = 0 = g(0) - h(0)$, 
yields
\[
\wh{f}(v) \, = \, \displaystyle{
\frac{g(\alpha v) - g(0)}{\alpha}
} - \displaystyle{
\frac{h(\alpha v) - h(0)}{\alpha}
} 
\]
Letting $\alpha \downarrow 0$ gives
\[
\wh{f}(v) \, = \, g^{\, \prime}(0;v) - h^{\, \prime}(0;v).
\]
Since a positively
homogeneous, closed proper convex function is the support function of a convex compact 
set (see \cite[Corollary~13.1.2]{Rockafellar70}), the desired quasi-differentiability
of $f$ at $\bar{x}$ follows readily.
\end{proof}

{\bf Remark.}. In the authors' original proof, the limit was taken for 
$\alpha \uparrow \infty$, leading to the alternative representation
$\wh{f}(v) = g_{\infty}(v) - h_{\infty}(v)$,
where $g_{\infty}$ and $h_{\infty}$ are the recession functions \cite{Rockafellar70} 
(as in convex analysis) of the convex functions $g$ and $h$, respectively.  
These recession functions are positively homogeneous, in addition to being convex
but possibly extended-valued.  It was suggested by Mr.\ Jingu Li, a Ph.D.\ student
at Tsinghua University in Beijing, to take the limit $\alpha \downarrow 0$; this
simple change of limits avoids the extended-valued representation.  \hfill $\Box$

\gap

An important subclass of the quasi-dc functions consists of those whose directional
derivative is a (positively homogeneous) convex function in the direction when the 
reference point is fixed.
Such a function has been called dd-convex (for directional derivative convex) 
in \cite[Definition~4.3.3]{CuiPang2021}.  As noted in the reference, the class of
dd-convex functions is very broad, including in particular an amenable function
\cite{PoliquinRockafellar93}, i.e., a convex composite with a differentiable function:
$f \circ g$, which $f$ is convex and $g$ is differentiable; so are a favorable dc
function (which is a dc function whose concave part is 
differentiable \cite[Definititon~6.1.5]{CuiPang2021}) and its
composition with a differentiable function.  

\gap 
 
Before studying properties of quasi-dc functions, we note that the univariate function 
\[
f(t) \, = \, \left\{ \begin{array}{ll}
t^2 \, \sin 1/t & \mbox{for $t \neq 0$} \\ [0.1in] 
0 & \mbox{for $t = 0$}
\end{array} \right.
\]
is quasi-dc but not dc.  The reason is that the function, which is everywhere
locally Lipschitz and differentiable, has its derivative given by
\[
f^{\, \prime}(t) \, = \, \left\{ \begin{array}{ll}
2t \, \sin 1/t - \cos 1/t & \mbox{for $t \neq 0$} \\ [0.1in] 
0 & \mbox{for $t = 0$}
\end{array} \right.
\]
which is not continuous at $t = 0$.
Therefore $f$ is quasi-dc (by Proposition~\ref{pr:univariate composition} below); 
moreover, by a theorem in \cite{Hiriart-Urruty85} (see also
\cite[Proposition~4.4.19]{CuiPang2021}), $f$ is not dc because a differentiable function
is dc if and only if it is continuously differentiable.  

\gap

Below, we collect some elementary (and known) classes of quasi-dc functions before 
we derive some results pertaining to the preservation of the quasi-dc property under
functional compositions, ending up with a final main result of the kind.  Throughout the
proofs of the results in the rest of the section, since the required positive homogeneity 
of the component convex functions in the quasi-difference-convexity can easily seen 
to hold, the arguments focus only on the dc decomposition.

\begin{proposition} \label{pr:classes of quasi-dc} \rm
The following statements are valid for a function $f : {\cal O} \to \mathbb{R}$ defined on
an open convex set ${\cal O} \subseteq \mathbb{R}^n$.

\gap

$\bullet $ If $f$ is difference-of-convex, then it is quasi-dc.  In particular, a convex 
(or concave) function is quasi-dc.  More generally, an implicitly convex-concave function
is quasi-dc.

\gap

$\bullet $ If $f$ is differentiable, then it is quasi-dc.

\gap

$\bullet $ If $f$ is given by (\ref{eq:max-max}) where the functions $g_i$ and $h_j$
are differentiable (not necessarily convex), then $f$ is quasi-dc;
thus a PA function is quasi-dc. \hfill $\Box$
\end{proposition}

Less straightforward is the quasi-dc property of a locally Lipschitz biconvex
function (as defined in the proposition below) and its induced univariate 
implicitly convex-convex function.  This
result requires a sum rule of the directional derivative of a bivariate function
first described in \cite{Robinson85}.
Specifically, let $g : {\cal O} \times {\cal O} \to \mathbb{R}$ be a Bouligand
differentiable jointly in its two arguments.  Then for any pair 
$(\bar{x},\bar{y}) \in {\cal O} \times {\cal O}$, the two partial directional
derivatives 
\[
g_x^{\, \prime}((\bar{x},\bar{y});u) \, \triangleq \,
g(\bullet,\bar{y})^{\, \prime}(\bar{x};u) \ \mbox{ and } \ 
g_y^{\, \prime}((\bar{x},\bar{y});v) \, \triangleq \,
g(\bar{x},\bullet)^{\, \prime}(\bar{y};v)
\]
exist, but the sum formula:
\begin{equation} \label{eq:sum formula for dd}
g^{\, \prime}((\bar{x},\bar{y});(u,v)) \, = \, 
g_x^{\, \prime}((\bar{x},\bar{y});u) + g_y^{\, \prime}((\bar{x},\bar{y});v) 
\end{equation}
may not be valid.  There are sufficient conditions (besides continuous differentiability)
and special classes
of bivariate functions for the formula to hold; for general discussion of this formula, 
see \cite[Exercise~4.1.3]{CuiPang2021}.  Among the functions for which the formula
holds is a convex-concave function \cite[Proposition~4.4.26]{CuiPang2021}.  For the
class of biconvex functions, we have the following result
\cite[Proposition~4.4.27]{CuiPang2021}.  See Corollary~\ref{co:composite quasi-dc}
for a generalization.  

\begin{proposition} \label{pr:biconvex} \rm
Let $g : {\cal O} \times {\cal O} \to \mathbb{R}$ be locally Lipschitz continuous
such that $g(\bullet,y)$ and $g(x,\bullet)$ are convex on ${\cal O}$ for
arbitrary fixed pairs $(x,y) \in {\cal O} \times {\cal O}$.  Under either
one of the following conditions: $g(\bullet,y)$ is differentiable on ${\cal O}$
for all $y \in {\cal O}$ or $g(x,\bullet)$ is differentiable on ${\cal O}$
for all $x \in {\cal O}$, then the sum formula (\ref{eq:sum formula for dd}) 
holds for all $(\bar{x},\bar{y}) \in {\cal O} \times {\cal O}$.  Hence, 
$g$ is dd-convex on ${\cal O} \times {\cal O}$;
therefore, the implicitly convex-convex 
function $f(x) \triangleq g(x,x)$ is dd-convex on ${\cal O}$.
\hfill $\Box$
\end{proposition}

The following preservation result is for a univariate composition. 

\begin{proposition} \label{pr:univariate composition} \rm
Let $f : {\cal O} \to \mathbb{R}$ be a quasi-difference-convex function
on the open convex set ${\cal O} \subseteq \mathbb{R}^n$ and
$g : \Omega \to \mathbb{R}$ be a univariate Bouligand differentiable
function on the open convex set $\Omega$ containing $f({\cal O})$.  
Then the composite function
$g \circ f : {\cal O} \to \mathbb{R}$ is quasi-dc on ${\cal O}$.  In particular,
every univariate Bouligand differentiable function is quasi-dc.
\end{proposition}

\begin{proof}  For any $v \in \mathbb{R}^n$, we have
\[
( g \circ f )^{\, \prime}(\bar{x};v) \, = \, 
g^{\, \prime}(f(\bar{x});f^{\, \prime}(\bar{x};v)) \, = \, 
f^{\, \prime}(\bar{x};v)_+ \, g^{\, \prime}(f(\bar{x});1) + 
f^{\, \prime}(\bar{x};v)_- \, g^{\, \prime}(f(\bar{x});-1),
\] 
where $f^{\, \prime}(\bar{x};v)_{\pm} = \max( \, \pm f^{\, \prime}(\bar{x};v), \, 0 \, )$.
Since $f^{\, \prime}(\bar{x};\bullet)$ is a dc function, so is 
$f^{\, \prime}(\bar{x};\bullet)_{\pm}$.  Since $g^{\, \prime}(f(\bar{x});\pm 1)$
are constants given $\bar{x}$, it follows readily that 
$( g \circ f )^{\, \prime}(\bar{x};\bullet)$ is dc.  The last statement follows 
readily with $n = 1$, $\Omega = {\cal O}$, and $f$ being the identity function.
\end{proof}

The next result is the multivariate generalization of the above univariate result;
it requires a coordinate-wise sum property of the directional derivatives.  

\begin{proposition} \label{pr:multivariate composition with dd condition} \rm
Let $f_j : {\cal O} \to \mathbb{R}$ be a quasi-difference-convex function
on the open convex set ${\cal O} \subseteq \mathbb{R}^n$ for $j = 1, \cdots, m$
and $g : \Omega \to \mathbb{R}$ be a multivariate function, with
$\Omega$ being an open convex set in $\mathbb{R}^m$ containing $F({\cal O})$,
where $F \triangleq ( f_j )_{j=1}^m$.Provided that 
\begin{equation} \label{eq:dd-sum equality}
g^{\, \prime}(y;v) \, = \, \displaystyle{
\sum_{j=1}^m
} \, g(\bullet,y^{-j})^{\, \prime}(y_j;v_j), \epc \forall \, 
(y,v) \, \in \, \mathbb{R}^{2m}.
\end{equation}
where $y^{-j}$ is the subvector of $y$ without the $j$th component, then the
composite function $g \circ F$, where $F \triangleq ( f_j )_{j=1}^m$ is
quasi-dc.
\end{proposition}

\begin{proof} This is fairly easy by the string of equalities:
\[ \begin{array}{l}
( g \circ F )^{\, \prime}(\bar{x};v) \\ [0.1in] 
= \, g^{\, \prime}(F(\bar{x});F^{\, \prime}(\bar{x};v)) \, = \, \displaystyle{
\sum_{j=1}^m
} \, g(\bullet,F^{-j}(\bar{x}))^{\, \prime}(f_j(\bar{x});f_j^{\, \prime}(\bar{x};v)),
\epc \mbox{where $F^{-j}(\bar{x}) \triangleq ( f_k(\bar{x}) )_{k \neq j}$} \\ [0.1in]
= \, \displaystyle{
\sum_{j=1}^m
} \, \left[ \, ( f_j^{\, \prime}(\bar{x};v) )_+ 
g(\bullet,F^{-j}(\bar{x}))^{\, \prime}(f_j(\bar{x});1) + 
( f_j^{\, \prime}(\bar{x};v) )_- 
g(\bullet,F^{-j}(\bar{x}))^{\, \prime}(f_j(\bar{x});-1) \, \right]
\end{array} \]
Since each $f_j^{\, \prime}(\bar{x};\bullet)$ is a dc function, the above identity
immediately yields that 	$( g \circ F )^{\, \prime}(\bar{x};\bullet)$ is a dc function,
as desired.
\end{proof}

It turns out that the directional derivative
equality (\ref{eq:dd-sum equality}) bears a very close connection to 
quasi-difference convexity.

\begin{proposition} \label{pr:dd sum} \rm
Let $f : {\cal O} \to \mathbb{R}$ be a Bouligand differentiable function
on the open convex set ${\cal O} \subseteq \mathbb{R}^n$.  For any vector
$\bar{x} \in {\cal O}$, the following statements are equivalent:

\gap

{\bf (A)} $f$ satisfies the equality (\ref{eq:dd-sum equality}) at $\bar{x}$; i.e.,
\[
f^{\, \prime}(\bar{x};v) \, = \, \displaystyle{
\sum_{j=1}^m
} \, f(\bullet,\bar{x}^{-j})^{\, \prime}(\bar{x}_j;v_j), \epc 
\forall \, v \, \in \, \mathbb{R}^n
\]
{\bf (B)} $f$ is quasi-dc at $\bar{x}$ and the upper and lower sets 
$\overline{\partial} f(\bar{x})$ and $\underline{\partial} f(\bar{x})$ are rectangles;
i.e., each is the Cartesian product of $n$ one-dimensional compact intervals.

\gap

{\bf (C)} $f^{\, \prime}(\bar{x};\bullet)$ is a separable function on $\mathbb{R}^n$.	
\end{proposition}

\begin{proof} (A) $\Rightarrow$ (B).  We fist show that the univariate function
$g_j \triangleq f(\bullet,\bar{x}^{-j})^{\, \prime}(\bar{x}_j;\bullet)$ is 
difference-convex.  In fact, letting 
$a_j \triangleq f(\bullet,\bar{x}^{-j})^{\, \prime}(\bar{x}_j;1)$ and
$b_j \triangleq f(\bullet,\bar{x}^{-j})^{\, \prime}(\bar{x}_j;-1)$, we have
\[ \begin{array}{lll}
g_j(t) & \triangleq & t_+ \,  f(\bullet,\bar{x}^{-j})^{\, \prime}(\bar{x}_j;1) +
t_- \,  f(\bullet,\bar{x}^{-j})^{\, \prime}(\bar{x}_j;-1) \\ [0.1in]
& = & \left\{ \begin{array}{ll}
\displaystyle{
\max_{\alpha \in [ \, -b_j, \, a_j \, ]}
} \, \alpha \, t & \mbox{if $a_j \geq 0$ and $b_j \geq 0$} \\ [0.15in]
\displaystyle{
\min_{\alpha \in [ \, a_j, \, -b_j \, ]}
} \, \alpha \, t & \mbox{if $a_j \leq 0$ and $b_j \leq 0$} \\ [0.15in]
\displaystyle{
\max_{\alpha \in [ \, 0, \, a_j ]}
} \, \alpha \, t + \displaystyle{
\min_{\beta \in [ \, 0, \, -b_j \,]}
} \, \beta \, t & \mbox{if $a_j > 0$ and $b_j < 0$} \\ [0.1in]
\displaystyle{
\min_{\alpha \in [ \, a_j, \, 0 \, ]}
} \, \alpha \, t + \displaystyle{
\max_{\beta \in [ \, -b_j, \, 0 \,]}
} \, \beta \, t & \mbox{if $a_j < 0$ and $b_j > 0$}
\end{array} \right. \\ [0.6in]
& = &  \displaystyle{
\max_{\gamma \in [ \, \underline{a}_j,\overline{a}_j \, ]}
} \, \gamma \, t - \displaystyle{
\max_{\gamma \in [ \, \underline{b}_j,\overline{b}_j \, ]}
} \, \gamma \, t \hspace{1in} \forall \, t \, \in \, \mathbb{R} 
\end{array} \]
for some scalars $\underline{a}_j \leq \overline{a}_j$ and $\underline{b}_j \leq
\overline{b}_j$.  Since 
$f^{\, \prime}(\bar{x};v) = \displaystyle{
\sum_{j=1}^n
} \, g_j(v_j)$ by assumption, 
it follows that $f^{\, \prime}(\bar{x};v) = \displaystyle{
\max_{a \in A(\bar{x})}
} \, a^{\top}v - \displaystyle{
\max_{b \in B(\bar{x})}
} \, b^{\top}v$, where $A(\bar{x}) \triangleq \displaystyle{
\prod_{j=1}^n
} \, [ \, \underline{a}_j,\overline{a}_j \, ]$ and $B(\bar{x}) \triangleq \displaystyle{
\prod_{j=1}^n
} \, [ \, \underline{b}_j,\overline{b}_j \, ]$.  Hence $f$ is quasi-dc. 

\gap

(B) $\Rightarrow$ (C).  This is obvious.

\gap

(C) $\Rightarrow $ (A).  If $f^{\, \prime}(\bar{x};v) = \displaystyle{
\sum_{j=1}^n
} \, g_j(v_j)$ for some univariate functions $g_j$, then it follows that
\[
f(\bullet,\bar{x}^{-j})^{\, \prime}(\bar{x}_j;v_j) \, = f^{\, \prime}(\bar{x};v_j e^j) 
\, = \, g_j(v_j) + \displaystyle{
\sum_{k \neq j}
} \, g_k(0),
\]  
where $e^j$ is the $j$th unit coordinate vector.  Since $f^{\, \prime}(\bar{x};0) = 0$,
it follows that $\displaystyle{
\sum_{j=1}^n
} \, g_j(0) = 0$.  Hence
\[
g_j(v_j) \, = \, f(\bullet,\bar{x}^{-j})^{\, \prime}(\bar{x}_j;v_j) + g_j(0),
\]
which yields $f^{\, \prime}(\bar{x};v) = \displaystyle{
\sum_{j=1}^n
} \, f(\bullet,\bar{x}^{-j})^{\, \prime}(\bar{x}_j;v_j)$, which is
the sum formula holds for $f^{\, \prime}(\bar{x};v)$.
\end{proof}

\begin{remark} \rm
It is well known that a Bouligand differentiable function with a linear directional
derivative at a reference vector, i.e., if $f^{\, \prime}(\bar{x};\bullet)$ is
linear, then $f$ is differentiable at $\bar{x}$.  Proposition~\ref{pr:dd sum} gives
two necessary and sufficient conditions for 
$f^{\, \prime}(\bar{x};\bullet)$ to be a separable function; interestingly,
one of these conditions turn out to be quasi-differentiability with rectangular 
upper and lower sets.  \hfill $\Box$
\end{remark}
 
The next result gives several broad families of the outer function $g$ that will yield the
quasi-difference-convexity of the composite function $g \circ F$; some concrete
examples of $g$ provide practical instances of composite quasi-dc functions that include,
for instance, sums of ratios and sum of products of quasi-dc functions.  

\gap

\begin{proposition} \label{pr:multivariate composition with dd condition} \rm
Let $f_k : {\cal O} \to \mathbb{R}$ be a quasi-difference-convex function
on the open convex set ${\cal O} \subseteq \mathbb{R}$ for $k = 1, \cdots, m$
and $g : \Omega \to \mathbb{R}$ be a multivariate function, with
$\Omega$ being an open convex set in $\mathbb{R}^m$ containing $F({\cal O})$,
where $F \triangleq ( f_k )_{k=1}^m$.  Then the composite function $g \circ F$  
is quasi-dc under the following conditions on $g$:

\gap

$\bullet $ $g$ is differentiable;

\gap

$\bullet $ $g$ is PA;

\gap

$\bullet $ $g$ is the $p$-norm for some $p \in [ \, 1, \infty ]$.

\gap

Hence the following statements are true:

\gap

$\bullet $ sums, products (thus squares), quotients (with positive denominators) of
quasi-dc functions are quasi-dc; 

\gap

$\bullet $ a simple non-function, which (by definition)
is the sum of a convex and a differentiable function, is quasi-dc;

\gap

$\bullet $ the $p$-power of a positive quasi-dc function is quasi-dc for $p \geq 1$.
\end{proposition}

\begin{proof}  If $g$ is differentiable, we have
\[
( g \circ F )^{\, \prime}(\bar{x};v) \, = \, \displaystyle{
\sum_{k=1}^m
} \, \displaystyle{
\frac{\partial g(F(\bar{x}))}{\partial y_k}
} \, f_k^{\, \prime}(\bar{x};v)
\]
Since linear combinations of dc functions are dc, it follows readily that 
$( g \circ F )^{\, \prime}(\bar{x};\bullet)$ is dc.

\gap

If $g$ is PA, we can write $g(y) = g(y) - h(y)$, where 
\[
g(y) \, \triangleq \, \displaystyle{
\max_{1 \leq i \leq I}
} \, ( a^i )^{\top}y + \alpha_i \epc \mbox{and} \epc 
h(y) \, \triangleq \, \displaystyle{
\max_{1 \leq j \leq J}
} \, ( b^{\, j} )^{\top}y + \beta_j
\]
for some vectors $a^i$ and $b^{\, j}$ in $\mathbb{R}^m$
and scalars $\alpha_i$ and $\beta_j$.  For an arbitrary vector $y \in \mathbb{R}^m$, 
we let
\[
{\cal M}_g(y) \, \triangleq \, \left\{ \, i \, \mid \, ( a^i )^{\top}y + \alpha_i \, = \, 
g(y) \, \right\}
\epc \mbox{and} \epc 
{\cal M}_h(y) \, \triangleq \, \left\{ \, j \, \mid \, ( b^{\, j} )^{\top}y + \beta_j 
\, = \, h(y) \, \right\}.
\]
For a given vector $\bar{x} \in \mathbb{R}^n$,
with $\bar{y} \triangleq F(\bar{x}) = \left( f_k(\bar{x} \right)_{k=1}^m$, we have
\[
( g \circ F )^{\, \prime}(\bar{x};v) \, = \, 
\displaystyle{
\max_{i \in {\cal M}_g(\bar{y})}
} \, \displaystyle{
\sum_{k=1}^m
} \, a^i_k \, f_k^{\, \prime}(\bar{x};v) - \displaystyle{
\max_{j \in {\cal M}_h(\bar{y})}
} \, \displaystyle{
\sum_{k=1}^m
} \, b^{\, j}_k \, f_k^{\, \prime}(\bar{x};v),
\]
from which the dc property of $( g \circ F )^{\, \prime}(\bar{x};\bullet)$ follows readily.

\gap

Next, suppose $g$ is the $p$-norm for some $p \in [ \, 1,\infty ]$.  If $p = 1$,
then $g(y) = \displaystyle{
\sum_{k=1}^m
} \, | \, y_k \, |$ is a piecewise linear function; so is $g$ when $p = \infty$.
Thus $g \circ F$ is quasi-dc in both cases.  For $p \in ( \, 1, \infty )$,
$g$ is continuously differentiable except at the origin where we have
\[
g^{\, \prime}(0;v) \, = \, \| \, v \, \|_p.
\] 
Hence, it suffices to show that $\| F^{\, \prime}(\bar{x};v) \|_p$ is a dc function 
in $v$ for $p \in ( \, 1,\infty )$.
We follow the proof for $p = 2$ given in \cite[Proposition~4.4.16(e)]{CuiPang2021}.
Specifically, since $\| \, y \, \|_p = \displaystyle{
\max_{a \, : \, \| a \|_q = 1}
} \, a^{\top} y$, where $q$ satisfies $\displaystyle{
\frac{1}{p}
} +  \displaystyle{
\frac{1}{q}
} = 1$, writing $f_k^{\, \prime}(\bar{x};v) = \phi_k(\bar{x};v) - \varphi_k(\bar{x};v)$, 
where $\phi_k(\bar{x};\bullet)$ and $\varphi_k(\bar{x};\bullet)$ are both convex functions,
we have
\[ \begin{array}{l}
\| \, F^{\, \prime}(\bar{x};v) \, \|_p  \, = \, \displaystyle{
\max_{a \, : \, \| a \|_q = 1}
} \, \displaystyle{
\sum_{k=1}^m
} \, a_k \, f_k^{\, \prime}(\bar{x};v) \\ [0.2in]
= \, -\displaystyle{
\sum_{k=1}^m
} \, \left[ \, \underbrace{\phi_k(\bar{x};v) + \varphi_k(\bar{x};v)}_{\mbox{convex in $v$}} \, \right] 
+ \displaystyle{
\max_{a \, : \, \| a \|_q = 1}
} \, \displaystyle{
\sum_{k=1}^m
} \, \left[ \, \underbrace{( a_k + 1 ) \, \phi_k(\bar{x};v) + 
( 1 - a_k ) \, \varphi_k(\bar{x};v)}_{\mbox{convex in $v$ because $| a_k | < 1$}} \, \right].
\end{array}
 \]
 Since the pointwise maximum of a family of convex function is convex, the above identity
 establishes that $g \circ F$ is quasi-dc when $g$ is the $p$-norm for some
 $p \in [ \, 1,\infty ]$ and each $f_k$ is quasi-dc.  Finally, the last three
 statements of
 the proposition are immediate consequence of the the previous cases. 
\end{proof}

The above proof of the PA and $p$-norm cases can be combined to yield the following 
preservation result of the quasi-dc property under quasi-dc composition on open convex 
sets.  The result includes the special case 
of a dc function composed with a differentiable function that is studied in the 
recent paper \cite{LeThiHuynhPham24} and generalizes the known fact that a dc function
composed with a dc function on open convex sets is dc \cite[Theorem~4.4.3]{CuiPang2021}.

\begin{proposition} \label{pr:multivariate composite quasi-dc} \rm
Let $f_k : {\cal O} \to \mathbb{R}$ be a quasi-difference-convex function
on the open convex set ${\cal O} \subseteq \mathbb{R}^n$ for $k = 1, \cdots, m$
and $g : \Omega \to \mathbb{R}$ be a quasi-dc function where
$\Omega$ is an open convex set in $\mathbb{R}^m$ containing $F({\cal O})$,
where $F \triangleq ( f_k )_{k=1}^m$.  Then the composite function $g \circ F$  
is quasi-dc.  In particular \cite{LeThiHuynhPham24}, if $g$ is dc and $F$ is 
differentiable, then $g \circ F$ is quasi-dc.
\end{proposition}

\begin{proof}  The claim follows quite easily from the compositional
formula of direction derivatives which states: 
$( g \circ F )^{\, \prime}(\bar{x};v) = g^{\, \prime}(F(\bar{x});F^{\, \prime}(\bar{x};v))$;
this shows that the left-hand directional derivative is the composition
of the scalar-valued dc function 
$g^{\, \prime}(F(\bar{x});\bullet)$ (defined on $\mathbb{R}^m$) and the 
vector-valued directional derivative $F^{\, \prime}(\bar{x};\bullet) = 
( f_k^{\, \prime}(\bar{x};\bullet) )_{k=1}^m$, where each $f_k^{\, \prime}(\bar{x};\bullet)$
is dc on $\mathbb{R}^n$.  Thus by the compositional property of dc functions, it
follows that $( g \circ F )^{\, \prime}(\bar{x};\bullet)$ is dc.  Below, we provide
a proof that reveals some details of the latter directional derivative.

\gap

For every $\bar{y} \in \Omega$, there exist compact convex sets 
$\overline{\partial} g(\bar{y})$ and $\underline{\partial} g(\bar{y})$ such that 
\[
g^{\, \prime}(\bar{y};w) \, = \, \displaystyle{
\max_{a \in \overline{\partial} g(\bar{y})}
} \, a^{\top}w - \displaystyle{
\max_{b \in \underline{\partial} g(\bar{y})}
} \, b^{\top}w, \epc \forall \, w \, \in \, \mathbb{R}^m.
\]
Therefore, for any $\bar{x} \in {\cal O}$, with $\bar{y} \triangleq F(\bar{x})$,
we have
\[
( g \circ F )^{\, \prime}(\bar{x};v) \, = \, \displaystyle{
\max_{a \in \overline{\partial} g(\bar{y})}
} \, \displaystyle{
\sum_{k=1}^m
} \, a_k \, f_k^{\, \prime}(\bar{x};v) - \displaystyle{
\max_{b \in \underline{\partial} g(\bar{y})}
} \, \displaystyle{
\sum_{k=1}^m
} \, b_k \, f_k^{\, \prime}(\bar{x};v) \epc \forall \, v \, \in \, \mathbb{R}^n.
\]
It suffices to show that each summand is a dc function in $v$, given $\bar{x}$.
We prove the first summand only.  Since $\overline{\partial} g(\bar{y})$ is
a compact set, there exists a scalar $\overline{\beta}(\bar{x})$ such that
$| \, a_k \, | \, \leq \, \overline{\beta}(\bar{x})$ for all $k = 1, \cdots, m$
and all $a \in \overline{\partial} g(\bar{y})$.  Hence, writing 
$f_k^{\, \prime}(\bar{x};v) = \phi_k(\bar{x};v) - \varphi_k(\bar{x};v)$, where
$\phi_k(\bar{x};\bullet)$ and $\varphi_k(\bar{x};\bullet)$ are both convex functions, 
we deduce
\[ \begin{array}{ll}
\displaystyle{
\max_{a \in \overline{\partial} g(\bar{y})}
} \, \displaystyle{
\sum_{k=1}^m
} \, a_k \, f_k^{\, \prime}(\bar{x};v) \, = & -\overline{\beta}(\bar{x}) \,
\displaystyle{
\sum_{k=1}^m
} \, \left[ \, \phi_k^{\, \prime}(\bar{x};v) + \varphi_k^{\, \prime}(\bar{x};v) \, \right] 
\, + \\ [0.2in]
& \displaystyle{
\max_{a \in \overline{\partial} g(\bar{y})}
} \, \displaystyle{
\sum_{k=1}^m
} \, \left[ \, \underbrace{( \, a_k + \overline{\beta}(\bar{x}) \, )}_{\mbox{$\geq \, 0$}} 
\, \phi_k^{\, \prime}(\bar{x};v) 
+ \underbrace{( \, \overline{\beta}(\bar{x}) - a_k \, )}_{\mbox{$\geq \, 0$}} 
\, \varphi_k^{\, \prime}(\bar{x};v) \, \right].
\end{array} \]
This is enough to prove that $( g \circ F )^{\, \prime}(\bar{x};\bullet)$ is a dc
function.
\end{proof}

The above proposition facilitates the construction of many quasi-dc functions.  
As an illustration, we present the following corollary that pertains to a bivariate
composition which provides a good lead-in to the subsequent algorithmic development.

\begin{corollary} \label{co:composite quasi-dc} \rm
Let $\phi : \Omega \times \Omega \to \mathbb{R}$ be a bivariate Bouligand differentiable
function on the open convex set $\Omega \subseteq \mathbb{R}^m$
satisfying the sum formula (\ref{eq:sum formula for dd}) for its directional derivatives
at all pairs $(\bar{y},\bar{z}) \in \Omega \times \Omega$; i.e.,
\begin{equation} \label{eq:sum formula again}
\phi^{\, \prime}((\bar{y},\bar{z});(u,w)) \, = \, 
\phi_y^{\, \prime}((\bar{y},\bar{z});u) + \phi_z^{\, \prime}((\bar{y},\bar{z});w),
\epc \forall \, (u,w) \, \in \, \mathbb{R}^{2m}.
\end{equation}  	
Let $(n_k,d_k) : {\cal O} \to \mathbb{R}^2$ be a pair of quasi-dc functions on
the open convex set ${\cal O} \subseteq \mathbb{R}^n$ with
$N({\cal O}) \subseteq \Omega$ and $D({\cal O}) \subseteq \Omega$, where
$N(x) \triangleq ( n_k(x) )_{k=1}^m$ and $D(x) \triangleq ( d_k(x) )_{k=1}^m$.
If $\phi(\bullet,z)$ and $\phi(y,\bullet)$ are quasi-dc on ${\cal O}$ for 
all fixed pairs $(y,z) \in {\cal O} \times {\cal O}$, then $\phi \circ (N,D)$ is
quasi-dc on ${\cal O}$.
\end{corollary}

\begin{proof}  The sum formula (\ref{eq:sum formula again}) shows that the
bivariate function $\phi$ is quasi-dc on ${\cal O} \times {\cal O}$.
The desired conclusion follows readily 
from Proposition~\ref{pr:multivariate composite quasi-dc}.
\end{proof}

\subsection{Stationarity of minimization problems} \label{subsec:stationarity}

For a quasi-dc function with $f^{\, \prime}(\bar{x};v) = \displaystyle{
\max_{a \in \overline{\partial} f(\bar{x})}
} \, a^{\top} v - \displaystyle{
\max_{b \in \underline{\partial} f(\bar{x})}
} \, b^{\top} v$ for two compact convex upper and lower sets 
$\overline{\partial} f(\bar{x})$ and $\underline{\partial} f(\bar{x})$, respectively, 
by definition \cite{DemyanovPolyakova79,DemyanovRubinov00}, $\bar{x}$ is 
an (unconstrained) {\sl inf-stationary solution} of $f$ on $\mathbb{R}^n$ if
$-\overline{\partial} f(\bar{x}) \subseteq \underline{\partial} f(\bar{x})$.
In terms of directional derivatives, this definition can be shown to be
equivalent to $f^{\, \prime}(\bar{x};v) \geq 0$ for all $v \in \mathbb{R}^n$,
by \cite[Corollary~13.1.1]{Rockafellar70}; we call the latter condition
{\sl directional stationarity}
to signify that the stationarity is based on the directional derivative.  The advantage
of the latter condition versus the former definition is that it is much easier and
more transparent to extend
to the constrained case.  In fact, while there are extensions of inf-stationarity
to constrained problems \cite{Gao00,Shapiro86}, they are based on multipliers of
the inequality constraints that are defined by quasidifferentiable functions and
tend to be quite abstract due to the complexity of upper and lower sets and the
reliance on the primal-dual formulations.  For our
purpose, we adopt the following definition of a directional stationary point $\bar{x}$
for the constrained problem (\ref{eq:unified framework}); namely, for a convex set $X$,
$\bar{x}$ is directionally stationary for this problem if 
\[
f^{\, \prime}(\bar{x};x - \bar{x}) \, \geq \, 0, \epc \forall \, x \, \in \, X.
\]
For many quasi-dc minimization problems, computing such a stationary solution is
not easy and stationary solutions of a relaxed kind are the limit of the state-of-the-art 
methods that can be developed.  We will introduce such a relaxed kind of stationary 
solutions in due course.

\section{Composite Quasi-dc Functions for Minimization} 
\label{sec:quasi-dc for minimization}

While the last section, particularly Proposition~\ref{pr:multivariate composite quasi-dc} 
that pertains to composite functions, shows the breadth
of the class of quasi-dc functions in potential applications, such a generality comes 
with the difficulty of designing practically effective algorithms for minimizing 
a quasi-dc function. With such a task in mind, we need to be realistic with its
daunting challenge and the attainable goal.  In principle,
if elements of the two upper and lower sets $\overline{\partial} f(\bar{x})$ and 
$\underline{\partial} f(\bar{x})$ are available, algorithms can be designed and their
convergence analyzed.  Nevertheless as seen from the proof of 
Proposition~\ref{pr:multivariate composite quasi-dc}, these two sets are not easily
available for composite functions in general.  From a practical perspective, it would 
be desirable to obtain these sets from the given function $f$ and to avoid enumerating
all the elements of the upper and lower sets, which may not be possible if these
are continuum sets.  Therefore, building on early proposals on algorithms for a class of
quasi-differenentiable programs
\cite{Bagirov00,DemyanovGSDixon86,DemyanovPolyakova79,DemyanovRubinov00,KDDixon86} and
on the recent references \cite{BotDaoLi22,BotLiTao24,LeThiHuynhPham24} for specialized
classes, our discussion in the rest of the paper focuses on several classes of 
composite quasi-dc functions for which contemporary convex-programming based 
algorithms can be designed and implemented.  The development borrows ideas from 
the family of majorization and surrogation methods described in \cite{CuiPang2021}.
A departure of the present approach is that the surrogation is applied to the directional
derivatives; in particular, the surrogation functions do not necessarily
majorize the given objective function, thus there is no guarantee of 
its {\sl direct} descent when descent directions are generated from the surrogation
functions.  This is in contrast to the basic surrogation algorithms in \cite{CuiPang2021}
for which such direct descent is possible.  A prime example of an algorithm with the latter
property is the basic DCA and its extensions presented in \cite[Section~7.1]{CuiPang2021},
all of which are based on the fact that the original function, and not just its directional
derivatives, is majorized.  Without direct descent, line searches are needed for descent 
in this extended context; one such line-search
based descent algorithm for an extended dc program can be found 
in \cite[Subsection~7.2.3]{CuiPang2021}.  Extending and unifying
the existing developments in the literature, the discussion below prepares
the algorithmic design for solving (\ref{eq:unified framework}).

\gap

Specifically, we consider  the following min-finite-max-composite-quasi-dc program: 
for a given integer $J \geq 1$,
\begin{equation} \label{eq:unified framework}
\displaystyle{
\operatornamewithlimits{\mbox{\bf minimize}}_{x \in X}
} \ \Theta_{\max}(x) \, \triangleq \, \displaystyle{
\max_{1 \leq j \leq J}
} \, \theta_j(x), \epc \mbox{where} \epc 
\theta_j(x) \, \triangleq \, \phi_j \circ P^j(x).
\end{equation}
Here $X$ is a closed convex subset of $\mathbb{R}^n$; called an {\sl outer function},
each $\phi_j : \mathbb{R}^K \to \mathbb{R}$ is a Bouligand differentiable 
function with properties to be specified; called an {\sl inner function},
$P^{\, j}(x) \triangleq \left( \, p_{jk}(x) \, \right)_{k=1}^K$ 
is a $K$-dimensional
vector function defined on $\mathbb{R}^n$ whose components $p_{jk}$ are
all Boulingand differentiable functions with properties to be specified.  Thus
each function $\theta_j$ is Boulingand differentiable, and so is the pointwise
maximum objective $\Theta_{\max}$.

\gap

The algorithm to be developed for computing a stationary solution of 
(\ref{eq:unified framework}) follows a classical descent method
with Armijo-type line searches \cite{Bertsekas16}.  Such a method requires
two main computations: (a) generation of a descent direction, 
and (b) step size determination.  The main challenges of these two tasks
lie in the nonconvex and nondifferentiable composite
structure of the present problem that demand a non-gradient based generation 
of the descent directions.  Borrowing many ideas from the family of surrogation
methods for modern nonconvex and nondifferentiable optimization described
in \cite[Chapter~7]{CuiPang2021}, we will first derive the convex
functions that are the building blocks for the construction of the descent directions
in the algorithms to be presented later.  As mentioned above,
one major departure of our present work
is that the surrogation functions are not necessarily majorizing the original
objective globally; thus no direct descent is guaranteed.  Instead, the majorization 
is for the directional derivatives and thus is valid for the original objective
only locally near the reference point. This motivates the return
to the classical methods of line searches.  Further, we do not restrict to the
methods using constant step sizes which typically require knowledge of some
model constants.  Such line-search free methods have occupied much of the recent 
literature on first-order methods and their analysis; our development will allow
the omission of line searches with sufficient prior problem information.

\subsection{Matching pairs for descent}

As the first step to derive a surrogate function for the directional derivative
$\Theta_{\max}^{\, \prime}(\bar{x};\bullet)$ of the pointwise maximum objective 
function at a given $\bar{x} \in X$, 
we provide a well-known formula for this directional derivative.  
Specifically, for any $x \in X$, define the index set:
\[
{\cal M}_{\Theta}(x) \, \triangleq \, \left\{ \, 
j \, \in \, [ \, J \, ] \, \mid \, \theta_j(\bar{x}) 
\, = \, \Theta_{\max}(x) \, \right\},
\]
For a given $\bar{x}$, we have the following
formula for the directional derivatives of $\Theta_{\max}$ in terms 
of this index set ${\cal M}_{\Theta}(\bar{x})$, 
\begin{equation} \label{eq:obj dd} 
\Theta_{\max}^{\, \prime}(\bar{x};v) \, \triangleq \,
\displaystyle{
\lim_{\tau \downarrow 0}
} \, \displaystyle{
\frac{\Theta_{\max}(\bar{x} + \tau v) - \Theta_{\max}(\bar{x})}{\tau}
} \, = \, \displaystyle{
\max_{j \in {\cal M}_{\Theta}(\bar{x})}
} \, \left\{ \, \theta_j^{\, \prime}(\bar{x};v) \, \right\}.
\end{equation}
For later purposes, we also define, for an arbitrary $\varepsilon > 0$, the expanded
index set:
\[
{\cal M}_{\Theta}^{\, \varepsilon}(x) \, \triangleq \, \left\{ \, 
j \, \in \, [ J ] \, \mid \, \theta_j(x) 
\, \geq \, \Theta_{\max}(x) - \varepsilon \, \right\}.
\]
For every fixed $\varepsilon > 0$ and all $x$ sufficiently close to $\bar{x}$, we have 
\begin{equation} \label{eq:index set inclusions}
{\cal M}_{\Theta}(x) \, \subseteq \, {\cal M}_{\Theta}(\bar{x}) \, \subseteq \, 
{\cal M}_{\Theta}^{\, \varepsilon}(x).
\end{equation}
In the families of composite functions presented in the next several subsections, the
inner functions are all dc plus a finite-max of differentiable functions; i.e., each
\begin{equation} \label{eq:dc+diff}
p_{jk}(x) \, \triangleq \, \underbrace{p_{jk}^{\, \rm cvx}(x) + 
p_{jk}^{\, \rm cve}(x)}_{\mbox{\begin{tabular}{l}
difference of convex \\ [2pt]
denoted $p_{jk}^{\, \rm dc}(x)$
\end{tabular}}} + \underbrace{\displaystyle{
\max_{1 \leq \ell \leq L}
} \, p^{\rm diff}_{jk\ell}(x)}_{\mbox{denoted $p_{jk;\max}^{\, \rm diff}(x)$}}, 
\epc x \, \in \, \mathbb{R}^n,
\end{equation}
where $p_{jk}^{\, \rm cvx/cve}$ are convex/concave (not necessarily differentiable) 
functions, respectively, and each
$p_{jk\ell}^{\, \rm diff}(x)$ is continuously differentiable but not necessarily 
convex or concave.  This class of inner functions unifies two sets of studies:
(a) the early works of \cite{Bagirov00,DemyanovGSDixon86,KDDixon86} that do not
have the dc summand but has the pointwise maximum summand $p_{jk;\max}^{\, \rm diff}$
and the recent
work of \cite{LeThiHuynhPham24} where the inner functions are differentiable; and
(b) the extensive works of fractional programs where the inner functions range
from differentiable ones to the modern class of simple non-functions 
\cite{BotDaoLi22,BotLiTao24} and other simpler settings.  In principle,
the outer function $\phi$ could also be of the same kind (\ref{eq:dc+diff}); however
we find it much clearer to divide the discussion into four types because the ways to
obtain the final approximations of the composite function $\theta_j$ 
are different and dependent on whether the outer function $\phi_j$ is differentiable,
convex, or dc.  These separate treatments will be put back together via a family
of surrogation functions with appropriate properties when we return 
to the final objective function $\Theta_{\max}$ that is pointwise maximum of the
composites $\theta_j$ for $j = 1, \cdots, J$.  To simplify the notation, we omit 
the index $j$ in the pair $( \phi_j,P^{\, j})$ in the following subsections
when we derive the surrogation functions for the directional derivatives.

\subsubsection{Type I compositions: Differentiable outer functions}

This class of composite functions is given by
\[
\theta_{\, \rm I}(x) \, = \, \phi \circ P(x), \epc \mbox{where} \epc 
P(x) \, \triangleq \, \left( p_k(x) \right)_{k=1}^K \, \equiv \, y,
\]
with $\phi : \mathbb{R}^K \to \mathbb{R}$ being continuously differentiable.
Similar to the index set  ${\cal M}_{\Theta}(x)$, we define
\[
{\cal M}_{k;\max}^{\, {\rm diff}}(x) \, \triangleq \, \left\{ \, 
\ell \, \in \, [ L ] \, \mid \, p_{k\ell}^{\rm diff}(x) \, = \, 
p_{k;\max}^{\, \rm diff}(x) \, \right\}.
\]
We have
\[
p_k^{\, \prime}(\bar{x};v) \, = \, ( p_k^{\rm cvx} )^{\, \prime}(\bar{x};v) +
( p_k^{\rm cve} )^{\, \prime}(\bar{x};v) + \displaystyle{
\max_{\ell \in {\cal M}_{k;\max}^{\rm diff}(\bar{x})}
} \, \nabla p_{k\ell}^{\rm diff}(\bar{x})^{\top}v, 
\epc \forall \, v \, \in \, \mathbb{R}^n.
\]
Hence with $\phi$ being continuously differentiable, we obtain
\[ \begin{array}{lll}
\theta_{\, \rm I}^{\, \prime}(\bar{x};v) & \triangleq & \displaystyle{
\sum_{k=1}^K
} \, \displaystyle{
\frac{\partial \phi(\bar{y})}{\partial y_k}
} \, p_k^{\, \prime}(\bar{x};v); \epc \mbox{where } \ \bar{y} \triangleq P(\bar{x})
\\ [0.2in]
& = & \displaystyle{
\sum_{k=1}^K
} \, \displaystyle{
\frac{\partial \phi(\bar{y})}{\partial y_k}
} \, \left[ \,  ( p_k^{\rm cvx} )^{\, \prime}(\bar{x};v) +
( p_k^{\rm cve} )^{\, \prime}(\bar{x};v) + \displaystyle{
\max_{\ell \in {\cal M}_{k;\max}^{\rm diff}(\bar{x})}
} \, \nabla p_{k\ell}^{\rm diff}(\bar{x})^{\top}v \, \right] \\ [0.2in]
& = & \displaystyle{
\sum_{k=1}^K
} \, \displaystyle{
\frac{\partial \phi(\bar{y})}{\partial y_k}
} \, \left[ \,  \displaystyle{
\max_{a \in \partial ( p_k^{\rm cvx} )(\bar{x})}
} \, a^{\top}v - \displaystyle{
\max_{b \in \partial ( -p_k^{\rm cve} )(\bar{x})}
} \, b^{\top}v + \displaystyle{
\max_{\ell \in {\cal M}_{k;\max}^{\rm diff}(\bar{x})}
} \, \nabla p_{k\ell}^{\rm diff}(\bar{x})^{\top}v \, \right].
\end{array}
\]
In order to introduce a surrogation function for 
$\theta_{\, \rm I}^{\, \prime}(\bar{x};\bullet)$,
we define two further index sets:
\[
{\cal K}_{\phi}^{\, +}(\bar{y}) \, \triangleq \, \left\{ \, k \, \in \, [ K ] \, \left| \,
\displaystyle{
\frac{\partial \phi(\bar{y})}{\partial y_k}
} \, > \, 0 \, \right. \right\} \ \mbox { and } \ 
{\cal K}_{\phi}^{\, -}(\bar{y}) \, \triangleq \, \left\{ \, k \, \in \, [ K ] \, \left| \,
\displaystyle{
\frac{\partial \phi(\bar{y})}{\partial y_k}
} \, < \, 0 \, \right. \right\}.
\]
We then define, for a given tuple $\boldsymbol{\xi} \, \triangleq 
\left( \mathbf{a}, \mathbf{b}, \boldsymbol{\mathcal L} \, \right)$ in the set
\begin{equation} \label{eq:definition of Xi} 
\boldsymbol{\Xi}(\bar{x}) \, \triangleq \, \left\{ \, \begin{array}{l}
\left( \mathbf{a}, \mathbf{b}, \boldsymbol{\mathcal L} \, \right) \, \triangleq \, 
( \, a^k,b^k, \ell_k\, )_{k=1}^K \ \mbox{such that for each $k = 1, \cdots, K$:} \\ [0.1in]
( \, a^k,b^k \, ) \, \in  \, \underbrace{\partial p_k^{\rm cvx}(\bar{x}) \, \times \, 
\partial (-p_k^{\rm cve})(\bar{x})}_{\mbox{compact and convex}} \ \mbox{ and } \
\ell_k \, \in \, \underbrace{{\cal M}_{k;\max}^{\, \rm diff}(\bar{x})}_{\mbox{
a finite set}} 
\end{array} \right\},
\end{equation} 
the convex function $\wh{\theta}_{\, \rm I}(\bullet;\bar{x};\boldsymbol{\xi})$:
\[ \begin{array}{l}
\wh{\theta}_{\, \rm I}(x;\bar{x};\boldsymbol{\xi}) \, \triangleq \, 
\underbrace{\theta_{\rm I}(\bar{x})}_{\mbox{constant}} \, +  \\ [0.3in]
\underbrace{\left\{ \begin{array}{l}
\displaystyle{
\sum_{k \in {\cal K}_{\phi}^{\, +}(\bar{y})}
} \, \underbrace{\displaystyle{
\frac{\partial \phi(\bar{y})}{\partial y_k}
}}_{\mbox{positive}} \, \left[ \, \underbrace{\underbrace{p_k^{\rm cvx}(x) - 
p_k^{\rm cvx}(\bar{x})}_{\mbox{original}} 
- ( b^{\, k} )^{\top}(x - \bar{x}) + \displaystyle{
\max_{\ell \in {\cal M}_{k;\max}^{\, \rm diff}(\bar{x})}
} \, \nabla p_{k \ell}^{\rm diff}(\bar{x})^{\top}(x - \bar{x})}_{\mbox{convex in $x$ for 
fixed $\bar{x}$}} \, \right] \, + \\ [0.6in]
\displaystyle{
\sum_{k \in {\cal K}_{\phi}^{\, -}(\bar{y})}
} \, \underbrace{\displaystyle{
\frac{\partial \phi(\bar{y})}{\partial y_k}
}}_{\mbox{negative}} \, \left[ \, \underbrace{( a^k )^{\top}(x - \bar{x}) + 
\underbrace{p_k^{\rm cve}(x) - 
p_k^{\rm cve}(\bar{x} )}_{\mbox{original}} \, + \,
\nabla p_{k \ell_k}^{\rm diff}(\bar{x})^{\top}( x - \bar{x})}_{\mbox{concave in $x$ for 
fixed $\bar{x}$}} \, \right] 
\end{array} \right\}}_{\mbox{surrogation of $\theta_{\rm I}^{\, \prime}(\bar{x};x - \bar{x})$
at $\bar{x}$; equal to zero for $x = \bar{x}$}} \\ [1.2in]
= \, \theta_{\rm I}(\bar{x}) \, + \\ [0.1in]
\epc \displaystyle{
\sum_{k=1}^K
} \, \max\left( \, \begin{array}{l}
\displaystyle{
\frac{\partial \phi(\bar{y})}{\partial y_k}
} \, \left[ \, p_k^{\rm cvx}(x) - p_k^{\rm cvx}(\bar{x})
- ( b^{\, k} )^{\top}(x - \bar{x}) + \displaystyle{
\max_{\ell \in {\cal M}_{k;\max}^{\, \rm diff}(\bar{x})}
} \, \nabla p_{k \ell}^{\rm diff}(\bar{x})^{\top}(x - \bar{x}) \, \right], \\ [0.2in]
\displaystyle{
\frac{\partial \phi(\bar{y})}{\partial y_k}
}\, \left[ \, ( a^k )^{\top}(x - \bar{x}) + p_k^{\rm cve}(x) - p_k^{\rm cve}(\bar{x}) +
\nabla p_{k \ell_k}^{\rm diff}(\bar{x})^{\top}( x - \bar{x})\, \right]
\end{array} \right)
\end{array}
\]
Besides convexity (hence continuity), the function 
$\wh{\theta}_{\, \rm I}(\bullet;\bar{x};\boldsymbol{\xi})$ satisfies:

\gap

$\bullet $ {\bf (touching at the reference vector)} 
$\wh{\theta}_{\rm I}(\bar{x};\bar{x};\boldsymbol{\xi}) = \wh{\theta}_{\rm I}(\bar{x})$;
for all tuples $\boldsymbol{\xi}$;

\gap

$\bullet $ {\bf (directional derivative dominance)} for all  tuples
$\boldsymbol{\xi} \in \boldsymbol{\Xi}(\bar{x})$ and all vectors 
$v \in \mathbb{R}^n$,
it holds that $\left[ \, \wh{\theta}_{\, \rm I}(\bullet;\bar{x};\boldsymbol{\xi}) 
\, \right]^{\, \prime}(\bar{x};v) 
\, \geq \, \theta_{\rm I}^{\, \prime}(\bar{x};v)$ moreover, for every $\bar{v}$,
there exists a tuple $\bar{\boldsymbol{\xi}} \in \boldsymbol{\Xi}(\bar{x})$ 
such that
$\left[ \, \wh{\theta}_{\, \rm I}(\bullet;\bar{x};\bar{\boldsymbol{\xi}}) 
\, \right]^{\, \prime}(\bar{x};\bar{v}) 
= \theta_{\rm I}^{\, \prime}(\bar{x};\bar{v})$;

\gap

$\bullet $ {\bf (directional derivative consistency)} if 
$\boldsymbol{\Xi}(\bar{x}) \equiv \left\{ \boldsymbol{\xi} \right\}$ 
is a singleton, then we have the equality
$\left[ \, \wh{\theta}_{\, \rm I}(\bullet;\bar{x};\boldsymbol{\xi}) 
\, \right]^{\, \prime}(\bar{x};v) 
\, = \, \theta_{\rm I}^{\, \prime}(\bar{x};v)$ for all $v$.

\gap

Hence, if 
$\left[ \, \wh{\theta}_{\, \rm I}(\bullet;\bar{x};\boldsymbol{\xi}) 
\, \right]^{\, \prime}(\bar{x};v) > \theta_{\rm I}^{\, \prime}(\bar{x};v)$, then there exists
$\bar{\tau} > 0$ (which is dependent on $v$) such that
$\wh{\theta}_{\, \rm I}(\bar{x} + \tau v;\bar{x};\boldsymbol{\xi})
> \theta_{\, \rm I}(\bar{x} + \tau v)$ for all $\tau \in ( 0, \bar{\tau} ]$; i.e.,
$\wh{\theta}_{\, \rm I}(\bullet;\bar{x};\boldsymbol{\xi})$ majorizes
$\theta_{\, \rm I}$ locally near $\bar{x}$ along the direction $v$; nevertheless, 
$\wh{\theta}_{\, \rm I}(\bullet;\bar{x};\boldsymbol{\xi})$ does not majorize
$\theta_{\, \rm I}$ globally.

\gap

Additionally, we can also define the parameter-free function
\[
\wh{\theta}_{\, \rm I}^{\, \min}(x;\bar{x}) \, \triangleq \, \displaystyle{
\min_{\boldsymbol{\xi} \in \boldsymbol{\Xi}(\bar{x})}
} \ \wh{\theta}_{\, \rm I}(x;\bar{x};\boldsymbol{\xi}),
\]
which although is not convex in $x$, provides a tighter majorization of 
the directional derivative $\theta_{\rm I}^{\, \prime}(\bar{x};\bullet)$;
namely, 
\[
\left[ \, \wh{\theta}_{\, \rm I}(\bullet;\bar{x};\boldsymbol{\xi}) 
\, \right]^{\, \prime}(\bar{x};v) \, \geq \, \displaystyle{
\min_{\boldsymbol{\xi}^{\, \prime} \in \boldsymbol{\Xi}(\bar{x})}
} \, \left\{ \left[ \, \wh{\theta}_{\, \rm I}(\bullet;\bar{x};\boldsymbol{\xi}^{\, \prime}) 
\, \right]^{\, \prime}(\bar{x};v) \, \right\} \, = \,
\left[ \, \wh{\theta}_{\, \rm I}^{\, \min}(\bullet;\bar{x}) 
\, \right]^{\, \prime}(\bar{x};v) \geq 
\theta_{\rm I}^{\, \prime}(\bar{x};v)
\]
for all $( v,\boldsymbol{\xi} ) \in \mathbb{R}^n \times \boldsymbol{\Xi}(\bar{x})$,
where the first inequality holds because 
$\wh{\theta}_{\, \rm I}(\bar{x};\bar{x};\bullet) \equiv \theta_{\rm I}(\bar{x})$ is
a constant on 
$\boldsymbol{\Xi}(\bar{x})$ and the second inequality holds by directional derivative
dominance.  

\gap

{\bf The parameterizing family $\boldsymbol{\Xi}(\bar{x})$:}  
Being dependent only on
the inner functions, and thus the same for all differentiable outer functions, 
this family plays the 
surrogate role of the upper and lower sets for the composite direction derivative 
$\theta^{\, \prime}(\bar{x};\bullet)$.  For general convex/concave functions
$p_k^{\, \rm cvx/cve}$, the set $\boldsymbol{\Xi}(\bar{x})$ is a continuum
even though the third component ${\cal M}_{k;\max}^{\, \rm diff}(\bar{x})$ 
is a finite set.  When the functions $p_k^{\, \rm cvx} \triangleq \displaystyle{
\max_{1 \leq i \leq I}
} \, g_{ki}$ and  $-p_k^{\, \rm cve} \triangleq \displaystyle{
\max_{1 \leq i \leq I}
} \, h_{ki}$ are of the 
finite-max-convex-differentiable kind (i.e.\ $g_{ki}$ and $h_{ki}$ are all convex and
differentiable), we can 
take $\boldsymbol{\Xi}(\bar{x})$ to be a finite
set as well by replacing $\partial p_k^{\, \rm cvx}(\bar{x})$ and 
$\partial (-p_k^{\, \rm cve})(\bar{x})$ by the respective gradients of the active functions 
at $\bar{x}$ that define these finite-pointwise-maxima functions.  This remark applies
similarly to the functions in the other types.

\gap

The family $\boldsymbol{\Xi}(\bar{x})$ can be considered as the image of the 
set-valued map $\boldsymbol{\Xi} : x \mapsto \boldsymbol{\Xi}(x)$ at the vector
$\bar{x}$.  With its definition (\ref{eq:definition of Xi}), it is easy to see
that this map is closed (equivalently, upper semi-continuous because of its
compact-valuedness); i.e., 
if $\{ x^{\nu} \} \to x^{\infty}$ and $\{ \boldsymbol{\xi}^{\, \nu} \} \to 
\boldsymbol{\xi}^{\, \infty}$ where $\boldsymbol{\xi}^{\, \nu} \in \boldsymbol{\Xi}(x^{\nu})$
for all $\nu$, then $\boldsymbol{\xi}^{\, \infty} \in  \boldsymbol{\Xi}(x^{\infty})$.  Here,
we need to be a bit clear about the closedness and upper semicontinuity.  With
$\boldsymbol{\xi}^{\, \nu} = \left( a^{\nu;k},b^{\nu;k};\ell_k^{\, \nu}\, \right)_{k=1}^K$,
the convergence of the pair $\left( a^{\nu;k},b^{\nu;k} \right) \in 
\partial p_k^{\rm cvx}(x^{\nu}) \times \partial (-p_k^{\rm cve})(x^{\nu})$ to a pair in
$\partial p_k^{\rm cvx}(x^{\infty}) \times \partial (-p_k^{\rm cve})(x^{\infty})$ requires
no explanation.  Nevertheless, since each $\ell_k^{\, \nu}$ is an integer, by the convergence
of $\{ \ell_k^{\, \nu} \}$ to $\ell_k^{\, \infty}$, it means that this is finite
sequence whose members are all equal to the limit for all $\nu$ sufficiently large.  
Furthermore, the upper semicontinuity of $\boldsymbol{\Xi}$ at $\bar{x}$ means that for
every $\varepsilon > 0$, there exists a neighborhood ${\cal N}$ of $\bar{x}$ such that
\[
\boldsymbol{\Xi}(x) \, \subseteq \, \boldsymbol{\Xi}(x^{\infty}) + \varepsilon \, 
\left[ \, \mathbb{B}_n(0,1) \times \mathbb{B}_n(0,1) \times \{ \ell_k^{\, \infty} \}_{k=1}^K
\, \right],
\]
where $\mathbb{B}_n(0,1)$ is the closed unit Euclidean ball in $\mathbb{R}^n$ with center
at the origin and unit radius.  It is well known that the subdifferential of a convex function
has this continuity property.  Another well-known property of a convex function $f$ is that
$f^{\, \prime}(\bullet;\bullet)$ is upper semicontinuous (usc) at every pair $(\bar{x},\bar{w})$;
i.e., for all sequences $\{ (x^{\nu},w^{\nu}) \}$ converging to $(\bar{x},\bar{w})$, it
holds that
\[
\displaystyle{
\limsup_{\nu \to \infty}
} \, f^{\, \prime}(x^{\nu};w^{\nu}) \, \leq \, f^{\, \prime}(\bar{x},\bar{w}).
\]
The following lemma generalizes the above directional derivative inequality to 
$\left[ \, \wh{\theta}_{\rm I}(\bullet;\bar{x};\boldsymbol{\xi}) \, \right]^{\, \prime}(\wh{x};w)$
as a nonconvex function of the reference vector $\bar{x}$, 
the pair $(\wh{x},w)$ and the parameter $\boldsymbol{\xi}$.   We 
call (\ref{eq:limsup inequality}) the {\sl dd-joint upper semicontinuity} of the 
function $\wh{\theta}_{\rm I}$.  The lemma also gives two additional limiting
properties of the function $\wh{\theta}_{\rm I}$.

\begin{proposition} \label{pr:convergence of theta dd} \rm
Suppose that $\phi$ and all the functions $p_{k\ell}^{\rm diff}$ are continuously
differentiable and that  $\boldsymbol{\Xi}$ is nonempty-valued, compact-valued
and upper semicontinuous.  Then

\gap

$\bullet $ {\bf (lower Lipschitz boundedness)} for every bounded 
set $S \subseteq \mathbb{R}^n$, the union $\displaystyle{
\bigcup_{x \in S}
} \, \boldsymbol{\Xi}(x)$ is bounded, moreover, corresponding to $S$,
there exists a constant $B > 0$ such that 
\[
\wh{\theta}_{\rm I}(x;\bar{x};\boldsymbol{\xi}) - 
\wh{\theta}_{\rm I}(\bar{x};\bar{x};\boldsymbol{\xi}) \, \geq \, 
-B \, \| \, x - \bar{x} \, \|_2, \epc \forall \, ( \, x,\bar{x} \, )  \, \in \, 
X \, \times \, S \ \mbox{ and all } \ \boldsymbol{\xi} \, \in \, \boldsymbol{\Xi}(\bar{x}).
\]
Furthermore, let the following sequences be given: $\{ \bar{x}^{\nu} \} \to \bar{x}^{\infty}$, 
$\{ \wh{x}^{\, \nu} \} \to \wh{x}^{\, \infty}$ and 
$\{ \boldsymbol{\xi}^{\, \nu} \} \to \boldsymbol{\xi}^{\, \infty}$ with each tuple 
belonging to $\boldsymbol{\Xi}(\bar{x}^{\nu})$, the following additional statements hold:

\gap

$\bullet $ {\bf (joint continuity at the reference vector})
if $\bar{x}^{\infty} = \wh{x}^{\, \infty}$, then
$\displaystyle{
\lim_{\nu \to \infty}
} \, \wh{\theta}_{\, \rm I}(\wh{x}^{\, \nu};\bar{x}^{\nu};\boldsymbol{\xi}^{\, \nu}) 
= \theta_{\rm I}(\bar{x}^{\infty})$;

\gap

$\bullet $ {\bf (uniform upper approximation at the reference vector)} 
if $\bar{x}^{\infty}$ is such that 
${\cal M}_{k;\max}^{\, \rm diff}(\bar{x}^{\infty})$ is a singleton for all $k$ satisfying
$\displaystyle{
\frac{\partial \phi(\bar{y}^{\infty})}{\partial y_k}
} > 0$, where $\bar{y}^{\infty} \triangleq P(\bar{x}^{\infty})$, then 
for every sequence $\{ \tau_{\nu} \} \downarrow 0$, it holds that:
\begin{equation} \label{eq:uniform upper dd}
\displaystyle{
\limsup_{\nu \to \infty}
} \ \displaystyle{
\frac{\theta_{\rm I}(\bar{x}^{\nu} + \tau_{\nu} ( \wh{x}^{\, \nu} - \bar{x}^{\nu})) 
- \theta_{\rm I}(\bar{x}^{\nu}) - \tau_{\nu} \, \left[ \, 
\wh{\theta_{\rm I}}(\wh{x}^{\, \nu};\bar{x}^{\nu};\boldsymbol{\xi}^{\, \nu})
- \theta_{\rm I}(\bar{x}^{\nu}) \, \right]}{\tau_{\nu}}
} \, \leq 0;
\end{equation}
$\bullet $ {\bf (dd-joint upper semicontinuity)} for every sequence 
$\{ w^{\nu} \} \to w^{\infty}$, it holds that
\begin{equation} \label{eq:limsup inequality}
\displaystyle{
\limsup_{\nu \to \infty}
} \, \left[ \, 
\wh{\theta}_{\rm I}(\bullet;\bar{x}^{\nu},\boldsymbol{\xi}^{\, \nu}) 
\, \right]^{\, \prime}
(\wh{x}^{\, \nu};w^{\nu}) \, \leq \, 
\left[ \, \wh{\theta}_{\rm I}(\bullet;\bar{x}^{\infty};\boldsymbol{\xi}^{\, \infty})
\, \right]^{\, \prime}(\wh{x}^{\, \infty};w^{\infty}).
\end{equation}
\end{proposition}

\begin{proof}  Omitting the proof of the former two properties, we prove only the last two. 
To prove (\ref{eq:uniform upper dd}), write 
$\bar{x}^{\tau;\nu} \triangleq \bar{x}^{\nu} + \tau_{\nu} ( \wh{x}^{\, \nu} - \bar{x}^{\nu})$,
$\bar{y}^{\nu} \triangleq P(\bar{x}^{\nu})$, 
$\bar{z}^{\nu} \triangleq P(\bar{x}^{\tau;\nu})$.  Observe that
\[
\displaystyle{
\lim_{\nu \to \infty}
} \, \bar{x}^{\tau;\nu} \, = \, \displaystyle{
\lim_{\nu \to \infty}
} \, \bar{x}^{\nu} \, = \, \bar{x}^{\infty} 
\epc \mbox{and} \epc \displaystyle{
\lim_{\nu \to \infty}
} \, P(\bar{x}^{\tau;\nu}) = \displaystyle{
\lim_{\nu \to \infty} 
} \, P(\bar{x}^{\nu}) \, = \, P(\bar{x}^{\infty}) \, = \, \bar{y}^{\infty}.
\]
By the singleton assumption of
${\cal M}_{k;\max}^{\, \rm diff}(\bar{x}^{\infty})$, say
${\cal M}_{k;\max}^{\, \rm diff}(\bar{x}^{\infty}) = \{ \, \bar{\ell}_k \, \}$,
it follows that for all $\nu$ sufficiently large,
\[
\displaystyle{
\frac{\partial \phi(\bar{y}^{\infty})}{\partial y_k}
} \, > \, 0 \ \Rightarrow \ {\cal M}_{k;\max}^{\, \rm diff}(\bar{x}^{\nu}) \, = \,
{\cal M}_{k;\max}^{\, \rm diff}(\bar{x}^{\tau;\nu}) \, = \, \{ \, \bar{\ell}_k \}.
\]
Hence, for all $k$ with $\displaystyle{
\frac{\partial \phi(\bar{y}^{\infty})}{\partial y_k}
} > 0$, which implies $\displaystyle{
\frac{\partial \phi(y)}{\partial y_k}
} \geq 0$ for all $y$ sufficiently close to $\bar{y}^{\infty}$,
we have, for all $\nu$ sufficiently large,
\[ \begin{array}{l}
p_{k;\max}^{\rm diff}(\bar{x}^{\tau;\nu}) - p_{k;\max}^{\rm diff}(\bar{x}^{\nu}) -
\tau_{\nu} \, \displaystyle{
\max_{\ell \in {\cal M}_{k;\max}^{\, \rm diff}(\bar{x}^{\nu})}
} \, \nabla p_{k \ell}^{\rm diff}(\bar{x}^{\nu})^{\top}( \wh{x}^{\, \nu} - \bar{x}^{\nu})
\\ [0.2in]
\, = \, p_{k \bar{\ell}_k}(\bar{x}^{\tau;\nu}) - 
p_{k \bar{\ell}_k}(\bar{x}^{\nu}) - \tau_{\nu} \, 
\nabla p_{k \bar{\ell}_k}^{\rm diff}(\bar{x}^{\nu})^{\top}( \wh{x}^{\, \nu} - \bar{x}^{\nu})
\\ [0.1in]
\, = \, \tau_{\nu} \, \left[ \, 
\nabla p_{k \bar{\ell}_k}(\bar{x}^{\tau^{\prime};\nu}) - 
\nabla p_{k \bar{\ell}_k}(\bar{x}^{\nu}) \, \right]^{\top} 
( \wh{x}^{\, \nu} - \bar{x}^{\nu}) \, = \, \mbox{o}( \tau_{\nu} ) 
\end{array} \]
because $\bar{x}^{\, \tau^{\, \prime};\nu}$ and $\bar{x}^{\nu}$ both converge 
to $\bar{x}^{\infty}$, where $\bar{x}^{\tau^{\, \prime};\nu} \triangleq 
\bar{x}^{\nu} + \tau_{\nu}^{\, \prime} ( \wh{x}^{\, \nu} - \bar{x}^{\nu})$
for some $\tau_{\nu}^{\, \prime} \in ( 0, \tau_{\nu} )$. 
On the other hand, for $k$ such that $\displaystyle{
\frac{\partial \phi(\bar{y}^{\infty})}{\partial y_k}
} < 0$, which implies $\displaystyle{
\frac{\partial \phi(y)}{\partial y_k}
} \leq 0$ for all $y$ sufficiently near $\bar{y}^{\infty}$, we have for any 
$\ell_k^{\, \nu} \in {\cal M}_{k;\max}^{\, \rm diff}(\bar{x}^{\nu})$,
\[ \begin{array}{l}
p_{k;\max}^{\rm diff}(\bar{x}^{\tau;\nu}) - p_{k;\max}^{\rm diff}(\bar{x}^{\nu}) -
\tau_{\nu} \, \nabla p_{k \ell_k^{\, \nu}}^{\rm diff}(\bar{x}^{\nu})^{\top}
( \wh{x}^{\, \nu} - \bar{x}^{\nu}) \\ [0.1in]
\, \geq \, p_{k\ell_k^{\, \nu}}(\bar{x}^{\tau;\nu}) - 
p_{k\ell_k^{\, \nu}}(\bar{x}^{\nu}) - \tau_{\nu} \, 
\nabla p_{k \ell_k^{\, \nu}}^{\rm diff}(\bar{x}^{\nu})^{\top}( \wh{x}^{\, \nu} - \bar{x}^{\nu})
\\ [0.1in]
\, = \, \tau_{\nu} \, \left[ \, 
\nabla p_{k\ell_k^{\, \nu}}(\bar{x}^{\tau^{\prime};\nu}) - 
\nabla p_{k\ell_k^{\, \nu}}(\bar{x}^{\nu}) \, \right]^{\top} 
( \wh{x}^{\, \nu} - \bar{x}^{\nu}) \, = \, \mbox{o}( \tau_{\nu} ),
\end{array}
\]
where $\bar{x}^{\tau^{\, \prime};\nu}$ is similarly defined.
Therefore, with  $\boldsymbol{\xi}^{\, \nu}
= \left( a^{\nu;k},b^{\nu;k};\ell_k^{\, \nu}\, \right)_{k=1}^K$, where 
$\ell_k^{\, \nu} = \bar{\ell}_k$ for all $k$ such that $\displaystyle{
\frac{\partial \phi(\bar{y}^{\infty})}{\partial y_k}
} > 0$, we have
\[ \begin{array}{l}
\theta_{\rm I}(\bar{x}^{\nu} + \tau_{\nu} ( \wh{x}^{\, \nu} - \bar{x}^{\nu})) - 
\theta_{\rm I}(\bar{x}^{\nu}) \, = \, \phi(\bar{z}^{\nu}) - \phi(\bar{y}^{\nu}) \\ [0.1in]
= \, \nabla \phi(\wt{z}^{\, t;\nu})^{\top}( \bar{z}^{\nu} - \bar{y}^{\nu} ),
\epc \mbox{where } \ \wt{z}^{\, t;\nu} \, \triangleq \, \bar{y}^{\nu} + 
t_{\nu} ( \bar{z}^{\nu} - \bar{y}^{\nu} ) \ \mbox{ for some } 
t_{\nu} \in ( \, 0, 1 \,) \\ [0.1in]
= \, \underbrace{\left[ \, \nabla \phi(\wt{z}^{\, t;\nu}) -  \nabla \phi(\bar{y}^{\nu}) 
\, \right]^{\top}( \bar{z}^{\nu} - \bar{y}^{\nu} )}_{\mbox{denoted $T1$}} +
\underbrace{\nabla \phi(\bar{y}^{\nu})^{\top}( \bar{z}^{\nu} - \bar{y}^{\nu} )}_{
\mbox{denoted $T2$}}.
\end{array} \]
Since both $\bar{x}^{\tau;\nu}$ and $\bar{x}^{\nu}$ converge to $\bar{x}^{\infty}$ and
$P$ is locally Lipschitz continuous near $\bar{x}^{\infty}$,
\[
\bar{z}^{\nu} - \bar{y}^{\nu} \, = \, P(\bar{x}^{\tau;\nu}) - P(\bar{x}^{\nu})
\, \leq \,\tau_{\nu} \, \mbox{Lip}_P \, \| \, \wh{x}^{\, \nu} - \bar{x}^{\nu} \, \|_2,
\]
for some local Lipschitz constant $\mbox{Lip}_P > 0$.  Moreover, since both
$\wt{z}^{\, t;\nu}$ and $\bar{y}^{\nu}$ converge to $\bar{y}^{\infty}$, it follows that
$T1 = \mbox{o}(\tau_{\nu})$.  For $T2$, we have
\[ \begin{array}{l}
T2 \, = \, \displaystyle{
\sum_{k \in {\cal K}_{\phi}^{\, +}(\bar{y}^{\nu})}
} \, \displaystyle{
\frac{\partial \phi(\bar{y}^{\nu})}{\partial y_k}
} \, \left[ \, p_k^{\rm cvx}(\bar{x}^{\tau;\nu}) - p_k^{\rm cvx}(\bar{x}^{\nu}) + 
p_k^{\rm cve}(\bar{x}^{\tau;\nu}) - p_k^{\rm cve}(\bar{x}^{\nu}) +
p_{k;\max}^{\rm diff}(\bar{x}^{\tau;\nu}) - p_{k;\max}^{\rm diff}(\bar{x}^{\nu}) 
\, \right] \\ [0.25in]
+ \,  \displaystyle{
\sum_{k \in {\cal K}_{\phi}^{\, -}(\bar{y}^{\nu})}
} \, \displaystyle{
\frac{\partial \phi(\bar{y}^{\nu})}{\partial y_k}
} \, \left[ \, p_k^{\rm cvx}(\bar{x}^{\tau;\nu}) - p_k^{\rm cvx}(\bar{x}^{\nu}) + 
p_k^{\rm cve}(\bar{x}^{\tau;\nu}) - p_k^{\rm cve}(\bar{x}^{\nu}) +
p_{k;\max}^{\rm diff}(\bar{x}^{\tau;\nu}) - p_{k;\max}^{\rm diff}(\bar{x}^{\nu}) 
\, \right] \\ [0.2in]
\leq \, \tau_{\nu} \displaystyle{
\sum_{k \in {\cal K}_{\phi}^{\, +}(\bar{y}^{\nu})}
} \, \displaystyle{
\frac{\partial \phi(\bar{y}^{\nu})}{\partial y_k}
} \, \left[ \,  \begin{array}{l}
p_k^{\rm cvx}(\wh{x}^{\, \nu}) - p_k^{\rm cvx}(\bar{x}^{\nu}) 
- ( b^{\, \nu;k} )^{\top}( \wh{x}^{\, \nu} - \bar{x}^{\nu} ) \, + \\ [0.1in]
\displaystyle{
\max_{\ell \in {\cal M}_{k;\max}^{\rm diff}(\bar{x}^{\nu})}
} \, \nabla p_{k \ell}^{\rm diff}(\bar{x}^{\nu})^{\top}( \wh{x}^{\, \nu} - \bar{x}^{\nu} ) 
+ \displaystyle{
\frac{\mbox{o}(\tau_{\nu})}{\tau_{\nu}}
} 
\end{array} \right] \, + \\ [0.35in]
\epc \tau_{\nu} \, \displaystyle{
\sum_{k \in {\cal K}_{\phi}^{\, -}(\bar{y}^{\nu})}
} \, \displaystyle{
\frac{\partial \phi(\bar{y}^{\nu})}{\partial y_k}
} \, \left[ \, \begin{array}{l}
( a^{\nu;k} )^{\top}( \wh{x}^{\, \nu} - \bar{x}^{\nu} ) +
p_k^{\rm cve}(\wh{x}^{\, \nu}) - p_k^{\rm cve}(\bar{x}^{\nu}) \, + \\ [0.1in]
\nabla p_{k \ell_k^{\nu}}^{\rm diff}(\bar{x}^{\nu})^{\top}
( \bar{x}^{\tau;\nu} - \bar{x}^{\nu} ) + \displaystyle{
\frac{\mbox{o}(\tau_{\nu})}{\tau_{\nu}}
} 
\end{array} \right] \ \left(\mbox{\begin{tabular}{l}
by convexity of $p_k^{\rm cvx}$ \\ [3pt]
and concavity of $p_k^{\rm cve}$
\end{tabular}} \right) \\ [0.35in]

= \, \tau_{\nu} \, \left[ \, 
\wh{\theta}_{\rm I}(\wh{x}^{\, \nu};\bar{x}^{\nu};\boldsymbol{\xi}) 
- \theta_{\rm I}(\bar{x}^{\nu}) \, \right] + \mbox{o}(\tau_{\nu}).
\end{array}
\]
Putting the bounds for $T1$ and $T2$ together establishes (\ref{eq:uniform upper dd}).
To show the last property, let additionally
$\bar{y}^{\infty} \triangleq P(\bar{x}^{\infty})$, we have, since
${\cal K}_{\phi}^{\pm}(\bar{y}^{\infty}) \subseteq {\cal K}_{\phi}^{\pm}(\bar{y}^{\nu})$
for all $\nu$ sufficiently large, and the union
${\cal K}_{\phi}^+(\bar{y}^{\infty}) \, \cup \, {\cal K}_{\phi}^-(\bar{y}^{\infty}) \, \cup \,
{\cal K}_{\phi}^0(\bar{y}^{\infty}) = {\cal K}_{\phi}^+(\bar{y}^{\nu}) \, \cup \, 
{\cal K}_{\phi}^-(\bar{y}^{\nu}) \, \cup \,
{\cal K}_{\phi}^0(\bar{y}^{\nu}) \, = \, \{ \, 1, \cdots, K \}$, where the 
superscript~$0$ corresponds to those
indices $k$ for which the partial derivatives of $\phi$ with respect to $y_k$ 
are equal to zero:
\begin{equation} \label{eq:theta surrogate dd}
\begin{array}{l}
\left[ \, \wh{\theta}_{\, \rm I}(\bullet;\bar{x}^{\nu};\boldsymbol{\xi}^{\, \nu}) 
\, \right]^{\, \prime}(\wh{x}^{\, \nu};w^{\nu}) \\ [0.1in] 
= \, \displaystyle{
\sum_{k \in {\cal K}_{\phi}^{\, +}(\bar{y}^{\nu})}
} \, \displaystyle{
\frac{\partial \phi(\bar{y}^{\nu})}{\partial y_k}
} \, \left[ \, ( p_k^{\rm cvx} )^{\, \prime}( \wh{x}^{\, \nu};w^{\nu} ) 
- ( b^{\, \nu;k} )^{\top}w^{\nu} + \displaystyle{
\max_{\ell \in {\cal M}_{k;\max}^{\rm diff}(\bar{x}^{\nu})}
} \, \nabla p_{k \ell}^{\rm diff}(\bar{x}^{\nu})^{\top}w^{\nu} \, \right] \, + \\ [0.3in]
\epc \displaystyle{
\sum_{k \in {\cal K}_{\phi}^{\, -}(\bar{y}^{\nu})}
} \, \displaystyle{
\frac{\partial \phi(\bar{y}^{\nu})}{\partial y_k}
} \, \left[ \, ( a^{\nu;k} )^{\top}w^{\nu} + 
( p_k^{\rm cve} )^{\, \prime}(\bar{x}^{\nu};w^{\nu} ) + 
\nabla p_{k \ell_k^{\nu}}^{\rm diff}(\bar{x}^{\nu})^{\top}w^{\nu}, \right] \\ [0.3in]
= \, \displaystyle{
\sum_{k \in {\cal K}_{\phi}^{\, +}(\bar{y}^{\infty})}
} \, \displaystyle{
\frac{\partial \phi(\bar{y}^{\nu})}{\partial y_k}
} \, \left[ \, ( p_k^{\rm cvx} )^{\, \prime}( \wh{x}^{\, \nu};w^{\nu} ) 
- ( b^{\, \nu;k} )^{\top}w^{\nu} + \displaystyle{
\max_{\ell \in {\cal M}_{k;\max}^{\rm diff}(\bar{x}^{\nu})}
} \, \nabla p_{k \ell}^{\rm diff}(\bar{x}^{\nu})^{\top}w^{\nu} \, \right] \, + \\ [0.3in]
\epc \displaystyle{
\sum_{k \in {\cal K}_{\phi}^{\, 0}(\bar{y}^{\infty})}
} \ \mbox{terms approaching zero} \\ [0.3in]
\epc \displaystyle{
\sum_{k \in {\cal K}_{\phi}^{\, -}(\bar{y}^{\infty})}
} \, \displaystyle{
\frac{\partial \phi(\bar{y}^{\nu})}{\partial y_k}
} \, \left[ \, ( a^{\nu;k} )^{\top}w^{\nu} + 
( p_k^{\rm cve} )^{\, \prime}(\bar{x}^{\nu};w^{\nu} ) + 
\nabla p_{k \ell_k^{\nu}}^{\rm diff}(\bar{x}^{\nu})^{\top}w^{\nu}, \right]
\end{array}
\end{equation}
versus
\[ \begin{array}{l}
\left[ \, \wh{\theta}_{\, \rm I}(\bullet;\bar{x}^{\infty};\boldsymbol{\xi}^{\, \infty}) 
\, \right]^{\, \prime}(\wh{x}^{\, \infty};w^{\infty}) \\ [0.1in] 
= \, \displaystyle{
\sum_{k \in {\cal K}_{\phi}^{\, +}(\bar{y}^{\infty})}
} \, \displaystyle{
\frac{\partial \phi(\bar{y}^{\infty})}{\partial y_k}
} \, \left[ \, ( p_k^{\rm cvx} )^{\, \prime}( \wh{x}^{\, \infty};w^{\infty} ) 
- ( b^{\, \infty;k} )^{\top}w^{\infty} + \displaystyle{
\max_{\ell \in {\cal M}_{k;\max}^{\rm diff}(\bar{x})}
} \, \nabla p_{k \ell}^{\rm diff}(\bar{x}^{\infty})^{\top}w^{\infty} \, \right] \, + \\ [0.3in]
\epc \displaystyle{
\sum_{k \in {\cal K}_{\phi}^{\, -}(\bar{y}^{\infty})}
} \, \displaystyle{
\frac{\partial \phi(\bar{y}^{\infty})}{\partial y_k}
} \, \left[ \, ( a^{\infty;k} )^{\top}w^{\infty} + 
( p_k^{\rm cve} )^{\, \prime}(\bar{x}^{\infty};w^{\infty} ) + 
\nabla p_{k \ell_k^{\infty}}^{\rm diff}(\bar{x}^{\infty})^{\top}w^{\infty}, \right]
\end{array}
\]
Comparing the summands in the above two expressions one by one, and noting that
the signs of $\displaystyle{
\frac{\partial \phi(\bar{y}^{\nu})}{\partial y_k}
}$ and $\displaystyle{
\frac{\partial \phi(\bar{y}^{\infty})}{\partial y_k}
}$ are the same provided that the latter is zero and also that $\ell_k^{\, \nu}$ is 
equal to $\ell_k^{\, \infty}$ for all but finitely many $\nu$, we see that 
each term (indexed by $k$)
in the former directional derivative is upper bounded by the corresponding one in the latter
when $\nu$ is passed to the limit, while those with $\displaystyle{
\frac{\partial \phi(\bar{y}^{\infty})}{\partial y_k}
} = 0$ converges to zero, all by the continuous differentiability of the functions
$\phi$ and $p_{k\ell}^{\rm diff}$.  This comparison is enough to establish the
inequality (\ref{eq:limsup inequality}).
\end{proof}
 
\begin{remark} \label{eq:singleton requirement} \rm
Unlike the other properties, the uniform upper approximation inequality 
(\ref{eq:uniform upper dd}) is the only one that requires the maximizing 
index set ${\cal M}_{k;\max}^{\, \rm diff}(\bar{x})$ to be a singleton 
at a given point $\bar{x}$ for $k$ corresponding to a positive partial derivative
$\displaystyle{
\frac{\partial \phi(\bar{y})}{\partial y_k}
}$.  This is a pointwise requirement and will be imposed at a limit point in
the convergence analysis of the
algorithm to be developed later; see Remark~\ref{rm:inequality in convergence}.  
\hfill $\Box$
\end{remark}

An extreme case of the parameterized family of functions 
$\wh{\theta}_{\, \rm I}(\bullet;\bar{x};\boldsymbol{\xi})$ is worthmentioning.
This is the case when no parameterization
by $\boldsymbol{\xi}$ is needed (in this case, we take $\boldsymbol{\Xi}(\bar{x})$ to be
a singleton).  We say that the pair $(\phi,P)$ satisfies the 

\gap

$\bullet $ {\sl aggregate composite convexity} (AC$^{\, 2}$) property 
at $\bar{x}$ if (a) the sum function $\displaystyle{
\sum_{k=1}^K
} \, \displaystyle{
\frac{\partial \phi(\bar{y})}{\partial y_k}
} \, p_k^{\rm dc}$ is convex, and (b) the max functions $p_{k;\max}^{\, \rm diff}$ 
are all differentiable at $\bar{x}$ for $k = 1, \cdots, K$
(or more restrictively, when $L = 1$);

\gap

$\bullet $ {\sl componentwise composite convexity} (C$^{\, 3}$) property 
at $\bar{x}$ if (a) the individual functions $\left\{ \displaystyle{
\frac{\partial \phi(\bar{y})}{\partial y_k}
} \, p_k^{\rm dc} \, \right\}_{k=1}^K$ are all convex and the same condition (b)
holds.

\gap

If either one of the above two composite convexity properties holds (for simplification,
take $L =1$ 
and write $p_k^{\rm diff}(x)$ for the differentiable summand of $p_k$), we can define
the unparameterized convex surrogate function $\wh{\theta}_{\, \rm I}(\bullet;\bar{x})$ by
\[
\wh{\theta}_{\, \rm I}(x;\bar{x}) \, = \, \theta_{\, \rm I}(\bar{x}) + \displaystyle{
\sum_{k=1}^K
} \, \displaystyle{
\frac{\partial \phi(\bar{y})}{\partial y_k}
} \, \left[ \, p_k^{\rm dc}(x) - p_k^{\rm dc}(\bar{x}) + 
\nabla p_k^{\rm diff}(\bar{x})^{\top} ( x - \bar{x} ) \, \right]; \epc x \, \in \, 
\mathbb{R}^n.
\]
Another situation for $\boldsymbol{\Xi}(\bar{x})$ to be a singleton without  
the aggregate $\displaystyle{
\sum_{k=1}^K
} \, \displaystyle{
\frac{\partial \phi(\bar{y})}{\partial y_k}
} \, p_k^{\rm dc}$ being convex is when
the pair $(\phi,P)$ satisfies the 

\gap

$\bullet $ {\sl componentwise composite differentiable convexity-concavity} 
(C$^{\, 2}$DC$^{\, 2}$) property at $\bar{x}$, which means that 
the following conditions hold for all $k = 1, \cdots, K$:

\gap

(a1) $\displaystyle{
\frac{\partial \phi(\bar{y})}{\partial y_k}
} > 0 \Rightarrow \partial (-p_k^{\rm cve})(\bar{x})$ is a singleton (equivalently,
$p_k^{\rm cve}$ is strictly differentiable at $\bar{x}$);

\gap

(a2) $\displaystyle{
\frac{\partial \phi(\bar{y})}{\partial y_k}
} < 0 \Rightarrow \partial (p_k^{\rm cvx})(\bar{x})$ is a singleton (equivalently,
$p_k^{\rm cvx}$ is strictly differentiable at $\bar{x}$);

\gap

(b) ${\cal M}_{k;\max}^{\, \rm diff}(\bar{x})$ is a singleton (equivalent, 
$p_{k;\max}^{\, \rm diff}$ is differentiable at $\bar{x}$).

\gap

Note that while the (C$^{\, 3}$) condition is a special case of the (AC$^{\, 2}$) condition,
the condition (C$^{\, 2}$DC$^{\, 2}$) is not; nevertheless, both properties will have 
an important role to play in the convergence
rate analysis.

\gap

{\bf Some examples:}.  We give some prominent examples of differentiable outer functions,
starting with several simple ones:

\begin{equation} \label{eq:quotients and products} 	
\begin{array}{l}
\phi_{\, \rm quot}(y,z) \, = \, \displaystyle{
\sum_{k=1}^K
} \ \displaystyle{
\frac{y_k}{z_k}
}, \epc ( y,z ) \, \in \, \mathbb{R}^K \, \times \mathbb{R}^K_{++} 
\ \mbox{ and} \\ [0.2in]
\phi_{\, \rm prod}(y,z) \, = \, \displaystyle{
\sum_{k=1}^K
} \ y_k \, z_k, \epc ( y,z ) \, \in \, \mathbb{R}^K \, \times \mathbb{R}^K,
\end{array} \end{equation}
and the trivariate extension:
\[
\phi(y,z,w) \, = \,  \displaystyle{
\sum_{k=1}^K
} \ \displaystyle{
\frac{y_k \, w_k}{z_k}
}, \epc ( y,z,w ) \, \in \, \mathbb{R}^K \, \times \mathbb{R}^K_{++} \times \,
\mathbb{R}^K_.
\]
In particular, the quotient function $\phi_{\rm quot}$ has provided a strong
initial motivation for the present study.
In \cite{ShenYu18-I}, an algorithm for solving the sum-of-ratios minimization
problem (without the outer pointwise maximum operator):
\[
\displaystyle{
\operatornamewithlimits{\mbox{\bf minimize}}_{x \in X}
} \ \theta_{\rm frac}(x) \, \triangleq \, \displaystyle{
\sum_{k=1}^K
} \ \displaystyle{
\frac{n_k(x)}{d_k(x)}
} \]
was developed based on the composite-function formulation
\[
\displaystyle{
\operatornamewithlimits{\mbox{\bf minimize}}_{x \in X}
} \ \theta_{\rm frac}(x) \, \triangleq \, \phi_{\rm spec}\left( \sqrt{N(x)},D(x) \right)
\]
where $\phi_{\rm spec}(y,z) = \displaystyle{
\sum_{k=1}^K
} \ \displaystyle{
\frac{y_k^2}{z_k}
}$, for $( y,z ) \in \mathbb{R}^K_+ \, \times \mathbb{R}^K_{++}$;
$\sqrt{N(x)} = \left( \, \sqrt{n_k(x)} \, \right)_{k=1}^K$ and 
$D(x) = \left( \, d_k(x) \, \right)_{k=1}^K$; here besides the positivity
of the denominator functions $d_k(x)$, the numerator functions $n_k(x)$ need 
to be nonnegative for their square roots to be well defined.  Thus this
special square-root formulation of the sum-of-ratios minimization problem
is covered by our composite setting with a differentiable outer function.

\gap

We illustrate the two composite convexity properties for the fractional function
$\theta_{\rm frac}$ using the two outer functions $\theta_{\rm quot}$ and
$\theta_{\rm spec}$.  For the former, the aggregate
composite convexity property holds if the sum function
\[
\displaystyle{
\sum_{k=1}^K
} \, \displaystyle{
\frac{1}{d_k(\bar{x})}
} \, \left[ \, n_k(x) - \displaystyle{
\frac{n_k(\bar{x})}{d_k(\bar{x})}
} \, d_k(x) \, \right]
\]
is convex; in turn, this holds if each $d_k$ is positive concave function, and
each $n_k$ is a nonnegative convex function, a setting that is akin to the classical
DinKelbach algorithm for fractional programs.  Clearly, if the latter
concavity/convexity condition is in place for the individual functions, 
then the componentwise convexity property also holds.  For
the special function $\phi_{\rm spec}$, the aggregate
composite convexity property holds if the sum function
\[
\displaystyle{
\sum_{k=1}^K
} \, \left[ \, \displaystyle{
\frac{2 n_k(\bar{x})}{d_k(\bar{x})}
} \,  \sqrt{n_k(x)} - \displaystyle{
\frac{n_k(\bar{x})^2}{d_k(\bar{x})^2}
} \, d_k(x) \, \right]
\]
is convex; this requirement and also the componentwise convexity property are clearly
more demanding because the two properties involve the convexity of the square root
functions $\sqrt{n_k(x)}$.  Lastly, for the product function
\[
\theta_{\rm prod}(x) \, = \, \displaystyle{
\sum_{k=1}^K
} \, n_k(x) \, d_k(x),
\]
the aggregate composite convexity property holds if the sum function
\[
\displaystyle{
\sum_{k=1}^K
} \, \left[ \, d_k(\bar{x}) \, n_k(x) + n_k(\bar{x}) \, d_k(x) \, \right]
\]
is convex; this is satisfied if for instance both functions $n_k$ and $d_k$ 
are nonnegative and convex, or more generally, there is a matching
between the signs and convexity/concavity of the functions $n_k$ and $d_k$.

\gap

A further class of the outer functions $\phi$ that is of interest is
the family of the deviation functions.  These may come in 
different forms; basic examples include the squared deviations such as
\[
\phi_{\, \rm SumSq}(y,z) \, = \, \displaystyle{
\sum_{k=1}^K
} \ \left( \, \displaystyle{
\frac{y_k}{z_k}
} - \alpha_k \, \right)^2 \epc \mbox{or} \epc
\phi_{\, \rm SqSum}(y,z) \, = \, \left( \, \displaystyle{
\sum_{k=1}^K
} \ \displaystyle{
\frac{y_k}{z_k}
} - \gamma_j \, \right)^2
\]
for given constants $\alpha_j$ and $\gamma_k$.  In statistical estimation
problems, it may be of interest to consider robust versions of the squared
deviations to guard against outliers in datasets, an example of which is
based on the Huber loss function \cite{Huber73} that is obtained by 
linearizing the square function (for squared errors) after a
threshold: $\phi(y,z) \, = h_{\delta} \circ g(y,z)$, where for a 
given scalar $\alpha > 0$,
\[
h_{\delta}(t) \, \triangleq \left\{ \begin{array}{ll}
\thalf \, t^{\, 2} & \mbox{if $| \, t \, | \, \leq \delta$} \\ [0.1in]
\thalf \, \delta^{\, 2} + \delta \, \left( | \, t \, | - \delta \right) 
& \mbox{if $| \, t \, | \, < \delta$},
\end{array} \right.
\]
and the function $g(y,z) = \displaystyle{
\sum_{k=1}^K
} \, \displaystyle{
\frac{y_k}{z_k}
} - \gamma_j$ (as an example).  One interpretation of the constant $\gamma_k$ 
in this context is a target of the sum of fractions for a dataset partitioned 
in $K$ groups and the overall objective is to minimize the maximum of 
some measure of the deviations from the targets each defined by the Huber function.  
We note that the Huber function $h_{\delta}$ is trivially seen to be convex and 
once continuously differentiable.  Parameterized by two parameter $\delta > 0$
and $\gamma > 2$,
the well-known {\sc scad} function \cite{FanLi01} used as an approximation 
of the $\ell_0$ function
$| \, t \, |_0 \triangleq \left\{ \begin{array}{ll}
1 & \mbox{if $t \neq 0$} \\
0 & \mbox{otherwise}
\end{array} \right.$ in sparse approximation \cite{HastieTibshiraniWainwright15}
and given by
\[
h_a^{\rm SCAD}(t;\delta) \, = \, \left\{ \begin{array}{cl}
\displaystyle{
\frac{2a}{( a + 1) \delta}
} \, | \, t \, | & \mbox{if $| \, t \, | \leq \delta/a$} \\ [0.2in]
1 - \displaystyle{
\frac{\left( \, \delta - | \, t \, | \, \right)^2}{\left( \, 1 - \displaystyle{
\frac{1}{a^2}		
} \, \right) \, \delta^2}
} & \mbox{if $\displaystyle{
\frac{\delta}{a}
} \, \leq \, | \, t \, | \, \leq \, \delta$} \\ [0.35in]
1 & \mbox{if $| \, t \, | \, \geq \, \delta$}
\end{array} \right\} \epc \mbox{for $t \in \mathbb{R}$}
\]
is once continuously differentiable.  Used for constraint selection (as a generalization
of variable selection), the composite 
function $\displaystyle{
\sum_{i=1}^I
} \, h_a^{\rm SCAD}(\bullet;\delta) \, \circ \, g_i(x)_+$ is an approximation 
of $\displaystyle{
\sum_{i=1}^I
} \, | \, g_i(x)_+ \, |_0$, where $t_+ \triangleq \max( t,0 )$ is the univariate
plus-function; the latter sum is a count
of the violation of the constraints: $g_i(x) \leq 0$, $i = 1, \cdots, I$.  Note
the nondifferentiability of the plus function in the {\sc scad}-composite function.

\subsubsection{Type II compositions: Vector-convexified outer functions} 

The building block of this class of composite functions is the univariate composition: 
\begin{equation} \label{eq:univariate convex}
\theta_{\rm II}(x) \, = \, \psi \circ p(x), \epc \mbox{where } \ \left\{ \begin{array}{l}
\psi : \mathbb{R} \to \mathbb{R} \mbox{ is convex} \\ [0.1in]
p(x) = p_{\rm cvx}(x) + p_{\rm cve}(x) + \underbrace{\displaystyle{
\max_{1 \leq \ell \leq L}
} \, p_{\ell}^{\, \rm diff}(x)}_{\mbox{denoted $p_{\rm max}^{\rm diff}(x)$}}, 	
\end{array} \right.
\end{equation}
with $p_{\rm cvx}$, $p_{\rm cve}$, and $p_{\ell}^{\rm diff}$ all being univariate functions
that are convex, concave, and differentiable, respectively.  
By \cite[Lemma~6.1.1]{CuiPang2021}, we can write $\psi = \psi_{\uparrow} + \psi_{\downarrow}$,
where $\psi_{\updownarrow}$ are convex monotonic functions with $\psi_{\uparrow}$ being
nondecreasing and $\psi_{\downarrow}$ being nonincreasing.  Moreover, $\psi_{\updownarrow}$ 
can be chosen to be Lipschitz continuous if $\psi$ is and continuously differentiable 
if $\psi$ is.  We write ${\cal M}_p^{\rm diff}(\bar{x})$ for 
${\cal M}_{k;\max}^{\rm diff}(\bar{x})$ without the subscript $k$ 
in the present univariate context; i.e.,
\[
{\cal M}_p^{\rm diff}(\bar{x}) \, \triangleq \, \left\{ \, \ell \, \in \, [ \, L \, ] 
\, : \, p_{\ell}^{\rm diff}(\bar{x}) \, = \, p_{\max}^{\rm diff}(\bar{x}) \, \right\}.
\]
Similar to the case of a differentiable outer function, for a given triple 
\[
\xi \, \triangleq \, ( a,b,\ell ) \, \in \, 
\partial p_{\rm cvx}(\bar{x}) \times \partial (-p_{\rm cve})(\bar{x})
\times {\cal M}_p^{\rm diff}(\bar{x}) \triangleq
\Xi(\bar{x}), 
\]
we can define the convex function $\wh{\theta}_{\, \rm II}(\bullet;\bar{x};\xi)$:
\[ \begin{array}{ll}
\wh{\theta}_{\, \rm II}(x;\bar{x};\xi) \, \triangleq & 
\psi_{\uparrow}\left( \, \underbrace{p_{\rm cvx}(x) +
p_{\rm cve}(\bar{x}) - b^{\top}( x - \bar{x} ) + p_{\max}^{\rm diff}(\bar{x}) + \displaystyle{
\max_{\ell^{\prime} \in {\cal M}_p^{\rm diff}(\bar{x})}
} \, \nabla p_{\ell^{\prime}}^{\rm diff}(\bar{x})^{\top}( x - \bar{x})}_{\mbox{convex in
$x$ for fixed $\bar{x}$}} \, \right) + \\ [0.45in]
& \psi_{\downarrow}\left( \, \underbrace{p_{\rm cvx}(\bar{x}) + a^{\top}( x - \bar{x} ) 
+ p_{\rm cve}(x) + p_{\max}^{\rm diff}(\bar{x}) +
\nabla p_{\ell}^{\rm diff}(\bar{x})^{\top}( x - \bar{x})}_{\mbox{concave in $x$ for
for fixed $\bar{x}$}} \, \right);
\end{array} \]
this function satisfies the same six properties as the ones of 
$\wh{\theta}_{\, \rm I}(\bullet;\bar{x};\boldsymbol{\xi})$; namely, 
touching at the reference vector, joint continuity at the reference vector, lower
Lipchitiz boundeness, directional
derivative dominance, directional derivative consistency, and
dd-joint upper semicontinuity; details are omitted.  Regarding the uniform upper
approximation property:
\begin{equation} \label{eq:uniform upper dd univariate}
\displaystyle{
\limsup_{\nu \to \infty}
} \ \displaystyle{
\frac{\theta_{\rm II}(\bar{x}^{\nu} + \tau_{\nu} ( \wh{x}^{\, \nu} - \bar{x}^{\nu})) 
- \theta_{\rm II}(\bar{x}^{\nu}) - \tau_{\nu} \, \left[ \, 
\wh{\theta}_{\rm II}(\wh{x}^{\, \nu};\bar{x}^{\nu};\xi^{\nu})
- \theta_{\rm II}(\bar{x}^{\nu}) \, \right]}{\tau_{\nu}}
} \, \leq 0
\end{equation}
for sequences $\{ \bar{x}^{\nu} \} \to \bar{x}^{\infty}$, 
$\{ \wh{x}^{\, \nu} \} \to \wh{x}^{\, \infty}$, $\{ \tau_{\nu} \} \downarrow 0$
and $\{ \xi^{\nu} \} \to \xi^{\infty}$, it can be shown to hold provided that
${\cal M}_p^{\rm diff}(\bar{x}^{\infty})$ is a singleton at the limit vector 
$\bar{x}^{\infty}$ if $\psi_{\uparrow}$ is not identically equal to zero.  We
note that while all these properties hold for the surrogate function $\wh{\theta}_{\rm II}$,
the function $\wh{\theta}_{\rm II}(\bullet;\bar{x};\xi)$ still does not majorize
$\wh{\theta}_{\rm II}$ because of the functions $p_{\ell}^{\rm diff}$ that lack favorable
properties except differentiability.

\gap

Similarly, we can define the parameter-free function
\[
\wh{\theta}_{\, \rm II}^{\, \min}(x;\bar{x}) \, \triangleq \, \displaystyle{
\min_{\xi \in \Xi(\bar{x})}
} \ \wh{\theta}_{\, \rm II}(x;\bar{x};\xi);
\]
we also have 
$\left[ \, \wh{\theta}_{\, \rm II}(\bullet;\bar{x};\xi) \, \right]^{\, \prime}(\bar{x};v) 
\, \geq \, \left[ \, \wh{\theta}_{\, \rm II}^{\, \min}(\bullet;\bar{x}) 
\, \right]^{\, \prime}(\bar{x};v) \, \geq \, \theta_{\rm II}^{\, \prime}(\bar{x};v)$ for 
all $(v,\xi) \in \mathbb{R}^n \times \Xi(\bar{x})$.

\gap

There are two kinds of multivariate convex compositions derived from the univariate convex
composition.  One is in terms of 
a separable convex vector function $\Psi(y) = \displaystyle{
\sum_{k=1}^K
} \, \psi_k(y_k)$ with $y = ( y_k )_{k=1}^K \in \mathbb{R}^K$, 
where each $\psi_k : \mathbb{R} \to \mathbb{R}$ is convex, resulting in 
\[
\theta_{\, \rm II}^{\, \rm sep}(x) \, = \, \Psi \circ P(x) \, = \, \displaystyle{
\sum_{k=1}^K
} \, \psi_k \circ p_k(x),
\]
where $P(x) \triangleq ( p_k(x) )_{k=1}^K$ and each $p_k$ is given by
(\ref{eq:dc+diff}).  The other is via the composition of a univariate convex
function with an affine function: i.e., the multivariate
function $\Psi(y) = \psi(a^{\top}y + \alpha)$ for some vector 
nonnegative vector $a \in \mathbb{R}^K_+$ and scalar $\alpha$,
yielding
\[
\theta_{\, \rm II}^{\, \oplus}(x) \, \triangleq \, \psi\left( \displaystyle{
\sum_{k=1}^K
} \, a_k \, p_k(x) + \alpha \right);
\]
we note that the nonnegative scalar multiplication of the function $p_k$ 
remains a function of the same kind.  
Both ways lead to a ``vector convexified'' outer function $\Psi$ composite 
with dc+finite-max-differentiable inner
functions for which the obtained surrogation functions continue to satisfy similar 
directional derivative dominance and consistency properties as described above.

\gap

Extending the univariate convex composition (\ref{eq:univariate convex}), we may
consider its multiviate analog where we assume a monotonic property of the outer
convex function:
\[
\theta_{{\rm II};\updownarrow}^{\, \rm cvx}(x) \, = \, 
\phi_{\, \updownarrow}^{\, \rm cvx} \circ P(x), \epc \mbox{where } \ 
\left\{ \begin{array}{l}
\phi_{\, \updownarrow}^{\, \rm cvx} : \mathbb{R}^K \to \mathbb{R} \mbox{ is convex 
and either isotone ($\uparrow$) or antitone ($\downarrow$)} 
\\ [0.1in]
P(x) \, = \, \left( \, p_k(x) \, \right)_{k=1}^K \ \mbox{where each} \\ [0.1in]
p_k(x) \, \triangleq \, p_k^{\rm cvx}(x) + p_k^{\rm cve}(x) + 
\underbrace{\displaystyle{
\max_{1 \leq \ell \leq L}
} \, p_{k \ell}^{\, \rm diff}(x)}_{\mbox{denoted $p_{k;\max}^{\, \rm diff}(x)$}}.
\end{array} \right.
\]
Here isotonicity means: $x \leq x^{\, \prime}$ (componentwise) implies 
$\phi_{\uparrow}^{\rm cvx}(x) \leq \phi_{\uparrow}^{\rm cvx}(x^{\, \prime})$ 
for all $x$ and $x^{\, \prime}$ and antitonicity means 
$x \leq x^{\, \prime}$ (componentwise) implies 
$\phi_{\downarrow}^{\rm cvx}(x) \geq \phi_{\downarrow}^{\rm cvx}(x^{\, \prime})$.  
For $\theta_{{\rm II};\uparrow}^{\, \rm cvx}$, we define
\[ \begin{array}{l}
\wh{\theta}_{{\rm II};\uparrow}^{\, \rm cvx}(x;\bar{x};\boldsymbol{\xi}) \, = \, 
\phi_{\, \uparrow}^{\, \rm cvx} \circ P_{\, \uparrow}(x;\bar{x};\boldsymbol{\xi}), 
\epc \mbox{where } \
P_{\, \uparrow}(x;\bar{x};\boldsymbol{\xi}) = 
\left( \, p_{k;\uparrow}(x;\bar{x};\boldsymbol{\xi}) \, \right)_{k=1}^K 
\ \mbox{with each} \\ [0.15in]
p_{k;\uparrow}(x;\bar{x};\boldsymbol{\xi}) \, \triangleq \, 
p_k^{\rm cvx}(x) + p_k^{\rm cve}(\bar{x})
- ( b^k )^{\top}( x - \bar{x}) + p_{k;\max}^{\, \rm diff}(\bar{x}) +  \displaystyle{
\max_{\ell \in {\cal M}_{k;\max}^{\rm diff}(\bar{x})}
} \, \nabla p_{k \ell}^{\, \rm diff}(\bar{x})^{\top}( x - \bar{x}).
\end{array}
\] 
and for $\theta_{{\rm II};\downarrow}^{\, \rm cvx}$, we define
\[ \begin{array}{l}
\wh{\theta}_{{\rm II};\downarrow}^{\, \rm cvx}(x;\bar{x};\boldsymbol{\xi}) \, = \, 
\phi_{\, \downarrow}^{\, \rm cvx} \circ P_{\, \downarrow}(x;\bar{x};\boldsymbol{\xi}), 
\epc \mbox{where } \
P_{\, \downarrow}(x;\bar{x};\boldsymbol{\xi}) \, = \,
\left( \, p_{k;\downarrow}(x;\bar{x};\boldsymbol{\xi}) \, \right)_{k=1}^K 
\ \mbox{with each} \\ [0.15in]
p_{k;\downarrow;}(x;\bar{x};\boldsymbol{\xi}) \, \triangleq \, 
p_k^{\, \rm cvx}(\bar{x}) + ( a^k )^{\top}( x - \bar{x}) +
p_k^{\rm cve}(x) - p_k^{\rm cve}(\bar{x}) + p_{k;\max}^{\, \rm diff}(\bar{x}) +  
\nabla p_{k \ell_k}^{\, \rm diff}(\bar{x})^{\top}( x - \bar{x}).
\end{array}
\] 
It can be shown that the multivariate 
surrogation function $\wh{\theta}^{\rm cvx}_{{\rm II};\updownarrow}$
satisfies the same properties as those of the univariate function $\wh{\theta}_{\rm II}$,
where for the isotone function $\wh{\theta}^{\rm cvx}_{{\rm II};\uparrow}$,
the singleton condition of ${\cal M}_{k;\max}^{\, \rm diff}(\bar{x}^{\infty})$
is needed for the uniform upper approximation to hold at $\bar{x}^{\infty}$.

\gap

Finally, there is the last type of multivariate convex outer compositions where 
unparameterized surrogate
functions can be readily obtained.   These are the compositions where the inner functions
are differentiable, resulting in the class of amenable functions.  Specifically,
such a function is:
\[
\theta_{\, \rm II}^{\, \rm amn}(x) \, \triangleq \, \phi \circ P(x), \epc \mbox{where } \
\phi : \mathbb{R}^K \to \mathbb{R} \mbox{ is convex
and $P : \mathbb{R}^n \to \mathbb{R}^K$ is differentiable}.
\]
In this case, we can define the un-parameterized function
\[
\wh{\theta}_{\, \rm II}^{\, \rm amn}(x;\bar{x}) \, \triangleq \, 
\phi \circ [ \, P(\bar{x}) + JP(\bar{x})( x - \bar{x} ) \, ],
\]
where $JP(\bar{x})$ is the Jacobian matrix of $P$ at $\bar{x}$.  We omit the discussion
of properties of the bivariate
function $\wh{\theta}_{\, \rm II}^{\, \rm amn}(\bullet;\bar{x})$ as it belongs
to a familiar class of composite functions; see the early 
papers \cite{Burke1985,BurkeFerris1995} and
the recent addition \cite{BurkeHoheiselNguyen21}.

\gap

{\bf Some examples of the outer function:}. Among the simplest and most prominent 
multivariate nondifferentiable convex functions is the $\ell_1$ function
$\displaystyle{
\sum_{k=1}^K
} \, | \, y_k \, |$ which is often used in statistical regression.  A
composite function $\displaystyle{
\sum_{i=1}^n
} \, | \, y_i - y_{i+1} \, |$ measures the total variation of an image 
\cite{TibshiraniSaunders2005} represented by the discrete data points 
$\{ y_i \}_{i=1}^{n+1}$.  Variations of this least-absolute measure abound; 
e.g., the objective $\displaystyle{
\sum_{i=1}^n
} \, \max\left( \, | y_i - b_{i+1} y_{i+1} |, \, \varepsilon \, \right)$ is a measure 
of scaled near isotonicity \cite{TibHTibshirani11} with margin $\varepsilon > 0$
of the data points $\{ y_i \}_{i=1}^{n+1}$.  More generally, the summation
$\displaystyle{ 
\sum_{(i,j) \in V}
} \, g_{ij}( a_{ij} y_i - b_{ij} y_j )$, where each univariate function
$g_{ij} : \mathbb{R} \to \mathbb{R}_+$ is 
a convex function, $V \subseteq \{ 1, \cdots, n \} \times \{ 1, \cdots, n \}$,
and $\{ a_{ij},b_{ij} \}$ are given scalars, is a generalized deviation-separation
measure \cite{GomezHePang2023} with a special case discussed in \cite{HochbaumLu2017}
for median regression.  Studied in \cite{CuiPangSen18},
the least-squares piecewise affine regression problem has provided a strong initial
motivation to investigate convex composite with piecewise affine functions; replacing
the $\ell_2$-norm by a scaled $\ell_1$-norm poses this nonconvex regression problem
as the optimization of a nondifferentiable convex outer composite with piecewise affine inner 
functions.

\subsubsection{Type III compositions: dd-convex outer functions}

This class of composite functions extend the type~II composite functions 
to the case where the outer function $\phi$ is dd-convex.  To derive the surrogate
functions, we present a representation
of the directional derivative $g^{\, \prime}(\bar{x};\bullet)$, which is a convex
function when $g$ is dd-convex, where
$g : {\cal O} \to \mathbb{R}$ is a Bouligand differentiable differentiable
on the open convex set ${\cal O} \subseteq \mathbb{R}^n$.  
The {\sl regular subdifferential} of $g$ at $\bar{x}$ is defined as
\[ \begin{array}{lll}
\wh{\partial} g(\bar{x}) & \triangleq & 
\left\{ \, v \, \in \mathbb{R}^n \; \bigg| \, \displaystyle{
\liminf_{\bar{x} \neq x\to\bar{x}}
} \; \frac{\phi(x) - \phi(\bar{x})-v^{\top}(x-\bar{x})}{
\| \, x-\bar{x} \, \|} \, \geq \, 0 \, \right\} \\ [0.2in]
& = & \left\{ \, v \, \in \mathbb{R}^n \; \bigg| \; g^{\, \prime}(\bar{x};v) 
\, \geq \, v^{\top}d \ \mbox{for all $d \in \mathbb{R}^n$} \, \right\}
\end{array} \]
In general, $\wh{\partial} g(\bar{x})$ is compact and convex 
\cite[Theorem~8.6]{RockafellarWets98}; it is nonempty if
$g$ is dd-convex \cite[Proposition~4.3.3]{CuiPang2021}.  In this case, we
have
\[
g^{\, \prime}(\bar{x};v) \, = \, \displaystyle{
\max_{a \in \wh{\partial} g(\bar{x})}
} \, v^{\top}a, \epc \forall \, v \, \in \, \mathbb{R}^n;
\]
moreover, for any closed convex set $X \subseteq \mathbb{R}^n$, we further have,
with ${\cal N}(\bar{x};X)$ denoting the normal cone of $X$ at $\bar{x}$ as in
convex analysis \cite{Rockafellar70},
\begin{equation} \label{eq:dd-convex dd}
g^{\, \prime}(\bar{x};x - \bar{x}) \, = \, \displaystyle{
\max_{a^{\, \prime} \in \wh{\partial} g(\bar{x}) + {\cal N}(\bar{x};X)}
} \, ( x - \bar{x} )^{\top} a^{\, \prime}, \epc \forall \, x \, \in \, X.
\end{equation}
Indeed for any $a^{\, \prime} = a + b$ with 
$( a,b ) \in \wh{\partial} g(\bar{x}) \times {\cal N}(\bar{x};X)$, we have
\[
( x - \bar{x} )^{\top} a^{\, \prime} \, \leq \, ( x - \bar{x} )^{\top} a \, \leq \, 
g^{\, \prime}(\bar{x};x - \bar{x}).
\] 
Thus $g^{\, \prime}(\bar{x};x - \bar{x}) \geq \displaystyle{
\sup_{a^{\, \prime} \in \wh{\partial} g(\bar{x}) + {\cal N}(\bar{x};X)}
} \, ( x - \bar{x} )^{\top} a^{\, \prime} \geq \displaystyle{
\sup_{a^{\, \prime} \in \wh{\partial} g(\bar{x})}
} \, ( x - \bar{x} )^{\top} a^{\, \prime}$ because $0 \in {\cal N}(\bar{x};X)$.
Hence, for a dd-convex function $g$, directional stationarity of a vector
$\bar{x} \in X$ can be equivalently stated that 
$0 \in \wh{\partial} g(\bar{x}) + {\cal N}(\bar{x};X)$.  Lastly, for any open
set $S$ on which $g$ is Lipschitz continuous, the union
$\displaystyle{
\bigcup_{x \in S}
} \, \wh{\partial} g(x)$ is bounded.

\gap

One weakness of the regular subdifferential is its lack of upper semicontinuity.
We have the following result that
gives necessary and sufficient conditions for the upper semicontinuity of
$\wh{\partial} g$ at a vector $\bar{x}$ in terms of the Clarke subdifferential
\begin{equation} \label{defn: Clarke2}
\begin{array}{lll}
\partial_{\rm C} g(\bar{x}) & \triangleq & \left\{ \, v \, \in \, \mathbb{R}^n \, \mid \,
g^{\circ}(\bar{x};d) \, \geq \, v^{\top}d \ \forall \, d \, \in \, \mathbb{R}^n 
\, \right\}, \mbox{where} \\ [0.1in]
g^{\circ}(\bar{x};v) & \triangleq & \displaystyle{
\limsup_{\substack{x \to \bar{x} \\ \tau \downarrow 0}} 
} \, \displaystyle{\frac{g(x + \tau v) - g(x)}{\tau}
}
\end{array}
\end{equation}
is the Clarke directional derivative of $g$ at $\bar{x}$.

\begin{proposition} \label{pr:usc dd-cvx} \rm
\cite[Proposition~4.3.2]{CuiPang2021}
Let $g : {\cal O} \to \mathbb{R}$ be a Bouligand differentiable function
on the open convex set ${\cal O} \subseteq \mathbb{R}^n$. If $g$ is dd-convex
at $\bar{x}$, then $\wh{\partial} g$ is upper semicontinuous at $\bar{x}$
if and only if $\wh{\partial} g(\bar{x}) = \partial_{\rm C} g(\bar{x})$.  In particular,
if $g$ is strictly differentiable at $\bar{x}$, then 
$\wh{\partial} g(\bar{x}) = \partial_{\rm C} g(\bar{x})$ is a singleton and
$\wh{\partial} g$ is upper semicontinuous at $\bar{x}$.  \hfill $\Box$
\end{proposition}

With the above preparations, we consider the composite function
\[
\theta_{\rm III}(x) \, \triangleq \, \phi \circ P(x), \epc \mbox{where
$\phi : \mathbb{R}^K \to \mathbb{R}$ is dd-convex and $P$ is of type I};
\]
i.e., the vector function 
$P \triangleq ( p_k )_{k=1}^K$ is such that each $p_k$ is given by (\ref{eq:dc+diff}).
We first observe that for every fixed but arbitrary vector $c \in \mathbb{R}^K$, 
the function $\theta_{\rm I}( \bullet;c) : x \mapsto c^{\top}P(x)$ is a 
type~I composite function; in fact, it is
the sum of dc plus a finite-max-minus-max differentiable function; we can still obtain
a family of convex (dd-)surrogation functions 
$\wh{\theta}_{\rm I}(\bullet;\bar{x};c;\boldsymbol{\xi})$ parameterized by
$\boldsymbol{\xi}$ which belongs to the same family 
$\boldsymbol{\Xi}(\bar{x})$, for all vectors $c$.  With the set 
$\boldsymbol{\Xi}(\bar{x})$ given by (\ref{eq:definition of Xi}), we define
\[ \begin{array}{l}
\wh{\theta}_{\, \rm I}(x;\bar{x};c;\boldsymbol{\xi}) \, \triangleq \, 
\displaystyle{
\sum_{k \, : \, c_k > 0}
} \, c_k \, \left[ \, p_k^{\rm cvx}(x) - p_k^{\rm cvx}(\bar{x})
- ( b^{\, k} )^{\top}(x - \bar{x}) + \displaystyle{
\max_{\ell \in {\cal M}_{k;\max}^{\rm diff}(\bar{x})}
} \, \nabla p_{k \ell}^{\rm diff}(\bar{x})^{\top}(x - \bar{x}) \, \right] \, + \\ [0.3in]
\hspace{1in} \displaystyle{
\sum_{k \, : \, c_k < 0}
} \, c_k \, \left[ \, ( a^k )^{\top}(x - \bar{x}) + p_k^{\rm cve}(x) - p_k^{\rm cve}(\bar{x}) +
\nabla p_{k \ell_k}^{\rm diff}(\bar{x})^{\top}( x - \bar{x})\, \right], \epc
x \, \in \, \mathbb{R}^n \\ [ 0.3in]
\hspace{0.5in} = \, 
\displaystyle{
\sum_{k=1}^K
} \, \max\left( \, \begin{array}{l}
c_k \, \left[ \, p_k^{\rm cvx}(x) - p_k^{\rm cvx}(\bar{x})
- ( b^{\, k} )^{\top}(x - \bar{x}) + \displaystyle{
\max_{\ell \in {\cal M}_{k;\max}^{\rm diff}(\bar{x})}
} \, \nabla p_{k \ell}^{\rm diff}(\bar{x})^{\top}(x - \bar{x}) \, \right], \\ [0.3in]
c_k \, \left[ \, ( a^k )^{\top}(x - \bar{x}) + p_k^{\rm cve}(x) - p_k^{\rm cve}(\bar{x}) +
\nabla p_{k \ell_k}^{\rm diff}(\bar{x})^{\top}( x - \bar{x})\, \right]
\end{array} \right)
\end{array}
\]
Based on the latter expression, the following basic properties of the function 
$\wh{\theta}_{\, \rm I}(x;\bar{x};\bullet;\boldsymbol{\xi})$ are clear.  No
proof is needed.

\begin{proposition} \label{pr:properties of type III} \rm
The function $\wh{\theta}_{\, \rm I}(x;\bar{x};\bullet;\boldsymbol{\xi})$ is continuous
and piecewise affine with pieces of linearity being the orthants of $\mathbb{R}^K$;
i.e., on every orthant of $\mathbb{R}^K$,
$\wh{\theta}_{\, \rm I}(x;\bar{x};\bullet;\boldsymbol{\xi})$ is an affine function;
moreover, $\wh{\theta}_{\, \rm I}(\bullet;\bar{x};c;\boldsymbol{\xi})$ is a convex
function.  \hfill $\Box$
\end{proposition}

Define the surrogate function $\wh{\theta}_{\rm III}(\bullet;\bar{x};\boldsymbol{\xi})$
by
\[
\wh{\theta}_{\rm III}(x;\bar{x};\boldsymbol{\xi}) \, \triangleq \, 
\theta_{\rm III}(\bar{x}) +\displaystyle{
\max_{c \in \wh{\partial} \phi(\bar{y})}
} \, \wh{\theta}_{\rm I}(x;\bar{x};c;\boldsymbol{\xi}), \epc \mbox{where } \ 
\bar{y} \, \triangleq \, P(\bar{x}),
\]
which is a well-defined value function of a maximization problem of a continuous
piecewise affine function $\wh{\theta}_{\, \rm I}(x;\bar{x};\bullet;\boldsymbol{\xi})$ 
on the compact convex set $\wh{\partial} \phi(\bar{y})$.
The following result summarizes the properties of this surrogation function, 
extending Proposition~\ref{pr:convergence of theta dd} to the case of an 
outer dd-convex function.

\begin{proposition} \label{pr:limsup dd type III} \rm
Suppose that $\phi$ is dd-convex and all the functions $p_{k\ell}^{\rm diff}$ are 
continuously differentiable.  Then the following statements hold:

\gap

{\bf (A)} $\wh{\theta}_{\rm III}(\bullet;\bar{x};\boldsymbol{\xi})$ is a convex function
for all pairs $(x,\boldsymbol{\xi})$ with $\boldsymbol{\xi} \in \boldsymbol{\Xi}(\bar{x})$; 

\gap

{\bf (B)} if $\phi$ is additionally isotone on $\mathbb{R}^K$, then 
$\wh{\theta}_{\rm I}(x;\bar{x};\bullet;\boldsymbol{\xi})$ is a linear function
on $\wh{\partial} \phi(\bar{y})$ for all triplets $(x,\bar{x}),\boldsymbol{\xi})$
with $\boldsymbol{\xi} \in \boldsymbol{\Xi}(\bar{x})$;

\gap

{\bf (C)} the function $\wh{\theta}_{\rm III}$ 
satisfies the same six properties: touching at the reference vector, 
joint continuity at the reference vector, lower Lipchitiz boundeness, directional
derivative dominance, directional derivative consistency, and
dd-joint upper semicontinuity;

\gap

{\bf (D)} the function $\wh{\theta}_{\rm III}$ satisfies the uniform upper approximation
property at $\bar{x}^{\infty}$ provided that $\phi$ is strictly differentiable at
$\bar{y}^{\infty} \triangleq P(\bar{x}^{\infty})$ and 
${\cal M}_{k;\max}^{\, \rm diff}(\bar{x}^{\infty})$ is a singleton for all $k$ satisfying
$\displaystyle{
\frac{\partial \phi(\bar{y}^{\infty})}{\partial y_k}
} > 0$.
\end{proposition}

\begin{proof}  Statement (A) holds because 
$\wh{\theta}_{\rm III}(\bullet;\bar{x};\boldsymbol{\xi})$ is the pointwise maximum
of a family of convex functions.  For statement (B), we first prove that every element
$c$ of $\wh{\partial} \phi(\bar{y})$ is
a nonnegative vector.  Indeed, by definition, such an element satisfies
\[
\displaystyle{
\liminf_{\bar{y} \neq y \to\bar{y}}
} \; \frac{\phi(y) - \phi(\bar{y}) - c^{\top}(y - \bar{y})}{
\| \, y - \bar{y} \, \|} \, \geq \, 0.
\]
By letting $y = \bar{y} - \tau e^k$ for $\tau \downarrow 0$, where $e^k$ is the 
$k$th-unit coordinate vector in $\mathbb{R}^K$, it follows readily that 
$c_k \geq 0$ by the isotonicity of $\phi$.
Hence $v \geq 0$.  With $\wh{\partial} \phi(\bar{y}) \subseteq \mathbb{R}^K_+$, it 
follows that
\[ \begin{array}{l}
\wh{\theta}_{\, \rm I}(x;\bar{x};c;\boldsymbol{\xi}) \, = \\ [0.1in]
\displaystyle{
\sum_{k=1}^K
} \, c_k \, \left[ \, p_k^{\rm cvx}(x) + p_k^{\rm cve}(\bar{x})
- ( b^{\, k} )^{\top}(x - \bar{x}) + p_{k;\max}^{\, \rm diff}(\bar{x}) + \displaystyle{
\max_{\ell \in {\cal M}_{k;\max}^{\rm diff}(\bar{x})}
} \, \nabla p_{k \ell}^{\rm diff}(\bar{x})^{\top}(x - \bar{x}) \, \right]
\end{array} \]
which clearly shows that $\wh{\theta}_{\, \rm I}(x;\bar{x};\bullet;\boldsymbol{\xi})$	 
is a linear function. 

\gap

For statement (C), the touching property is clear.  For the direction derivative
dominance, we write down
\[
\theta_{\rm III}^{\, \prime}(\bar{x};v) \, = \, \displaystyle{
\max_{c \, \in \, \wh{\partial} \phi(\bar{y})}
} \, \displaystyle{
\sum_{k=1}^K
} \, c_k \left[ \, ( p_k^{\, \rm cvx} )^{\, \prime}(\bar{x};v) + 
( p_k^{\, \rm cve} )^{\, \prime}(\bar{x};v) + \displaystyle{
\max_{\ell \in {\cal M}_{k;\max}^{\rm diff}(\bar{x})}
} \, \nabla p_{k \ell}^{\, \rm diff}(\bar{x})^{\top} v \, \right]
\]
and compare this directional derivative per each summand with the corresponding one in
\[ \begin{array}{l}
\left[ \, \wh{\theta}_{\, \rm III}(\bullet;\bar{x};c;\boldsymbol{\xi}) 
\, \right]^{\, \prime}(\bar{x};v) \\ [0.1in]
\epc = \, \displaystyle{
\max_{c \, \in \, \wh{\partial} \phi(\bar{y})}
} \, \displaystyle{
\sum_{k=1}^K
} \, \max\left( \, \begin{array}{l}
c_k \, \left[ \, ( \, p_k^{\rm cvx} \, )^{\, \prime}(\bar{x};v) 
- ( b^{\, k} )^{\top}v + \displaystyle{
\max_{\ell \in {\cal M}_{k;\max}^{\rm diff}(\bar{x})}
} \, \nabla p_{k \ell}^{\rm diff}(\bar{x})^{\top}v \, \right], \\ [0.3in]
c_k \, \left[ \, ( a^k )^{\top}v + ( p_k^{\rm cve} )^{\, \prime}(\bar{x})^{\top}v  
+ \nabla p_{k \ell_k}^{\rm diff}(\bar{x})^{\top}v \, \right]
\end{array} \right),
\end{array} \]
we can readily prove the two assertions about the directional derivative dominance 
and consistency.  To show the Lipschitz boundedness, it suffices to note that for any
bounded subset $S$ of $\mathbb{R}^K$, the two unions $\displaystyle{
\bigcup_{x \in S}
} \, \wh{\partial} \phi(P(x))$ and $\displaystyle{
\bigcup_{x \in S}
} \, \boldsymbol{\Xi}(x)$ are bounded.

\gap

To prove the dd-joint-upper semicontinuity inequality 
(\ref{eq:limsup inequality}) for $\wh{\theta}_{\rm III}$, 
let sequences $\{ \bar{x}^{\nu} \} \to \bar{x}^{\infty}$, 
$\{ \wh{x}^{\, \nu} \} \to \wh{x}^{\, \infty}$, $\{ w^{\nu} \} \to w^{\infty}$, and 
$\{ \boldsymbol{\xi}^{\, \nu} \} \to \boldsymbol{\xi}^{\infty}$ with each tuple
$\boldsymbol{\xi}^{\, \nu} = \left( a^{\nu;k},b^{\nu;k};\ell_k^{\, \nu}\, \right)_{k=1}^K$
belonging to $\boldsymbol{\Xi}(\bar{x}^{\nu})$ for all $\nu$ be given.  To show
\begin{equation} \label{eq:limsup inequality III}
\displaystyle{
\limsup_{\nu \to \infty}
} \, \left[ \, 
\wh{\theta}_{\rm III}(\bullet;\bar{x}^{\nu},\boldsymbol{\xi}^{\, \nu}) 
\, \right]^{\, \prime}
(\wh{x}^{\, \nu};w^{\nu}) \, \leq \, 
\left[ \, \wh{\theta}_{\rm III}(\bullet;\bar{x}^{\infty};\boldsymbol{\xi}^{\, \infty})
\, \right]^{\, \prime}(\wh{x}^{\, \infty};w^{\infty}).
\end{equation}
Let $\kappa$ be an infinite subset of $\{ 1, 2, \cdots, \}$ such that
\[
\displaystyle{
\limsup_{\nu \to \infty}
} \, \left[ \, 
\wh{\theta}_{\rm III}(\bullet;\bar{x}^{\nu},\boldsymbol{\xi}^{\, \nu}) 
\, \right]^{\, \prime}
(\wh{x}^{\, \nu};w^{\nu}) \, = \, \displaystyle{
\lim_{\nu (\in \kappa) \to \infty}
} \, \left[ \, 
\wh{\theta}_{\rm III}(\bullet;\bar{x}^{\nu},\boldsymbol{\xi}^{\, \nu}) 
\, \right]^{\, \prime}
(\wh{x}^{\, \nu};w^{\nu}).
\]
For every $\nu$, there exist (a), a subset 
$C^{\, \nu} \subseteq \wh{\partial} \phi(\bar{y}^{\nu})$ and (b) a pair of disjoint
index sets ${\cal K}_{\pm}^{\, \nu}$ partitioning 
$\{ 1, \cdots, K \}$ such that for all $c \in C^{\, \nu}$, 
$c_k \geq 0$ for all $k \in {\cal K}_+^{\nu}$ and 
$c_k < 0$ for all $k \in {\cal K}_-^{\nu}$, and
\[ \begin{array}{l}
\left[ \, \wh{\theta}_{\, \rm III}(\bullet;\bar{x}^{\nu};\boldsymbol{\xi}^{\nu}) 
\, \right]^{\, \prime}(\wh{x}^{\, \nu};w^{\nu}) \\ [0.1in]
\epc = \, \displaystyle{
\max_{c \, \in \, C^{\, \nu}}
} \, \left\{ \, \begin{array}{l}	
\displaystyle{
\sum_{k \in {\cal K}_+^{\, \nu}}
} \, c_k \, \left[ \, ( \, p_k^{\rm cvx} \, )^{\, \prime}(\wh{x}^{\, \nu};w^{\nu}) 
- ( b^{\, \nu;k} )^{\top}w^{\nu} + \displaystyle{
\max_{\ell \in {\cal M}_{k;\max}^{\rm diff}(\bar{x}^{\nu})}
} \, \nabla p_{k \ell}^{\rm diff}(\bar{x}^{\nu})^{\top}w^{\nu} \, \right] \, + \\ [0.3in]
\displaystyle{
\sum_{k \in {\cal K}_-^{\, \nu}}
} \, c_k \, \left[ \, ( a^{\nu;k} )^{\top}w^{\nu} + 
( p_k^{\rm cve} )^{\, \prime}(\bar{x}^{\nu})^{\top}w^{\nu}  
+ \nabla p_{k \ell_k^{\nu}}^{\rm diff}(\bar{x}^{\nu})^{\top}w^{\nu} \, \right]
\end{array} \right).
\end{array} \]
The set $C^{\, \nu} = \displaystyle{
\operatornamewithlimits{\mbox{\bf argmax}}_{c \in \wh{\partial} \phi(\bar{y}^{\nu})}
} \, \wh{\theta}_{\rm I}(\wh{x}^{\nu};\bar{x}^{\nu};c;\boldsymbol{\xi}^{\nu})$.
There exist an infinite subset $\kappa^{\, \prime} \subseteq \kappa$ and a triplet
$( \, C^{\, \infty}, {\cal K}_{\pm}^{\, \infty} \, )$ such that 
$( \, C^{\, \nu}, {\cal K}_{\pm}^{\, \nu} \, ) = 
( \, C^{\, \infty}, {\cal K}_{\pm}^{\, \infty} \, )$ for all $\nu \in \kappa^{\, \prime}$.
By the closedness of the map $\wh{\partial} \phi$ at $\bar{y}$, we must have
$C^{\, \infty} \subseteq \wh{\partial} \phi(\wh{y}^{\, \infty})$.  Moreover,
\[
{\cal K}_+^{\infty} \, \subseteq \, \left\{ \, k \, \in \, [ K ] \, : \, 
c_k \, \geq \, 0, \ \forall \, c \, \in \, C^{\infty} \, \right\} \ \mbox{ and } \
{\cal K}_-^{\infty} \, \subseteq \, \left\{ \, k \, \in \, [ K ] \, : \, 
c_k \, < \, 0, \ \forall \, c \, \in \, C^{\infty} \, \right\}.
\]
Hence, it follows that
\[ \begin{array}{l}
\displaystyle{
\limsup_{\nu \to \infty}
} \, \left[ \, 
\wh{\theta}_{\rm III}(\bullet;\bar{x}^{\nu},\boldsymbol{\xi}^{\, \nu}) 
\, \right]^{\, \prime}
(\wh{x}^{\, \nu};w^{\nu}) \\ [0.2in]
\leq \, \displaystyle{
\max_{c \, \in \, C^{\infty}}
} \, \left\{ \, \begin{array}{l}	
\displaystyle{
\sum_{k \in {\cal K}_+^{\, \infty}}
} \, c_k \, \left[ \, ( \, p_k^{\rm cvx} \, )^{\, \prime}(\wh{x}^{\, \infty};w^{\infty}) 
- ( b^{\, \infty;k} )^{\top}w^{\infty} + \displaystyle{
\max_{\ell \in {\cal M}_{k;\max}^{\rm diff}(\bar{x}^{\infty})}
} \, \nabla p_{k \ell}^{\rm diff}(\bar{x}^{\infty})^{\top}w^{\infty} \, \right] 
\, + \\ [0.3in]
\displaystyle{
\sum_{k \in {\cal K}_-^{\, \infty}}
} \, c_k \, \left[ \, ( a^{\infty;k} )^{\top}w^{\infty} + 
( p_k^{\rm cve} )^{\, \prime}(\bar{x}^{\infty})^{\top}w^{\infty}  
+ \nabla p_{k \ell_k}^{\rm diff}(\bar{x}^{\infty})^{\top}w^{\infty} \, \right]
\end{array} \right) \\ [0.5in]
\leq \, \displaystyle{
\max_{c \, \in \, \wh{\partial} \phi(\bar{y}^{\infty})}
} \, \left\{ \, \begin{array}{l}	
\displaystyle{
\sum_{k : c_k > 0}
} \, c_k \, \left[ \, ( \, p_k^{\rm cvx} \, )^{\, \prime}(\wh{x}^{\, \infty};w^{\infty}) 
- ( b^{\, \infty;k} )^{\top}w^{\infty} + \displaystyle{
\max_{\ell \in {\cal M}_{k;\max}^{\rm diff}(\bar{x}^{\infty})}
} \, \nabla p_{k \ell}^{\rm diff}(\bar{x}^{\infty})^{\top}w^{\infty} \, \right] \, + \\ [0.3in]
\displaystyle{
\sum_{k : c_k < 0}
} \, c_k \, \left[ \, ( a^{\infty;k} )^{\top}w^{\infty} + 
( p_k^{\rm cve} )^{\, \prime}(\bar{x}^{\infty})^{\top}w^{\infty}  
+ \nabla p_{k \ell_k}^{\rm diff}(\bar{x}^{\infty})^{\top}w^{\infty} \, \right]
\end{array} \right),
\end{array} \]
which is the desired inequality (\ref{eq:limsup inequality III}).  Finally, the
proof of statement (D) is the same as that for the type~I composition under the strict
differentiability of $\phi$ at $\bar{y}^{\infty}$.
\end{proof}

{\bf Some examples:}  Several of the type I and type II composite functions are dd-convex;
for instance, the product of two plus functions such as $\max(f(x),0) \max(g(x),0)$
where $f$ and $g$ are convex is dd-convex.  A class of the dd-convex functions that
are not necessarily a type I/II composite function consists of 
the bivariate locally Lipschitz, convex-convex functions
$\psi(x,y)$, where $\psi(x,\bullet)$ is convex on $\mathbb{R}^m$ and $\psi(\bullet,y)$
is convex on $\mathbb{R}^n$ and one of the latter two functions is differentiable.
For such a bivariate function, we have the sum 
equality \cite[Proposition~4.4.27]{CuiPang2021}:
\[
\psi^{\, \prime}((\bar{x},\bar{y});(u,v)) \, = \,  
( \psi(\bullet,\bar{y}) )^{\, \prime}(\bar{x};u) +
( \psi(\bar{x},\bullet) )^{\, \prime}(\bar{y};v) 
\]
which shows that $\psi^{\, \prime}((\bar{x},\bar{y});(\bullet,\bullet))$ is jointly
convex in the direction arguments. 

\subsubsection{Type IV compositions: concave+(dd-convex) outer functions}

With the type III composition in the background, it suffices to consider the class
\[
\theta_{\rm IV}(x) \, \triangleq \, \phi \circ P(x), \epc \mbox{where
$\phi : \mathbb{R}^K \to \mathbb{R}$ is concave and $P$ is of type I}.
\]
We have, with $\bar{y} \triangleq P(\bar{x})$, 
\[
\theta_{\rm IV}^{\, \prime}(\bar{x};v) \, = \, 
\phi^{\, \prime}(\bar{y};P^{\, \prime}(\bar{x};v)) \, = \, \displaystyle{
\min_{-a \in \partial (-\phi)(\bar{y})}
} \, a^{\top}P^{\, \prime}(x;v) \, \leq \, a^{\top}P^{\, \prime}(\bar{x};v), \epc
\forall \, a \, \in \, -\partial (-\phi)(\bar{y}).
\]
Based on the same observation that the function $x \mapsto c^{\top}P(x)$ 
is a type~I composite function for every fixed but arbitrary vector $c \in \mathbb{R}^K$,
we obtain a family of doubly parameterized surrogation functions 
\[
\left\{ \, \wh{\theta}_{\rm IV}(x;\bar{x};a;\boldsymbol{\xi}) \, : \, 
( \, a;\boldsymbol{\xi} \, ) \, \in \, \underbrace{\partial (-\phi)(\bar{y}) \, \times \, 
\boldsymbol{\Xi}(\bar{x})}_{\mbox{denoted $\wh{\boldsymbol{\Xi}}(\bar{x})$}} \, \right\}
\]
which share similar properties to those for the former three types of composite functions, 
except that the parameter family $\wh{\boldsymbol{\Xi}}(\bar{x})$ has another parameter
set $\partial (-\phi)(\bar{x})$ due to the concave outer function $\phi$.
Thus when $\phi$ is a dc function, the resulting family $\wh{\boldsymbol{\Xi}}(\bar{x})$
will also involve the outer function $\phi$ when its concave component is nondifferentiable.

\gap

{\bf Some examples:}. As already mentioned before, difference-of-convex functions 
abound in applications; in particular, continuous piecewise smooth functions defined 
on open convex domains
are dc, and so are many basic statistical functionals \cite{NouiehedPangRaza18};
in particular, a truncated squared function, which has interesting application
in composite energy minimization in imaging \cite{GeipingMoeller18}, is used
as a case study for the algorithm proposed in \cite{LeThiHuynhPham24} 
for solving dc composite differentiable optimization problems.
Shown to be dc functions \cite{AhnPangXin17,LeThiPhamVo15}  the class of
{\sl folded concave functions} \cite{FanLi01,FanXueZou14} is used extensively
to approximate the $\ell_0$-function in sparsity optimization.  More recently,
adding to these fairly well-known applications,
dc functions are employed to approximate Heaviside functions, which are the indicator
functions of (open or closed) intervals; see \cite{CuiLiuPang2023}.  The related
reference \cite{CuiLiuPang2022} discusses how such approximations can be used for
solving chance-constrained stochastic programs.  Defined as the
difference of two conditional-values-at-risk, 
the {\sl interval conditional-value-at-risk} \cite{LiuPang2023} is used as a 
risk measure in statistical learning as a tool to exclude outliers in data sets.  
All these dc functions can be composed with some base functionals that themselves 
may be of the difference-finite-max kind, resulting in the class of dc composite
dc-plus-difference-finite-max optimization problems to be solved.  The design
of effective algorithms for solving this class of optimization problems was 
left as an open question in the reference \cite{CuiPang2021}, which is to a large
extent covered by our present unifying framework of quasi-dc programming.

\subsection{A summary: Surrogation of the overall objective function $\Theta_{\max}$} 
\label{subsec:summary of types} 

We return to the function $\Theta_{\max}$ defined as the pointwise maximum
of the functions $\theta_j$ for $j = 1, \cdots, J$, each of which is one
of the four kinds of composite functions discussed above.  Associated with $\theta_j$
is a parameterized family of convex functions 
$\left\{ \, \wh{\theta}_j(\bullet;\bar{x};\boldsymbol{\xi}) \, : \, 
\boldsymbol{\xi} \, \in \, \wh{\boldsymbol{\Xi}}^{\, j}(\bar{x}) \right\}$
at any given vector $\bar{x} \in X$.  We recap the properties of this family
of functions as follows: for each $j = 1, \cdots, J$,

\gap

$\bullet $ {\bf (the parameterizing family)} 
the set-valued map $\wh{\boldsymbol{\Xi}}^{\, j}$
is nonempty-valued, compact-valued, and upper semicontinuous;

\gap 

%

$\bullet $ {\bf (touching at the reference point)} 
$\wh{\theta}_j(\bar{x};\bar{x};\boldsymbol{\xi}) \equiv \theta_j(\bar{x})$ for all
$\boldsymbol{\xi} \in \wh{\boldsymbol{\Xi}}^{\, j}(\bar{x})$;

\gap

$\bullet $ ({\bf joint continuity at the reference vector}) for all $x^{\infty}$ and 
all sequences $\{ ( \wh{x}^{\, \nu},\bar{x}^{\nu} ) \} \to 
( \bar{x}^{\infty},\bar{x}^{\infty} )$, and $\{ \boldsymbol{\xi}^{\, \nu} \}$
such that $\boldsymbol{\xi}^{\, \nu} \in \wh{\boldsymbol{\Xi}}^{\, j}(x^{\nu})$ 
for every $\nu$,
$\displaystyle{
\lim_{\nu \to \infty}
} \, \wh{\theta}_j(\wh{x}^{\, \nu};\bar{x}^{\nu};\boldsymbol{\xi}^{\, \nu}) 
= \theta_j(\bar{x}^{\infty})$;

\gap

$\bullet $ {\bf (lower Lipschitz boundedness)} for every bounded subset $S$, 
there exists a scalar $B_j > 0$ such that 
\[
\wh{\theta}_j(x;\bar{x};\boldsymbol{\xi}) -
\wh{\theta}_j(\bar{x};\bar{x};\boldsymbol{\xi})  
\, \geq \, -B_j\, \| \, x - \bar{x} \, \|_2, \epc \forall \, ( \, x,\bar{x} \, )  \, \in \, 
X \, \times \, S \ \mbox{ and all } \ \boldsymbol{\xi} \, \in \, 
\wh{\boldsymbol{\Xi}}^{\, j}(\bar{x}); 
\]
$\bullet $ {\bf (directional derivative dominance)} for all 
$\boldsymbol{\xi} \in \wh{\boldsymbol{\Xi}}^{\, j}(\bar{x})$ 
$\left[ \, \wh{\theta}_j(\bullet;\bar{x};\boldsymbol{\xi}) 
\, \right]^{\, \prime}(\bar{x};v) 
\, \geq \, \theta_j^{\, \prime}(\bar{x};v)$ for all $v \in \mathbb{R}^n$; 
moreover, for every $\bar{v}$, a tuple
$\bar{\boldsymbol{\xi}} \in \wh{\boldsymbol{\Xi}}^{\, j}(\bar{x})$  exists
such that $\left[ \, \wh{\theta}_j(\bullet;\bar{x};\bar{\boldsymbol{\xi}}) 
\, \right]^{\, \prime}(\bar{x};\bar{v}) 
= \theta_j^{\, \prime}(\bar{x};\bar{v})$;

\gap

$\bullet $ {\bf (directional derivative consistency)} if 
$\wh{\boldsymbol{\Xi}}^{\, j}(\bar{x}) \equiv \left\{ \boldsymbol{\xi} \right\}$ 
is a singleton, then we have the equality
$\left[ \, \wh{\theta}_j(\bullet;\bar{x};\boldsymbol{\xi}) 
\, \right]^{\, \prime}(\bar{x};v) 
\, = \, \theta^{\, \prime}(\bar{x};v)$ for all $v$;

\gap
 
$\bullet $ {\bf (dd-joint upper semicontinity)} for all sequences
$\{ \bar{x}^{\nu} \} \to \bar{x}^{\infty}$, 
$\{ \wh{x}^{\, \nu} \} \to \wh{x}^{\, \infty}$, $\{ w^{\nu} \} \to w^{\infty}$, and 
$\{ \boldsymbol{\xi}^{\, \nu} \} \to \boldsymbol{\xi}^{\infty}$ with each tuple
$\boldsymbol{\xi}^{\, \nu}$ 
belonging to $\wh{\boldsymbol{\Xi}}^{\, j}(\bar{x}^{\nu})$ for all $\nu$,
\begin{equation} \label{eq:limsup inequality for thetaj}
\displaystyle{
\limsup_{\nu \to \infty}
} \, \left[ \, 
\wh{\theta}_j(\bullet;\bar{x}^{\nu},\boldsymbol{\xi}^{\, \nu}) 
\, \right]^{\, \prime}
(\wh{x}^{\, \nu};w^{\nu}) \, \leq \, 
\left[ \, \wh{\theta}_j(\bullet;\bar{x}^{\infty};\boldsymbol{\xi}^{\, \infty})
\, \right]^{\, \prime}(\wh{x}^{\, \infty};w^{\infty}).
\end{equation}
For an arbitrary scalar $\varepsilon \geq 0$, define the family 
\begin{equation} \label{eq:the family F}
\boldsymbol{\mathcal F}_{\Theta}^{\, \varepsilon}(\bar{x}) \, \triangleq \, \left\{ \, 
\wh{\Theta}_{\max}^{\, \varepsilon}(\bullet;\bar{x};\wh{\boldsymbol{\xi}})
\, \triangleq \, \displaystyle{
\max_{j \in {\cal M}_{\Theta}^{\, \varepsilon}(\bar{x})}
} \, \wh{\theta}_j(\bullet;\bar{x};\boldsymbol{\xi}^{\, j}) \, : \, 
\wh{\boldsymbol{\xi}} \, \triangleq \, \left( \boldsymbol{\xi}^{\, j} \right)_{j=1}^J 
\, \in \, \displaystyle{
\prod_{j=1}^J
} \ \wh{\boldsymbol{\Xi}}^{\, j}(\bar{x}) \, \triangleq \, 
\wh{\boldsymbol{\Xi}}_{\Theta}(\bar{x}) \, \right\}.
\end{equation}
We note that for any $\wh{\boldsymbol{\xi}} \in \wh{\boldsymbol{\Xi}}_{\Theta}(\bar{x})$,
\begin{equation} \label{eq:dd equality at xi}
\left[ \, \wh{\Theta}_{\max}^{\, \varepsilon}(\bullet;\bar{x};\wh{\boldsymbol{\xi}})
\, \right]^{\, \prime}(\bar{x};v) \, = \, \displaystyle{
\max_{j \in {\cal M}_{\Theta}(\bar{x})}
} \, \left[ \, \wh{\theta}_j(\bullet;\bar{x};\boldsymbol{\xi}^{\, j})
\, \right]^{\, \prime}(\bar{x};v), \epc \forall \, v \, \in \, \mathbb{R}^n
\end{equation}
because $\displaystyle{
\operatornamewithlimits{\mbox{argmax}}_{j \in {\cal M}_{\Theta}^{\, \varepsilon}(\bar{x})}
} \ \wh{\theta}_j(\bar{x};\bar{x};\boldsymbol{\xi}^{\, j}) =  {\cal M}_{\Theta}(\bar{x})$.
Associated with each function in the family 
$\boldsymbol{\mathcal F}_{\Theta}^{\, \varepsilon}(\bar{x})$, we define the strongly convex
program: for an arbitrary scalar $\rho > 0$,
\begin{equation} \label{eq:workhorse}
\displaystyle{
\operatornamewithlimits{\mbox{\bf minimize}}_{x \in X}
} \ \wh{\Theta}_{\max}^{\, \varepsilon}(x;\bar{x};\wh{\boldsymbol{\xi}}) + \displaystyle{
\frac{\rho}{2}
} \, \| \, x - \bar{x} \, \|_2^2,
\end{equation}
which will be the computational workhorse of the algorithm to be described for solving 
the problem (\ref{eq:unified framework}).  Note that (\ref{eq:workhorse}) is a strongly
convex program and thus has a unique globally optimal solution.
An enhanced computational scheme may involve a double minimization:
\begin{equation} \label{eq:enhanced workhorse}
\displaystyle{
\operatornamewithlimits{\mbox{\bf minimize}}_{
\wh{\boldsymbol{\xi}} \in \wh{\boldsymbol{\Xi}}_{\Theta}(\bar{x})}
} \, \left[ \, \displaystyle{
\operatornamewithlimits{\mbox{\bf minimum}}_{x \in X}
} \ \wh{\Theta}_{\max}^{\, \varepsilon}(x;\bar{x};\wh{\boldsymbol{\xi}}) + \displaystyle{
\frac{\rho}{2}
} \, \| \, x - \bar{x} \, \|_2^2 \, \right],
\end{equation}
that is computationally implementable  when 
$\wh{\boldsymbol{\Xi}}_{\Theta}(\bar{x})$ is a finite set
(see problem (\ref{eq:max-max all subproblem}) in 
Subsection~\ref{subsec:computing directional}), albeit could be demanding.  
Note the use of a possibly positive $\varepsilon$ in defining the
functions in the family $\boldsymbol{\mathcal F}_{\Theta}^{\, \varepsilon}(\bar{x})$.  
On one hand, this offers 
the benefit of providing control on the number of maximands
to the minimized in (\ref{eq:workhorse}) as a surrogate for the 
problem (\ref{eq:unified framework}).   This flexibility is certainly 
a computational advantage in practice over the choice of the full index set 
$[ J ] \triangleq \{ 1, \cdots, J \}$.  On the other hand, when 
${\cal M}_{\Theta}^{\varepsilon}(\bar{x})$ is a proper subset of $[ J ]$, 
the direct decrease of the objective function $\Theta_{\max}$ from a current
iterate $\bar{x}$ is jeopardized 
by the omission of the functions $\theta_j$ for 
$j \not\in {\cal M}_{\Theta}^{\varepsilon}(\bar{x})$ when we
obtain a new iterate $\wh{x}$ by
minimizing $\wh{\Theta}_{\max}^{\, \varepsilon}(\bullet;\bar{x};\wh{\boldsymbol{\xi}})$;
see Subsection~\ref{subsec:descent analysis}.  It happens that as far as stationarity
is concerned, any nonnegative $\varepsilon$ achieves the same goal.

\begin{proposition} \label{pr:unified stationarity certificate} \rm
Let $X$ be a closed convex set in $\mathbb{R}^n$.  Let $\bar{x} \in X$ be arbitrary
and let $\boldsymbol{\mathcal F}_{\Theta}^{\, \varepsilon}(\bar{x})$ be the family of functions
defined above.  Let $\rho > 0$ be arbitrary.  The following statements hold.

\gap 

{\bf (A)} For any $\varepsilon \geq 0$, the following two statements are equivalent:

\gap

(Ai) $\bar{x}$ is an optimal solution of (\ref{eq:workhorse}) for all 
$\wh{\boldsymbol{\xi}} \in \wh{\boldsymbol{\Xi}}_{\Theta}(\bar{x})$;

\gap

(Aii) $\bar{x}$ is an optimal solution of 
\[
\displaystyle{
\operatornamewithlimits{\mbox{\bf minimize}}_{x \in X}
} \, \left\{ \, \left[ \, \displaystyle{
\operatornamewithlimits{\mbox{\bf minimum}}_{
\wh{\boldsymbol{\xi}} \in \wh{\boldsymbol{\Xi}}_{\Theta}(\bar{x})} 
} \ \wh{\Theta}_{\max}^{\, \varepsilon}(x;\bar{x};\wh{\boldsymbol{\xi}}) 
\, \right] + \displaystyle{
\frac{\rho}{2}
} \, \| \, x - \bar{x} \, \|_2^2 \, \right\}
\]
which is equivalent to (\ref{eq:enhanced workhorse}).

\gap

{\bf (B)} If $\bar{x}$ is a directional stationary solution 
of $\Theta_{\max}$ on $X$, then both statements (Ai) and (Aii) hold 
for all $\varepsilon \geq 0$; conversely, if either statement (Ai) or (Aii) holds 
for some $\varepsilon \geq 0$, then $\bar{x}$ is a directional stationary solution 
of $\Theta_{\max}$ on $X$.
\end{proposition}

\begin{proof}  (Ai) $\Leftrightarrow$ (Aii).  Since
$\wh{\Theta}_{\max}^{\, \varepsilon}(\bar{x};\bar{x};\wh{\boldsymbol{\xi}}) =
\Theta_{\max}(\bar{x})$ for all 
$\wh{\boldsymbol{\xi}} \in \wh{\boldsymbol{\Xi}}_{\Theta}(\bar{x})$
and all $\varepsilon \geq 0$, it is clear that (A), (B), and (\ref{eq:enhanced workhorse}) 
are all saying the same thing, namely,
\[
\Theta_{\max}(\bar{x}) \, \leq \, 
\wh{\Theta}_{\max}^{\, \varepsilon}(x;\bar{x};\wh{\boldsymbol{\xi}}) + \displaystyle{
\frac{\rho}{2}
} \, \| \, x - \bar{x} \, \|_2^2, \epc \forall \, ( x,\wh{\boldsymbol{\xi}} ) \, \in \, 
X \, \times \, \wh{\boldsymbol{\Xi}}_{\Theta}(\bar{x}).
\]
(B)  Suppose 
$\Theta_{\max}^{\, \prime}(\bar{x};x - \bar{x}) \geq 0$ for all $x \in X$.  Then
\[ \begin{array}{lll}
0 & \leq & \displaystyle{
\max_{j \in {\cal M}_{\Theta}(\bar{x})}
} \, \theta_j^{\, \prime}(\bar{x};x - \bar{x}) \epc \mbox{for all $x \in X$} \\ [0.15in]
& \leq & \displaystyle{
\max_{j \in {\cal M}_{\Theta}(\bar{x})}
} \, \left\{ \, \left[ \, \wh{\theta}_j(\bullet;\bar{x};\boldsymbol{\xi}^{\, j}) 
\, \right]^{\, \prime}(\bar{x};x - \bar{x}) \, \right\} \\ [0.2in]
& & \epc \forall \,  
\boldsymbol{\xi}^{\, j} \in \wh{\boldsymbol{\Xi}}^{\, j}(\bar{x}) \ 
\mbox{by directional derivative dominance} \\ [0.15in]
& = & \left[ \, \wh{\Theta}_{\max}^{\, \varepsilon}(\bullet;\bar{x};\wh{\boldsymbol{\xi}}) 
\, \right]^{\, \prime}(\bar{x};x - \bar{x}) \epc \mbox{for all 
$\wh{\boldsymbol{\xi}} \in \wh{\boldsymbol{\Xi}}_{\Theta}(\bar{x})$ and all $x \in X$, by
(\ref{eq:dd equality at xi})},
\end{array} \]
where the last equality holds by the functional touching property:
$\wh{\theta}_j(\bar{x};\bar{x};\boldsymbol{\xi}^j) = \theta_j(\bar{x})$
for all $j$.  Thus, $\bar{x}$ globally minimizes 
$\wh{\Theta}_{\max}^{\, \varepsilon}(\bullet;\bar{x};\wh{\boldsymbol{\xi}})$ on $X$
because the function is convex.  

\gap

Conversely, suppose that for some $\varepsilon \geq 0$,
$\bar{x}$ is an optimal solution of (\ref{eq:workhorse}) 
for all $\wh{\boldsymbol{\xi}} \in \wh{\boldsymbol{\Xi}}_{\Theta}(\bar{x})$.  To
show the desired directional stationarity of $\wh{x}$, i.e., 
$\Theta_{\max}^{\, \prime}(\bar{x};x - \bar{x}) \geq 0$
for all $x \in X$, let $x \in X$ be arbitrary.  
By the second part of the directional derivative dominance condition, there exists 
$\wh{\boldsymbol{\xi}} = \left( \, \boldsymbol{\xi}^{\, j} \right)_{j=1}^J 
\in \wh{\boldsymbol{\Xi}}_{\Theta}(\bar{x})$ such that
\[
\left[ \, \wh{\theta}_j(\bullet;\bar{x};\boldsymbol{\xi}^{\, j}) \, \right]^{\, \prime}(
\bar{x};x - \bar{x}) \, = \, \theta_j^{\, \prime}(\bar{x};x - \bar{x}), \epc \forall \, j.
\]
With this particular $\wh{\boldsymbol{\xi}}$, we have, by (\ref{eq:dd equality at xi}),
\[ \begin{array}{lll}
\Theta_{\max}^{\, \prime}(\bar{x};x - \bar{x}) & = & \displaystyle{
\max_{j \in {\cal M}_{\Theta}(\bar{x})}
} \, \theta_j^{\, \prime}(\bar{x}; x - \bar{x}) \\ [0.1in]
& = & \displaystyle{
\max_{j \in {\cal M}_{\Theta}(\bar{x})}
} \, \left\{ \, \left[ \, \wh{\theta}_j(\bullet;\bar{x};\boldsymbol{\xi}^{\, j}) 
\, \right]^{\, \prime}(\bar{x};x - \bar{x}) \, \right\} \\ [0.2in]
& = & \left[ \, \wh{\Theta}_{\max}^{\, \varepsilon}(\bullet;\bar{x};\wh{\boldsymbol{\xi}}) 
\, \right]^{\, \prime}(\bar{x};x - \bar{x})
\, \geq \, 0,
\end{array}
\]
where the last inequality holds by assumption on $\bar{x}$. 
\end{proof}

It is important to point out the ``{\sl for all } 
$\wh{\boldsymbol{\xi}} \in \wh{\boldsymbol{\Xi}}(\bar{x})$'' requirement, 
as opposed to ``for some'', in the converse statement of part (B).  This distinction 
is well addressed in the case of
dc programming that distinguishes a directional stationary solution under the
``for all'' requirement from a ``critical solution'' under the ``for some''
requirement.  Since a dc program is a simple special case of our setting,
the distinction persists.  To highlight their difference and to change the
terminology accordingly, we formally introduce the following.

\begin{definition} \label{df:weak stationarity} \rm
The vector $\bar{x}$ is a {\sl weak directional stationary 
solution} of (\ref{eq:unified framework}) relative to the family 
$\wh{\boldsymbol{\Xi}}_{\Theta}(\bar{x})$ if {\sl there exists}
$\wh{\boldsymbol{\xi}} \in \wh{\boldsymbol{\Xi}}_{\Theta}(\bar{x})$ such that
$\bar{x}$ is an optimal solution of the strongly convex program (\ref{eq:workhorse})
defined by this tuple $\wh{\boldsymbol{\xi}}$ and $\bar{x}$ itself, i.e.,
\begin{equation} \label{eq:fixed point stationarity}
\{ \bar{x} \} \, = \, \displaystyle{
\operatornamewithlimits{\mbox{\bf argmin}}_{x \in X}
} \, \left[ \, \displaystyle{
\max_{j \in {\cal M}_{\Theta}(\bar{x})}
} \, \wh{\theta}_j(x;\bar{x};\boldsymbol{\xi}^{\, j}) + \displaystyle{
\frac{\rho}{2}
} \, \| \, x - \bar{x} \, \|_2^2\, \right].
\end{equation}
This is in contrast to a directional stationary solution for which 
(\ref{eq:fixed point stationarity}) holds for all 
$\wh{\boldsymbol{\xi}} \in \wh{\boldsymbol{\Xi}}_{\Theta}(\bar{x})$.  \hfill $\Box$
\end{definition}

The above definition formally identifies the difference between the two kinds of stationarity
solutions and connects them well in terms of the parameterized family 
$\wh{\boldsymbol{\Xi}}_{\Theta}(\bar{x})$.

%
%

\section{Iterative Descent Algorithms} \label{sec:iterative algorithms}

The development in this section is 
based on the family $\boldsymbol{\mathcal F}_{\Theta}^{\, \varepsilon}(x^{\nu})$ at
various iterates $x^{\nu}$.  The focus here is on the presentation of 
two versions of an algorithm for computing a
stationary solution of the problem (\ref{eq:unified framework}) and establishing
their subsequential convergence.  One version solves
only one subproblem (\ref{eq:workhorse}) for an arbitrarily chosen member
$\wh{\boldsymbol{\xi}}$ from the family $\wh{\boldsymbol{\Xi}}_{\Theta}(x^{\nu})$; whereas
the other version solves such subproblems for all members, assuming that
$\wh{\boldsymbol{\Xi}}_{\Theta}(x^{\nu})$ is a finite family.  Needless to say, the latter
version is computationally much more demanding; its benefit is that the iterates
will provably accumulate to a directional stationary solution; this is in contrast to the
former version which computes only a weak directional stationary solution.  In the
description of the algorithms, we leave open how the subproblems 
(\ref{eq:workhorse}) are actually solved; more importantly, our analysis assumes
that the iterates are their exact optimal solutions.  Thus if the subproblems
are solved by gradient descent methods (for instance), only inexact solutions 
can be obtained in practical computations.  
In principle, our algorithms should account for this possibility of the inexact solution
of the subproblems; we believe that this is possible for the subsequential
convergence analysis, but to avoid complex notations and more involved analysis, we focus only
on the exact solution of the subproblems in our description and analysis
of the algorithms.  

\gap

\begin{algorithm}
\caption{A descent algorithm for (\ref{eq:unified framework})}
\label{alg I:descent}
\begin{algorithmic}
\State{\bf Initialization:}  Let $\rho$, $\sigma$ and
$\beta$ be positive scalars with $\sigma$ and $\beta$ both less than 1; these are
fixed throughout the following iterations.  Let $\{ \varepsilon_{\nu} \}$ be a
sequence of nonnegative scalars converging to the limit $\varepsilon_{\infty} \geq 0$.
Let $x^0 \in X$ be given; let $\nu = 0$.
\For{$\nu=0, 1, \cdots$,} 
\begin{enumerate}
\item  Pick a tuple $\wh{\boldsymbol{\xi}}^{\, \nu} \in \wh{\boldsymbol{\Xi}}(x^{\nu})$
       and compute the globally optimal solution of (\ref{eq:workhorse})
       with this tuple $\wh{\boldsymbol{\xi}}^{\, \nu}$ at $\bar{x} = x^{\nu}$ and
       $\varepsilon = \varepsilon_{\nu}$. Denote the computed solution by
       $x^{\nu + 1/2}$.
\item  If $x^{\nu + 1/2} = x^{\nu}$, stop; $x^{\nu}$ is a weak directional stationary
       solution of (\ref{eq:unified framework}). Otherwise,  let 
       $d^{\, \nu+1} \triangleq x^{\nu+1/2} - x^{\nu}$ and proceed to the next step.     
\item (Armijo line search) Let  
      $m_{\nu}$ be the smallest integer $m \geq 0$ such that \begin{equation} \label{eq:Armijo descent}
\Theta_{\max}(x^{\nu} + \beta^m d^{\, \nu + 1}) - \Theta_{\max}(x^{\nu})
\, \leq \, -\displaystyle{
\frac{\sigma \, \rho}{2}
} \, \, \beta^m \, \| \, x^{\nu + 1/2} - x^{\nu} \, \|_2^2.
\end{equation}
\item Let $x^{\nu+1} = x^{\nu} + \tau_{\nu+1} d^{\, \nu+1}$ where $\tau_{\nu+1} \triangleq
      \beta^{\, m_{\nu}}$.  Let $\nu \leftarrow \nu + 1$ and return to Step~1.
\end{enumerate}
\EndFor	 \hfill $\Box$
\end{algorithmic}
\end{algorithm}

The use of a sequence of (nonnegative) scalars $\{ \varepsilon_{\nu} \}$ converging 
to some $\varepsilon_{\infty} \geq 0$ adds flexibility and stability to the algorithm.  
It turns out that in principle whether the limit $\varepsilon_{\infty}$ is zero or 
not does not affect the stationarity of the limit point of the sequence of 
iterates $\{ x^{\nu} \}$.
Nevertheless, small $\varepsilon_{\nu}$'s offer computational savings; larger
$\varepsilon_{\nu}$'s provide numerical stability in the selection of the families 
of functions $\wh{\theta}_j(\bullet;x^{\nu};\boldsymbol{\xi}^j)$ to be optimized 
in the subproblems.  The appropriate choice of $\varepsilon_{\nu}$ in practice has to be
learnt from computational experience in order to balance excessive computational efforts 
and potential numerical issues.


\subsection{Direct descent in special instances} \label{subsec:descent analysis}

The first thing to do in analyzing the convergence of the algorithm is to justify the 
well-definedness of the sequence $\{ x^{\nu+1} \}$; this amounts to showing that  
each integer $m_{\nu}$
in the Armijo line search can be determined in finitely many trials.  A special instance
of this situation is for $m_{\nu} = 0$, i.e., for the Armijo test to yield a successful
step at the first trial with appropriate choice of the parameter $\sigma$ that may
depend on the proximal parameter $\rho$ and some model constants. We discuss this 
instance in the context of the  following single-ratio
(i.e., $K = 1$) fractional program:
\begin{equation} \label{eq:K=1}
\displaystyle{
\operatornamewithlimits{\mbox{\bf minimize}}_{x \in X}
} \, \Theta_{\max}^{\, K=1}(x) \, \triangleq \, 
\displaystyle{
\max_{1 \leq j \leq J}
} \, \displaystyle{
\frac{n_j(x)}{d_j(x)}
} \, \equiv \, \displaystyle{
\frac{n^{\rm ndff}_j(x) + n^{\rm diff}_j(x)}{d^{\, \rm ndff}_j(x) 
+ d^{\, \rm diff}_j(x))}
},
\end{equation}
where $n_j^{\rm diff}$ and $d_j^{\, \rm diff}$ are differentiable functions
and $n_j^{\rm ndff}$ and $d_j^{\, \rm ndff}$ are nondifferentiable functions
with properties to be specified below.  For a given $\bar{x} \in X$, we 
consider the subproblem:
\begin{equation} \label{eq:subproblem for max of ratio}
\displaystyle{
\operatornamewithlimits{\mbox{\bf minimize}}_{x \in X}
} \ \left\{ \begin{array}{l}
\displaystyle{
\max_{1 \leq j \leq J}
} \ \left[ \begin{array}{l}
\displaystyle{
\frac{1}{d_j(\bar{x})}
} \, \left( \, n^{\rm ndff}_j(x) - n^{\rm ndff}_j(\bar{x}) + 
\nabla n_j^{\rm diff}(\bar{x})^{\top}( x - \bar{x} ) \, \right) \, - \\ [0.2in]
\displaystyle{
\frac{n_j(\bar{x})}{d_j(\bar{x})^2}
} \, \left( \, d^{\, \rm ndff}_j( \wh{x} ) - d^{\, \rm ndff}_j(\bar{x})  
+ \nabla d_j^{\, \rm diff}(\bar{x}^{\top}( x - \bar{x} ) \, \right)
\end{array} \, \right] \, + \\ [0.4in]
\hspace{0.1in} \displaystyle{
\frac{\rho}{2}
} \, \| \, x - \bar{x} \, \|_2^2 
\end{array} \right\},
\end{equation}
which is a strongly convex program under the following specifications
wherein direct descent of the objective function $\Theta_{\max}$ is possible.

\gap

\noindent $\bullet $ {\bf Convex-over-concave quotient:}  This case 
is essentially the Dinkelbach single-fraction problem 
extended to the case of
a generalized fractional program.  The specifications of this case are
formally as follows: for all $j = 1, \cdots, J$, the functions
$n^{\rm diff}_j = d^{\, \rm diff}_j\equiv 0$, $n^{\rm ndff}_j$, written as
$n^{\rm cvx}_j$, is convex and nonnegative and $d^{\, \rm ndff}_j$, written as
$d^{\, \rm cve}_j$, is concave and positive.  The optimal solution $\wh{x}$ of
the problem (\ref{eq:subproblem for max of ratio}) satisfies
\begin{equation} \label{eq:optimality of whx}
\displaystyle{
\max_{1 \leq j \leq J}
} \ \displaystyle{
\frac{1}{d^{\, \rm cve}_j(\bar{x})}
} \, \left[ \, n^{\rm cvx}_j( \wh{x} ) - n^{\rm cvx}_j(\bar{x}) - \displaystyle{
\frac{n^{\rm cvx}_j(\bar{x})}{d^{\, \rm cve}_j(\bar{x})}
} \, ( d^{\, \rm cve}_j( \wh{x} ) - d^{\, \rm cve}_j(\bar{x}) ) \, \right] 
+ \displaystyle{
\frac{\rho}{2}
} \, \| \, \wh{x} - \bar{x} \, \|_2^2 \, \leq \, 0,
\end{equation}
which easily yields 
\[ \begin{array}{lll}
\Theta_{\max}^{\rm cvx/cve}( \wh{x} ) & \triangleq & 
\displaystyle{
\max_{1 \leq j \leq J}
} \, \displaystyle{
\frac{n_j^{\rm cvx}( \wh{x} )}{d_j^{\, \rm cve}( \wh{x} )}
} \\ [0.2in]
& \leq & \displaystyle{
\max_{1 \leq j \leq J}
} \, \displaystyle{
\frac{n_j^{\rm cvx}(\bar{x})}{d_j^{\, \rm cve}(\bar{x})} 
} - \displaystyle{
\frac{\rho}{2}
} \, \| \, \wh{x} - \bar{x} \, \|_2^2 \, \left[ \, \displaystyle{
\min_{1 \leq j \leq J}
} \, \displaystyle{
\frac{d_j^{\, \rm cve}(\bar{x})}{d_j^{\, \rm cve}(\wh{x})} 
} \, \right] \\ [0.2in]
& = & \Theta_{\max}^{\, \rm cvx/cve}( \bar{x} ) - \displaystyle{
\frac{\rho}{2}
} \, \| \, \wh{x} - \bar{x} \, \|_2^2 \, \left[ \, \displaystyle{
\min_{1 \leq j \leq J}
} \, \displaystyle{
\frac{d_j^{\, \rm cve}(\bar{x})}{d_j^{\, \rm cve}(\wh{x})} 
} \, \right] ,
\end{array} \]
where the inequality is derived from (\ref{eq:optimality of whx})
upon multiplication by $d^{\, \rm cve}_j(\bar{x})$ and division by
$d^{\, \rm cve}_j(\wh{x})$; these two operations do not yield this inequality  
when there is an outer summation, i.e., when $K > 1$.  The above inequality shows that
if there are positive constants $\overline{\alpha} > \underline{\alpha} > 0$ such that
$\underline{\alpha} \leq d_j^{\, \rm cve}(x) \leq \overline{\alpha}$ 
for all $j = 1, \cdots, J$ and all $x \in X$, then the following holds:
\[
\Theta_{\max}^{\rm cvx/cve}( \wh{x} ) \, \leq \,  
\Theta_{\max}^{\rm cvx/cve}(\bar{x}) - \underbrace{\displaystyle{
\frac{\rho \, \underline{\alpha}}{2 \, \overline{\alpha}}
}}_{\mbox{a constant}} \, \| \, \wh{x} - \bar{x} \, \|_2^2,
\]
which yields sufficient descent with a unit step size, provided that 
$\sigma \triangleq
\displaystyle{
\frac{\underline{\alpha}}{\overline{\alpha}}
}$.

\gap

\noindent $\bullet $ {\bf Convex-over-differentiably-convex quotient:}
This is the case
where for all $j = 1, \cdots, J$, the functions
$n^{\rm diff}_j = d^{\, \rm ndff}_j \equiv 0$, $n^{\rm ndff}_j$, written as
$n^{\rm cvx}_j$, is convex and nonnegative and $d^{\, \rm diff}_j$, written as
$d^{\, \rm cvx}_j$, is convex and positive in addition to being differentiable. 
Similar to (\ref{eq:optimality of whx}), we have, 
\[ \begin{array}{lll}
0 & \geq & \displaystyle{
\max_{1 \leq j \leq J}
} \ \displaystyle{
\frac{1}{d^{\, \rm cvx}_j(\bar{x})}
} \, \left[ \, n^{\rm cvx}_j( \wh{x} ) - n^{\rm cvx}_j(\bar{x}) -
\displaystyle{
\frac{n^{\rm cvx}_j(\bar{x})}{d^{\, \rm cvx}_j(\bar{x})}
} \, \nabla d^{\, \rm cvx}_j(\bar{x})^{\top}( \wh{x} - \bar{x} ) \, \right] 
+ \displaystyle{
\frac{\rho}{2}
} \, \| \, \wh{x} - \bar{x} \, \|_2^2 \\ [0.2in]
& \geq & \displaystyle{
\max_{1 \leq j \leq J}
} \ \displaystyle{
\frac{1}{d^{\, \rm cvx}_j(\bar{x})}
} \, \left[ \, n^{\rm cvx}_j( \wh{x} ) - n^{\rm cvx}_j( \bar{x} ) - \displaystyle{
\frac{n^{\rm cvx}_j(\bar{x})}{d^{\, \rm cvx}_j(\bar{x})}
} \, ( \, d^{\, \rm cvx}_j( \wh{x} ) - d^{\, \rm cvx}_j(\bar{x}) ) \, \right]
+ \displaystyle{
\frac{\rho}{2}
} \, \| \, \wh{x} - \bar{x} \, \|_2^2 \\ [0.2in]
& &  \mbox{by convexity of $d^{\, \rm cvx}_j$ and nonnegativity of the fraction}
\\ [0.1in]
& = & \displaystyle{
\max_{1 \leq j \leq J}
} \ \displaystyle{
\frac{1}{d^{\, \rm cvx}_j(\bar{x})}
} \, \left[ \, n^{\rm cvx}_j( \wh{x} ) - \displaystyle{
\frac{n^{\rm cvx}_j(\bar{x})}{d^{\, \rm cvx}_j(\bar{x})}
} \, d^{\, \rm cvx}_j( \wh{x} ) \, \right] + \displaystyle{
\frac{\rho}{2}
} \, \| \, \wh{x} - \bar{x} \, \|_2^2 .
\end{array}
\]
Hence, we also obtain the descent: 
\[ \begin{array}{lll}
\Theta_{\max}^{\rm cvx/cvx}( \wh{x} ) & \triangleq & 
\displaystyle{
\max_{1 \leq j \leq J}
} \, \displaystyle{
\frac{n_j^{\rm cvx}( \wh{x} )}{d_j^{\, \rm cvx}( \wh{x} )}
} \\ [0.2in]
& \leq & \displaystyle{
\max_{1 \leq j \leq J}
} \, \displaystyle{
\frac{n_j^{\rm cvx}(\bar{x})}{d_j^{\, \rm cvx}(\bar{x})}
} - \displaystyle{
\frac{\rho}{2}
} \, \| \, \wh{x} - \bar{x} \, \|_2^2 \, \left[ \, \displaystyle{
\min_{1 \leq j \leq J}
} \, \displaystyle{
\frac{d_j^{\, \rm cvx}(\bar{x})}{d_j^{\, \rm cvx}(\wh{x})} 
} \, \right] \\ [0.2in]
& = & \Theta_{\max}^{\rm cvx/cvx}(\bar{x}) - \displaystyle{
\frac{\rho}{2}
} \, \| \, \wh{x} - \bar{x} \, \|_2^2 \, \left[ \, \displaystyle{
\min_{1 \leq j \leq J}
} \, \displaystyle{
\frac{d_j^{\, \rm cvx}(\bar{x})}{d_j^{\, \rm cvx}(\wh{x})} 
} \, \right],
\end{array}
\]
which is similar to the previous case with each $d_j^{\, \rm cvx}$ being
bounded away from zero and above on $X$.

\gap

\noindent $\bullet $ {\bf Differentiably-concave-over-concave quotient:}
This is the case
where for all $j = 1, \cdots, J$, the functions
$n^{\rm ndff}_j = d^{\, \rm diff}_j(x)\equiv 0$; $n^{\rm diff}_j$, written as
$n^{\rm cve}_j$, is concave and nonnegative and $d^{\, \rm ndff}_= j$, written as
$d^{\, \rm cve}_k$, is concave and positive.  Similarly, we have
\[ \begin{array}{lll}
0 & \geq & \displaystyle{
\max_{1 \leq j \leq J}
} \ \displaystyle{
\frac{1}{d^{\, \rm cve}_j(\bar{x})}
} \, \left[ \, \nabla n^{\rm cve}_j(\bar{x})^{\top}
( \wh{x} - \bar{x} ) - \displaystyle{
\frac{n^{\rm cve}_j(\bar{x})}{d^{\, \rm cve}_j(\bar{x})}
} \, \left( \, d^{\, \rm cve}( \wh{x} ) - d^{\, \rm cve}( \bar{x} ) \, \right) \, \right] 
\\ [0.2in]
& \geq & \displaystyle{
\max_{1 \leq j \leq J}
} \ \displaystyle{
\frac{1}{d^{\, \rm cve}_j(\bar{x})}
} \, \left[ \, n^{\rm cve}_j( \wh{x} ) - n^{\rm cve}_j( \bar{x} ) - \displaystyle{
\frac{n^{\rm cve}_k(\bar{x})}{d^{\, \rm cve}_j(\bar{x})}
} \, \left( \, d^{\, \rm cve}_j( \wh{x} ) - d^{\, \rm cve}_j(\bar{x}) \, \right)
\, \right] \\ [0.2in]
& & \mbox{by concavity of $n^{\, \rm cve}_j$ and nonnegativity of the fraction} 
\\ [0.1in]
& = & \displaystyle{
\max_{1 \leq j \leq J}
} \ \displaystyle{
\frac{1}{d^{\, \rm cve}_j(\bar{x})}
} \, \left[ \, n^{\rm cve}_j( \wh{x} ) - \displaystyle{
\frac{n^{\rm cve}_j(\bar{x})}{d^{\, \rm cve}_j(\bar{x})}
} \, d^{\, \rm cve}_j( \wh{x} ) \, \right]
\end{array}
\]
which also yields $\Theta_{\max}^{\rm cve/cve}( \wh{x} ) \triangleq 
\displaystyle{
\max_{1 \leq j \leq J}
} \, \displaystyle{
\frac{n_j^{\rm cve}( \wh{x} )}{d_j^{\, \rm cve}( \wh{x} )}
} \leq \displaystyle{
\max_{1 \leq j \leq J}
} \, \displaystyle{
\frac{n_j^{\rm cve}(\bar{x})}{d_j^{\, \rm cve}(\bar{x})} 
} = \Theta_{\max}^{\rm cve/cve}( \bar{x} )$, omitting the proximal term.  \hfill $\Box$

\subsection{Subsequential convergence: Setting of Subsection~\ref{subsec:summary of types}} 

We establish the following lemma
that justifies the finiteness of the Armijo line searches in the iterations, 
thus the well-definedness of the sequence $\{ x^{\nu+1} \}$.

\begin{lemma} \label{lm:Armijo well-defined} \rm
Let $\varepsilon \geq 0$ and $\sigma \in (0,1)$ be arbitrary scalars. Let
$\wh{x}$ be the optimal solution of (\ref{eq:workhorse}) corresponding to an
arbitrary member $\wh{\boldsymbol{\xi}} \in \wh{\boldsymbol{\Xi}}_{\Theta}(\bar{x})$.
If $\wh{x} \neq \bar{x}$,  then there
exists $\bar{\tau} > 0$ (dependent on $\bar{x}$, $\varepsilon$, $\rho$, and
$\wh{\boldsymbol{\xi}}$) such that for all $\tau \in ( \, 0,\bar{\tau} \, ]$,
with $\bar{x}^{\tau} \triangleq \bar{x} + \tau ( \wh{x} - \bar{x} )$, it holds that 
\begin{equation} \label{eq:string descent} 
\Theta_{\max}(\bar{x}^{\, \tau}) - \Theta_{\max}(\bar{x}) 
\leq \, -\displaystyle{
\frac{\sigma \, \tau \, \rho}{2}
} \, \| \, \wh{x} - \bar{x} \, \|_2^2 \, < \, 0.
\end{equation}
\end{lemma}

\begin{proof} We note that for all $\tau > 0$ sufficiently small,
${\cal M}_{\Theta}(\bar{x}^{\tau}) \subseteq \, {\cal M}_{\Theta}(\bar{x})$.
Assume by way of contradiction
that no such $\bar{\tau}$ exists.  Then there exists a sequence
$\{ \tau_{\nu} \} \downarrow 0$ such that for each $\nu$,
\begin{equation} \label{eq:contradiction for unified descent}
\Theta_{\max}(\bar{x}^{\tau_{\nu}}) - \Theta_{\max}(\bar{x}) \, > \,
 -\displaystyle{
\frac{\sigma \, \tau_{\nu} \, \rho}{2}
} \, \| \, \wh{x} - \bar{x} \, \|_2^2.
\end{equation}
Let $j_{\nu} \in {\cal M}_{\Theta}(\bar{x}^{\tau_{\nu}})$ be such that
$\theta_{j_{\nu}}(\bar{x}^{\tau_{\nu}}) = \Theta_{\max}(\bar{x}^{\tau_{\nu}})$.   Then 
$j_{\nu} \in {\cal M}_{\Theta}(\bar{x}) \subseteq 
{\cal M}_{\Theta}^{\, \varepsilon}(\bar{x})$ for all $\nu$ 
sufficiently large; hence,
\[ \begin{array}{l}
\Theta_{\max}(\bar{x}^{\tau_{\nu}}) - \Theta_{\max}(\bar{x}) \, \leq \,
\theta_{j_{\nu}}(\bar{x}^{\, \tau_{\nu}}) - \theta_{j_{\nu}}(\bar{x}) \\ [0.1in]
\epc = \, \tau_{\nu} \, \theta_{j_{\nu}}^{\, \prime}(\bar{x};\wh{x} - \bar{x}) + 
\mbox{o}( \tau_{\nu} ) \epc \mbox{by Bouligand differentiability} \\ [0.1in]
\epc \leq \, \tau_{\nu} \, \left[ \, 
\wh{\theta}_{j_{\nu}}(\bullet;\bar{x};\boldsymbol{\xi}^{\, j_{\nu}}) 
\, \right]^{\, \prime}(\bar{x};\wh{x} - \bar{x}) + \mbox{o}( \tau_{\nu} )
\epc \mbox{by directional derivative dominance} \\ [0.1in]
\epc \leq \, \tau_{\nu} \, \left[ \, 
\wh{\theta}_{j_{\nu}}(\wh{x};\bar{x};\boldsymbol{\xi}^{\, j_{\nu}}) -
\wh{\theta}_{j_{\nu}}(\bar{x};\bar{x};\boldsymbol{\xi}^{\, j_{\nu}})
\, \right] + \mbox{o}( \tau_{\nu} )
\epc \mbox{by convexity of 
$\wh\theta_{j_{\nu}}(\bullet;\bar{x};\boldsymbol{\xi}^{\, j_{\nu}})$} \\ [0.15in]
\epc \leq \tau_{\nu} \, \left[ \, \displaystyle{
\max_{j \in {\cal M}_{\Theta}^{\, \varepsilon}(\bar{x})}
} \ \wh{\theta}_j(\wh{x};\bar{x};\boldsymbol{\xi}^{\, j}) -
\displaystyle{
\max_{j \in {\cal M}_{\Theta}^{\, \varepsilon}(\bar{x})}
} \ \wh{\theta}_j(\bar{x};\bar{x};\boldsymbol{\xi}^{\, j}) \, \right] 
+ \mbox{o}( \tau_{\nu} ) \\ [0.2in]
\hspace{0.45in}
\mbox{because $j_{\nu} \in {\cal M}_{\Theta}^{\, \varepsilon}(\bar{x})$ and} \\ [0.1in]
\hspace{0.4in} \wh{\theta}_{j_{\nu}}(\bar{x};\bar{x};\boldsymbol{\xi}^{\, j_{\nu}}) =
\theta_{j_{\nu}}(\bar{x}) = \displaystyle{
\max_{j \in {\cal M}_{\Theta}(\bar{x})}
} \theta_j(\bar{x}) = \displaystyle{
\max_{j \in {\cal M}_{\Theta}^{\, \varepsilon}(\bar{x})}
} \theta_j(\bar{x}) =\displaystyle{
\max_{j \in {\cal M}_{\Theta}^{\, \varepsilon}(\bar{x})}
} \ \wh{\theta}_j(\bar{x};\bar{x};\boldsymbol{\xi}^{\, j}) \\ [0.15in]
\epc \leq \, -\displaystyle{
\frac{\tau_{\nu} \, \rho}{2}
} \, \| \, \wh{x} - \bar{x} \, \|_2^2 + \mbox{o}( \tau_{\nu} ), \epc
\mbox{by optimality of $\wh{x}$ for (\ref{eq:workhorse})}.
\end{array} \]
Hence, from (\ref{eq:contradiction for unified descent}), we deduce
$-( \, 1 - \sigma \, ) \, \displaystyle{
\frac{\rho}{2}
} \, \| \, \wh{x} - \bar{x} \, \|_2^2 + \displaystyle{
\frac{\mbox{o}( \tau_{\nu} )}{\tau_{\nu}}
} \, \geq \, 0$. 
Letting $\nu \to \infty$ contradicts $\wh{x} \neq \bar{x}$ since $\sigma < 1$.
\end{proof}

We establish the following subsequential convergence result for Algorithm~I.  
In the theorem, we employ the level set corresponding to the initial iterate $x^0$ 
which is fixed throughout the analysis:
\[
\mbox{Lev}(x^0) \, \triangleq \, \left\{ \, x \, \in \, X \, \mid \, \Theta_{\max}(x) 
\, \leq \, \Theta_{\max}(x^0) \, \right\}.
\]
The proof of the theorem
follows that for a local descent method in differentiable optimization
except for the way the descent direction is generated and the non-directional
derivative based Armijo criterion in each iteration.  

\begin{theorem} \label{th:subsequential convergence I} \rm
Let $X$ be a closed convex set in $\mathbb{R}^n$.  For any $\bar{x} \in X$, let 
$\boldsymbol{\mathcal F}_{\Theta}^{\, \varepsilon}(\bar{x} )$ be the family of 
pointwise maximum functions whose maximands satisfy the six
properties summarized in Subsection~\ref{subsec:summary of types}.   
Suppose that the level set 
$\mbox{Lev}(x^0)$ is bounded.  Let $\{ x^{\nu} \}$ be a sequence generated 
by Algorithm~1.  Let $x^{\infty}$ be the limit of a convergent subsequence 
of $\{ x^{\nu} \}_{\nu \in \kappa}$ for an infinite subset $\kappa$ of $\{ 1, 2, \cdots \}$.
The following two statements hold:

\gap

{\bf (A)} if $\displaystyle{
\liminf_{\nu ( \in \kappa ) \to \infty}
} \, \tau_{\nu+1} > 0$, then $x^{\infty}$ is a weak directional stationary solution 
of (\ref{eq:unified framework}).

\gap

{\bf (B)} if $\displaystyle{
\liminf_{\nu ( \in \kappa ) \to \infty}
} \, \tau_{\nu+1} = 0$ and if the uniform upper approximation condition
(\ref{eq:uniform upper dd}) holds at $x^{\infty}$ for the functions $\wh{\theta}_j$
for all $j \in {\cal M}_{\Theta}(x^{\infty})$, 
then $x^{\infty}$ is a weak directional stationary solution 
of (\ref{eq:unified framework}).

\gap

Thus, if 
$\wh{\boldsymbol{\Xi}}_{\Theta}(x^{\infty})$ is a singleton and
the uniform upper approximation condition
(\ref{eq:uniform upper dd}) holds at $x^{\infty}$ for the functions $\wh{\theta}_j$
for all $j \in {\cal M}_{\Theta}(x^{\infty})$, 
then $x^{\infty}$ is a directional stationary solution of (\ref{eq:unified framework}).
\end{theorem}

\begin{proof} 
It $x^{\nu + 1/2} = x^{\nu}$, then $x^{\nu}$ is the optimal solution of the problem:
\[
\displaystyle{
\operatornamewithlimits{\mbox{\bf minimize}}_{x \in X}
} \, \displaystyle{
\max_{j \in {\cal M}_{\Theta}^{\, \varepsilon_{\nu}}(x^{\nu})}
} \wh{\theta}_j(x;x^{\nu};\boldsymbol{\xi}^{\, \nu;j}) + \displaystyle{
\frac{\rho}{2}
} \, \| \, x - x^{\nu} \|_2^2.
\]
For every 
$j \in {\cal M}_{\Theta}^{\, \varepsilon_{\nu}}(x^{\nu}) \setminus {\cal M}_{\Theta}(x^{\nu})$,
we have
\[ 
\wh{\theta}_j(x^{\nu};x^{\nu};\boldsymbol{\xi}^{\, \nu;j}) \, = \, \theta_j(x^{\nu}) \, < \,
\Theta_{\max}(x^{\nu}) \, = \, \theta_{j^{\, \prime}}(x^{\nu}), \epc
\forall \, j^{\, \prime} \, \in \, {\cal M}_{\Theta}(x^{\nu}).
\]
Since ${\cal M}_{\Theta}^{\, \varepsilon_{\nu}}(x^{\nu})$ contains
${\cal M}_{\Theta}(x^{\nu})$, it follows that
\[ \begin{array}{lll}
0 & \leq & \left[ \, \displaystyle{
\max_{j \in {\cal M}_{\Theta}^{\, \varepsilon}(x^{\nu})}
} \wh{\theta}_j(\bullet;x^{\nu};\boldsymbol{\xi}^{\, \nu;j}) \, \right]^{\, \prime}(
x^{\nu};x - x^{\nu}) \\ [0.2in]
& = &  \left[ \, \displaystyle{
\max_{j \in {\cal M}_{\Theta}(x^{\nu})}
} \wh{\theta}_j(\bullet;x^{\nu};\boldsymbol{\xi}^{\, \nu;j}) \, \right]^{\, \prime}(
x^{\nu};x - x^{\nu}) \epc \forall \, \, x \in \, X.
\end{array} \]
Thus, $x^{\nu}$ is also the optimal solution of the problem:
\[
\displaystyle{
\operatornamewithlimits{\mbox{\bf minimize}}_{x \in X}
} \, \displaystyle{
\max_{j \in {\cal M}_{\Theta}(x^{\nu})}
} \wh{\theta}_j(x;x^{\nu};\boldsymbol{\xi}^{\, \nu;j}) + \displaystyle{
\frac{\rho}{2}
} \, \| \, x - x^{\nu} \|_2^2,
\]
establishing the weak stationarity of $x^{\nu}$ for (\ref{eq:unified framework}).

\gap

Assume now that $x^{\nu + 1/2} \neq x^{\nu}$ for all $\nu$.  We then have
\[
\Theta_{\max}(x^{\nu+1}) - \Theta_{\max}(x^{\nu}) \, \leq \, -\displaystyle{
\frac{\sigma \, \rho}{2}
} \, \tau_{\nu+1} \, \| \, x^{\nu + 1/2} - x^{\nu} \, \|_2^2 \, < \, 0
\]	
Thus, the sequence $\{ \Theta_{\max}(x^{\nu}) \}$ is nonincreasing.  Since
the sequence $\{ x^{\nu} \}$ belongs to the set $\mbox{Lev}(x^0)$ which is
bounded by assumption, the sequence $\{ \Theta_{\max}(x^{\nu}) \}$ is
therefore bounded and thus converges.  Hence 
\[
\displaystyle{
\lim_{\nu \to \infty}
} \, \tau_{\nu+1} \, \| \, x^{\nu + 1/2} - x^{\nu} \|_2^2 \, = \, 0.
\]
Moreover, since the level set $\mbox{Lev}(x^0)$ is bounded, it follows that
the sequence $\{ x^{\nu} \}$ has at least one accumulation point.  This justifies
the well-definedness of the accumulation point
$x^{\infty}$.  We divide the remaining proof into two cases (A) and (B).

\gap

{\bf Case (A):}  $\displaystyle{
\liminf_{\nu ( \in \kappa ) \to \infty}
} \, \tau_{\nu+1} \, > \, 0$.  Then, 
$\displaystyle{
\lim_{\nu ( \in \kappa ) \to \infty}
} \, x^{\nu + 1/2} \, = \, x^{\infty}$.  By definition, for each $\nu$, there is
a tuple $\wh{\boldsymbol{\xi}}^{\, \nu} \triangleq 
\left( \, \boldsymbol{\xi}^{\, \nu;j} \, \right) \in \, \wh{\boldsymbol{\Xi}}_{\Theta}(x^{\nu})$
such that $x^{\nu + 1/2}$ is a minimizer of 
$\Theta_{\max}^{\, \varepsilon_{\nu}}(\bullet;x^{\nu};\wh{\boldsymbol{\xi}}^{\, \nu}) 
= \displaystyle{
\max_{j \in {\cal M}_{\Theta}^{\, \varepsilon_{\nu}}(x^{\nu})}
} \ \wh{\theta}_j(\bullet;x^{\nu};\boldsymbol{\xi}^{\, \nu;j})$ on $X$.
For each $\nu \in \kappa$, there is a subset 
\[
{\cal M}_{\nu + 1/2} \, \triangleq \, \left\{ \, j \, \in \, 
{\cal M}_{\Theta}^{\, \varepsilon_{\nu}}(x^{\nu})
\, \left| \, \wh{\theta}_j(x^{\nu+1/2};x^{\nu};\boldsymbol{\xi}^{\, \nu;j}) 
\, = \, \displaystyle{
\max_{j^{\prime} \in {\cal M}_{\Theta}^{\, \varepsilon_{\nu}}(x^{\nu})}
} \, \wh{\theta}_{j^{\prime}}(x^{\nu+1/2};x^{\nu};\boldsymbol{\xi}^{\, \nu;j^{\prime}}) 
\, \right. \right\}.
\]
There exist a further subsequence $\kappa^{\prime} \subseteq \kappa$ and a constant index
set ${\cal M}_{\infty}$ such that ${\cal M}_{\nu + 1/2} = {\cal M}_{\infty}$ for all
$\nu \in \kappa^{\, \prime}$.   There are two subcases: 

\gap

{\bf Subcase (Ai):} $\varepsilon_{\infty} = 0$.  In this case,
the set ${\cal M}_{\infty}$ must be a subset of
${\cal M}_{\Theta}(x^{\infty})$, as can be argued as follows.  Indeed, if 
$j$ belongs to ${\cal M}_{\infty}$ which is a subset of 
${\cal M}_{\Theta}^{\, \varepsilon_{\nu}}(x^{\nu})$, then
$\theta_j(x^{\nu}) \geq \, \Theta_{\max}(x^{\nu}) - \varepsilon_{\nu}$ 
for all $\nu \in \kappa^{\, \prime}$. 
Passing to the limit $\nu (\in \kappa^{\, \prime}) \to \infty$ in the
inequality $\theta_j(x^{\nu}) \geq \, \Theta_{\max}(x^{\nu}) - \varepsilon_{\nu}$ 
yields $j \in {\cal M}_{\Theta}(x^{\infty})$.  By the optimality 
of $x^{\nu + 1/2}$, we have
\begin{equation} \label{eq:stationarity jth}
\displaystyle{
\max_{j \in {\cal M}_{\infty}}
} \, \left\{ \left[ \, \wh{\theta}_j(\bullet;x^{\nu};\boldsymbol{\xi}^{\, \nu;j}) 
\, \right]^{\, \prime}(x^{\nu + 1/2}; x - x^{\nu + 1/2}) \, \right\} 
\, \geq \, 0, \epc \forall \, x \in X.
\end{equation}
Since the union $\displaystyle{
\bigcup_{\nu}
} \ \wh{\boldsymbol{\Xi}}(x^{\nu})$ is bounded, 
the sequence $\{ \wh{\boldsymbol{\xi}}^{\, \nu} \}_{\nu \in \kappa}$ 
has an accumulation point.  Without loss of generality, we may assume that 
the sequence $\{ \wh{\boldsymbol{\xi}}^{\, \nu} \}_{\nu \in \kappa^{\prime}}$ 
converges to a tuple
$\wh{\boldsymbol{\xi}}^{\, \infty} \triangleq 
\left( \, \boldsymbol{\xi}^{\, \infty;j} \, \right)_{j=1}^J$ that belongs to 
$\wh{\boldsymbol{\Xi}}_{\Theta}(x^{\infty})$.  By the dd-joint-usc property
of each function $\wh{\theta}_j$, we deduce that, by passing to the limit
$\nu (\in \kappa^{\prime}) \to  \infty$ in (\ref{eq:stationarity jth}),
\[ 
\displaystyle{
\max_{j \in {\cal M}_{\Theta}(x^{\infty})}
} \, \left\{ \left[ \, \wh{\theta}_j(\bullet;x^{\infty};\boldsymbol{\xi}^{\, \infty;j}) 
\, \right]^{\, \prime}(x^{\infty}; x - x^{\infty}) \, \right\} 
\, \geq \, \displaystyle{
\max_{j \in {\cal M}_{\infty}}
} \, \left\{ \left[ \, \wh{\theta}_j(\bullet;x^{\infty};\boldsymbol{\xi}^{\, \infty;j}) 
\, \right]^{\, \prime}(x^{\infty}; x - x^{\infty}) \, \right\} 
\, \geq \, 0
\]
for all $x \in $X,
establishing the weak stationarity of $x^{\infty}$ for (\ref{eq:unified framework}).

\gap

{\bf Subcase (Aii):} $\varepsilon_{\infty} > 0$.  In this case, we must have
${\cal M}_{\Theta}(x^{\infty}) \subseteq {\cal M}_{\Theta}^{\, \varepsilon_{\nu}}(x^{\nu})$
for all $\nu \in \kappa$ sufficiently large.  Hence for all $j \in {\cal M}_{\infty}$,
\[
\wh{\theta}_j(x^{\nu+1/2};x^{\nu};\boldsymbol{\xi}^{\, \nu;j}) \, \geq \, 
\wh{\theta}_{j^{\, \prime}}(x^{\nu+1/2};x^{\nu};\boldsymbol{\xi}^{\, \nu;j^{\, \prime}}), 
\epc \forall \, j^{\, \prime} \, \in \, {\cal M}_{\Theta}(x^{\infty}).
\]
Passing to the limit yields $j \in {\cal M}_{\Theta}(x^{\infty})$, by the joint
continuity of $\wh{\theta}_j$ and $\wh{\theta}_{j^{\prime}}$ at $x^{\infty}$.  
The remaining proof is the same as for subcase (Ai).

\gap


{\bf Case (B):} $\displaystyle{
\liminf_{\nu (\in \kappa) \to \infty}
} \, \tau_{\nu+1} = 0$.  In this case, the following reverse Armijo condition
holds for infinity many $\nu \in \kappa$: i.e.,
\begin{equation} \label{eq:reverse Armijo}
\Theta_{\max}\left( \underbrace{x^{\nu} + \displaystyle{
\frac{\tau_{\nu+1}}{\beta}
} \, d^{\, \nu+1}}_{\mbox{denoted $\wh{x}^{\, \nu}$}} \right) - \Theta_{\max}(x^{\nu}) 
\, > \, -\displaystyle{
\frac{\sigma \, \rho}{2}
}  \, \displaystyle{
\frac{\tau_{\nu+1}}{\beta}
} \, \| \, x^{\nu + 1/2} - x^{\nu} \, \|_2^2.
\end{equation}
By working with a further subsequence if needed, we may assume that
(\ref{eq:reverse Armijo}) holds for all $\nu \in \kappa$.
We first show that the sequence $\{ x^{\nu + 1/2} \}_{\nu \in \kappa}$ is bounded.
By the lower Lipschitz boundedness condition: there exists a constant $B > 0$ such that 
for all $\nu$,
\[
-\displaystyle{
\frac{\rho}{2}
} \, \| \, x^{\nu + 1/2} - x^{\nu} \, \|_2^2 \, \geq \, 
\wh{\Theta}_{\max}^{\, \varepsilon_{\nu}}(x^{\nu + 1/2};x^{\nu};\wh{\boldsymbol{\xi}}^{\, \nu}) 
- \wh{\Theta}_{\max}^{\, \varepsilon_{\nu}}(x^{\nu};x^{\nu};\wh{\boldsymbol{\xi}}^{\, \nu}) 
\, \geq \, -B \, \| \, x^{\nu + 1/2} - x^{\nu} \, \|_2.
\]
Thus, the boundedness of $\{ x^{\nu + 1/2} \}_{\nu \in \kappa}$ follows.
Hence, $\displaystyle{
\lim_{\nu (\in \kappa) \to \infty}
} \, \wh{x}^{\, \nu} = x^{\infty}$ also.
Let $j_{\nu}$ be an index belonging to ${\cal M}_{\Theta}(\wh{x}^{\, \nu})$
for infinitely many $\nu \in \kappa$.  We have, with 
\[ \begin{array}{lll}
\Theta_{\max}( \wh{x}^{\, \nu}) - \Theta_{\max}(x^{\nu}) & \leq &
\theta_{j_{\nu}}(\wh{x}^{\, \nu}) - \theta_{j_{\nu}}(x^{\nu}) \\ [0.15in]
& \leq & \displaystyle{
\frac{\tau_{\nu+1}}{\beta}
} \, \left[ \, \wh{\theta}_{j_{\nu}}(x^{\nu+1/2};x^{\nu};\boldsymbol{\xi}^{{\, \nu};j_{\nu}})
- \theta_{j_{\nu}}(x^{\nu}) \, \right] + \mbox{o}(\tau_{\nu+1}), \epc
\mbox{by (\ref{eq:uniform upper dd})} \\ [0.15in]
& \leq & -\displaystyle{
\frac{\rho}{2}
} \, \displaystyle{
\frac{\tau_{\nu+1}}{\beta}
} \, \| \, x^{\nu + 1/2} - x^{\nu} \, \|_2^2 + \mbox{o}( \tau_{\nu+1} ),
\end{array} \]
which together with (\ref{eq:reverse Armijo}) yields
\[
0 \, \leq \, 
-\displaystyle{
\frac{( 1 - \sigma ) \, \rho}{2 \beta}
} \, \| \, x^{\nu + 1/2} - x^{\nu} \, \|_2^2 + \displaystyle{
\frac{\mbox{o}(\tau_{\nu+1})}{\tau_{\nu+1}}
} .
\]
Hence passing to the limit $\nu (\in \kappa) \to \infty$, we may deduce that
$\{ x^{\nu + 1/2} \}_{\nu \in \kappa}$ converges to $x^{\infty}$.  Now the proof
of case (A) applies, yielding the same stationarity property of $x^{\infty}$.
The last statement of the theorem does not require a proof.
\end{proof}

\begin{remark} \label{rm:inequality in convergence} \rm
The inequality (\ref{eq:uniform upper dd}) is assumed only at an accumulation point
of the sequence of iterates produced by the algorithm, but not at the iterates
themselves.  Admittedly, this requirement prevents us from fully treating the general
case with the pointwise maximum of multiple differentiable functions that are
not assumed to be either convex or concave.  In view of the early 
work \cite{DemyanovGSDixon86} where the inner functions do not have the dc summands, 
it might be possible to relax this requirement by extending the algorithm in 
the cited work, 
which nevertheless is computationally more demanding than Algorithm~1 and computes
only ``$\varepsilon$-inf stationary solutions''.  \hfill $\Box$
\end{remark}

Due to the importance of the type~I compositions, we state a corollary of
Theorem~\ref{th:subsequential convergence I} for composite functions of this
kind.  Note that the assumption of composite convexity already
contains the singleton requirement of the maximizing index 
sets ${\cal M}_{jk;\max}^{\rm diff}$.  No proof is required of the corollary.
Similar conclusions can be made for the other 
types of composite functions under appropriate assumptions and are omitted.

\begin{corollary} \label{co:subsequent type I} \rm
Let $X$ be a closed convex set in $\mathbb{R}^n$.  Suppose that the level 
set $\mbox{Lev}(x^0)$ is bounded.  For a composite function of type I with
outer differentiable functions $\phi_j$ and inner functions $p_{jk}$
given by (\ref{eq:dc+diff}),
suppose that the (C$^{\, 2}$DC$^{\, 2}$) 
assumption holds for all pairs $( \phi_j,P^{\, j} )$ for $j = 1, \cdots, J$
at all $\bar{x} \in X$.  The every accumulation point
produced by Algorithm~1 is a directional stationary solution of
(\ref{eq:unified framework}).  \hfill $\Box$ 	
\end{corollary}

As an ending remark for Algorithm~1, we suggest for possible future
research the topic of incorporating the ideas of bundle methods for nonsmooth
DC programming into the treatment of the problem (\ref{eq:unified framework}).  

\subsection{Computing directional stationary solution} \label{subsec:computing directional}

Without the singleton assumption of $\wh{\boldsymbol{\Xi}}_{\Theta}(x^{\infty})$ at a limit
point $x^{\infty}$, Algorithm~1 is not sufficient to yield a directional stationary
solution of (\ref{eq:unified framework}).  Extending the references 
\cite{PangRazaAlvarado17,LuZhou19,LuZhouSun19} for dc programs with 
difference-of-finite-max-differentiable functions (\ref{eq:max-max}), we consider
the problem (\ref{eq:unified framework}), where each $\phi_j$ is continuously differentiable
and each
\begin{equation} \label{eq:max-max inner}
p_{jk}^{\rm cvx}(x) \, = \, \displaystyle{
\max_{1 \leq \ell \leq L_{\rm x}}
} \, p_{jk\ell}^{\, \rm cvx}(x), \epc 
p_{jk}^{\, \rm cve}(x) \, = \, \displaystyle{
\min_{1 \leq \ell \leq L_{\rm e}}
} \, p_{jk\ell}^{\, \rm cve}(x), \epc \mbox{and} \epc
p_{jk}^{\, \rm diff}(x) \, = \, \displaystyle{
\max_{1 \leq \ell \leq L_{\rm d}}
} \, p_{jk\ell}^{\, \rm diff}(x),
\end{equation}
for some positive integers $L_{\rm x}$, $L_{\rm e}$, and $L_{\rm d}$,
where each $p_{jk\ell}^{\, \rm cvx}$ is a convex, continuous differentiable function,
$p_{jk\ell}^{\, \rm cve}$ is a concave, continuously differentiable function, 
and $p_{jk\ell}^{\, \rm diff}$ is a continuously differentiable function that is
neither convex nor concave.  For any given vector $\bar{x} \in X$ and scalar
$\delta \geq 0$, define the index sets
\[ \begin{array}{lll}
{\cal M}_{jk;\delta}^{\, \rm cvx}(\bar{x}) & \triangleq & \left\{ \, 
\ell \, \in \, [ \, L_{\rm x} \, ] \, = \, p_{jk\ell}^{\, \rm cvx}(\bar{x}) \, \geq \,
p_{jk}^{\rm cvx}(x) - \delta \, \right\} \\ [0.1in]
{\cal M}_{jk;\delta}^{\, \rm cve}(\bar{x}) & \triangleq & \left\{ \, 
\ell \, \in \, [ \, L_{\rm e} \, ] \, = \, p_{jk\ell}^{\, \rm cve}(\bar{x}) \, \leq \,
p_{jk}^{\rm cve}(x) + \delta \, \right\} \\ [0.1in]
{\cal M}_{jk;\delta}^{\, \rm diff}(\bar{x}) & \triangleq & \left\{ \, 
\ell \, \in \, [ \, L_{\rm d} \, ] \, = \, p_{jk\ell}^{\, \rm diff}(\bar{x}) \, \geq \,
p_{jk}^{\rm diff}(x) - \delta \, \right\};
\end{array} \] 
we write these index sets without the subscript $\delta$ when it is zero.  For a
tuple of indices
\[
{\cal L} \, \triangleq \, \left\{ \, ( \, \ell^{\, \rm x}_{jk} \, \ell^{\, \rm e}_{jk}, \,
\ell^{\, \rm d}_{jk} ) \, : (j,k)  \in [ J ] \times [ K ] \, \right\} \in 
\boldsymbol{\mathcal L}_{\delta}(\bar{x}) = \displaystyle{
\prod_{j=1}^J
} \ \displaystyle{
\prod_{k=1}^K
} \, \left[ \, {\cal M}_{jk;\delta}^{\, \rm cvx}(\bar{x}) \, \times \, 
{\cal M}_{jk;\delta}^{\, \rm cve}(\bar{x}) \, \times \, 
{\cal M}_{jk;\delta}^{\, \rm diff}(\bar{x}) \, \right]
\]
define the convex function $\wh{\theta}_j(\bullet;\bar{x};{\cal L})$ as
(where $\bar{y}^j \triangleq P^j(\bar{x})$)
\[ 
\wh{\theta}_j(x;\bar{x};{\cal L}) \, \triangleq \, \theta_j(\bar{x}) + 
\left[ \, \begin{array}{l}
\displaystyle{
\sum_{k=1}^K
} \, \max\left( \, \begin{array}{l}
\displaystyle{
\frac{\partial \phi_j(\bar{y}^j)}{\partial y_k}
} \, \left[ \, \begin{array}{l}
p_{jk}^{\rm cvx}(x) - p_{jk}^{\rm cvx}(\bar{x})
+ \nabla p_{jk\ell_{jk}^{\, \rm e}}^{\, \rm cve}(\bar{x})^{\top}(x - \bar{x}) \, + \\ [0.1in]
\displaystyle{
\max_{\ell \in {\cal M}_{jk;\delta}^{\rm diff}(\bar{x})}
} \, \nabla p_{jk \ell}^{\rm diff}(\bar{x})^{\top}(x - \bar{x}) 
\end{array} \right], \\ [0.3in]
\displaystyle{
\frac{\partial \phi_j(\bar{y}^j)}{\partial y_k}
} \, \left[ \, \begin{array}{l}
\nabla p_{jk\ell_{jk}^{\, \rm x}}^{\, \rm cvx}(\bar{x})^{\top}(x - \bar{x}) + 
p_{jk}^{\rm cve}(x) 
- p_{jk}^{\rm cve}(\bar{x}) \, + \\ [0.15in]
\nabla p_{jk \ell_{jk}^{\, \rm d}}^{\, \rm diff}(\bar{x})^{\top}( x - \bar{x})
\end{array} \right]
\end{array} \right)
\end{array} \right]
\]
and consider the individual program:
\begin{equation} \label{eq:max-max individual subproblem} 
\displaystyle{
\operatornamewithlimits{\mbox{\bf minimize}}_{x \in X}
} \ \wh{\Theta}_{\max}^{\, \varepsilon;\rho}(x;\bar{x};{\cal L}) \, \triangleq \, 
\displaystyle{
\max_{j \in {\cal M}_{\Theta}^{\, \varepsilon}(\bar{x})}
} \, \wh{\theta}_j(x;\bar{x};{\cal L}) + \displaystyle{
\frac{\rho}{2}
} \, \| \, x - \bar{x} \, \|_2^2
\end{equation}
as well as the minimum over all tuples 
${\cal L} \in \boldsymbol{\mathcal L}_{\delta}(\bar{x})$:
\begin{equation} \label{eq:max-max all subproblem}
\displaystyle{
\operatornamewithlimits{\mbox{\bf minimum}}_{
{\cal L} \, \in \, \boldsymbol{\mathcal L}_{\delta}(\bar{x})}
} \, \left[ \, \displaystyle{
\operatornamewithlimits{\mbox{\bf minimum}}_{x \in X}
} \ \wh{\Theta}_{\max}^{\, \varepsilon;\rho}(x;\bar{x};{\cal L}) \, \right].
\end{equation}

\begin{algorithm}
\caption{A descent algorithm for (\ref{eq:unified framework}): with inner functions
given by (\ref{eq:max-max inner})}
\label{alg II:descent}
\begin{algorithmic}
\State{\bf Initialization:}  Let $\rho$, $\delta$, $\sigma$ and
$\beta$ be positive scalars with $\sigma$ and $\beta$ both less than 1; these are
fixed throughout the following iterations.  Let $\{ \varepsilon_{\nu} \}$ be a
sequence of nonnegative scalars converging to the limit $\varepsilon_{\infty} \geq 0$.
Let $x^0 \in X$ be given; let $\nu = 0$.
\For{$\nu=0, 1, \cdots$,} 
\begin{enumerate}
\item Compute the minimum in (\ref{eq:max-max all subproblem}); let 
      ${\cal L}_{\min}^{\, \nu} \in \boldsymbol{\mathcal L}_{\delta}(x^{\nu})$ 
      be a minimizing tuple and $x^{\nu+1/2}$ be the minimizer of
      $\wh{\Theta}_{\max}^{\, \varepsilon_{\nu};\rho}(\bullet;x^{\nu};{\cal L}^{\, \nu}_{\min})$
      on $X$.
\item  If $x^{\nu + 1/2} = x^{\nu}$, stop; $x^{\nu}$ is a desired directional stationary
       solution of (\ref{eq:unified framework}). Otherwise, let
       $d^{\, \nu+1} \triangleq x^{\nu+1/2} - x^{\nu}$ and proceed to the next step. 
\item (Armijo line search) Let  
      $m_{\nu}$ be the smallest integer $m \geq 0$ such that 
\begin{equation} \label{eq:Armijo descent}
\Theta_{\max}(x^{\nu} + \beta^m d^{\, \nu + 1}) - \Theta_{\max}(x^{\nu})
\, \leq \, -\displaystyle{
\frac{\sigma \, \rho}{2}
} \, \, \beta^m \, \| \, x^{\nu + 1/2} - x^{\nu} \, \|_2^2.
\end{equation}
\item Let $x^{\nu+1} = x^{\nu} + \tau_{\nu+1} d^{\, \nu+1}$ where $\tau_{\nu+1} \triangleq
      \beta^{\, m_{\nu}}$.  Let $\nu \leftarrow \nu + 1$ and return to Step~1.
\end{enumerate}
\EndFor	 
\end{algorithmic}
\end{algorithm}

We omit the proof of finite termination when $x^{\nu+1/2} = x^{\nu}$ and proceed directly to
the limit analysis.

\begin{theorem} \label{th:subsequential convergence II} \rm
Let $X$ be a closed convex set.  Let each $\phi_j$ be continuously differentiable
and each $p_{jk}(x) = p_{jk}^{\, \rm cvx}(x) + p_{jk}^{\, \rm cve}(x) +
p_{jk;\max}^{\, \rm diff}(x)$ with the individual functions being given by
(\ref{eq:max-max inner}).  Suppose that the level set $\mbox{Lev}(x^0)$ is bounded.
If $x^{\infty}$ is an accumulation point of a sequence $\{ x^{\nu} \}$
generated by Algorithm~2 such that ${\cal M}_{jk;\max}^{\rm diff}(x^{\infty})$ 
is a singleton for all $k$ satisfying $\displaystyle{
\frac{\partial \phi_j(y^{j;\infty})}{\partial y_k}
} > 0$, where $y^{j;\infty} \triangleq P^j(x^{\infty})$, then 
$x^{\infty}$ is a directional stationary solution of (\ref{eq:unified framework}).
\end{theorem}

\begin{proof} Following the proof of 
Theorem~\ref{th:subsequential convergence I}, we can show that an accumulation point
$x^{\infty}$ as stated is a weak stationary solution of (\ref{eq:unified framework}).  
[Note: the assumption that ${\cal M}_{jk;\max}^{\rm diff}(x^{\infty})$ 
is a singleton for all $k$ satisfying $\displaystyle{
\frac{\partial \phi_j(y^{j;\infty})}{\partial y_k}
} > 0$ is needed to apply the inequality (\ref{eq:uniform upper dd}) in the proof
of case (B).]
More importantly, letting $\{ x^{\nu} \}_{\nu \in \kappa}$ be the subsequence 
with $x^{\infty}$ as its
limit, it then follows by the subsequential convergence proof that the sequence
$\{ x^{\nu + 1/2} \}_{\nu \in \kappa}$ also converges to $x^{\infty}$.  Moreover,
by working with a further subsequence, we may assume that members of the family
$\{ {\cal L}_{\min}^{\nu} \}_{\nu \in \kappa}$ of minimizing tuples are all equal to
the same tuple, which we denote by ${\cal L}_{\min}^{\infty}$.  
To prove the claim that $x^{\infty}$ is a desired
directional stationary solution, it suffices to show, by the converse statement
in part (B) of Proposition~\ref{pr:unified stationarity certificate}, that for all tuples
${\cal L} \triangleq \left\{ \, ( \, \ell^{\, \rm x}_{jk} \, \ell^{\, \rm e}_{jk}, \,
\ell^{\, \rm d}_{jk} ) \, : (j,k)  \in [ J ] \times [ K ] \, \right\}
\in \boldsymbol{\mathcal L}(x^{\infty})$, 
\begin{equation} \label{eq:intermediate step for directional}
\{ \, x^{\infty} \} \, = \, \displaystyle{
\operatornamewithlimits{\mbox{\bf argmin}}_{x \in X}
} \, \displaystyle{
\max_{j \in {\cal M}_{\Theta}^{\varepsilon_{\infty}}(\bar{x})}
} \, \wh{\theta}_j(x;x^{\infty};{\cal L}) + \displaystyle{
\frac{\rho}{2}
} \, \| \, x - x^{\infty} \, \|_2^2;
\end{equation}
For the given $\delta > 0$, any tuple 
$( \, \ell^{\, \rm x}_{jk} \, \ell^{\, \rm e}_{jk}, \, \ell^{\, \rm d}_{jk} ) \in 
{\cal M}_{jk}^{\, \rm cvx}(x^{\infty}) \times 
{\cal M}_{jk}^{\, \rm cve}(x^{\infty}) \times {\cal M}_{jk}^{\, \rm diff}(x^{\infty})$
must belong to ${\cal M}_{jk;\delta}^{\, \rm cvx}(x^{\nu}) \times 
{\cal M}_{jk;\delta}^{\, \rm cve}(x^{\nu}) \times {\cal M}_{jk;\delta}^{\, \rm diff}(x^{\nu})$
for all $\nu \in \kappa$ sufficiently large.  Therefore, for any $x \in X$, we have
\begin{equation} \label{eq:intermediate theta}
\displaystyle{
\max_{j \in {\cal M}_{\Theta}^{\, \varepsilon_{\nu}}(x^{\nu})}
} \, \wh{\theta}_j(x;x^{\nu};{\cal L}) + \displaystyle{
\frac{\rho}{2}
} \, \| \, x - x^{\nu} \, \|_2^2 \, \geq \, \displaystyle{
\max_{j \in {\cal M}_{\Theta}^{\, \varepsilon_{\nu}}(x^{\nu})}
} \, \wh{\theta}_j(x^{\nu+1/2};x^{\nu};{\cal L}_{\min}^{\infty}) + \displaystyle{
\frac{\rho}{2}
} \, \| \, x^{\nu+1/2} - x^{\nu} \, \|_2^2.
\end{equation}
By working with a further subsequence, we may assume that the maximzing index sets
\[
\left\{ \, j \, \in \, 
{\cal M}_{\Theta}^{\, \varepsilon_{\nu}}(x^{\nu})
\, \left| \, \wh{\theta}_j(x^{\nu+1/2};x^{\nu};{\cal L}) 
\, = \, \displaystyle{
\max_{j^{\prime} \in {\cal M}_{\Theta}^{\, \varepsilon_{\nu}}(x^{\nu})}
} \, \wh{\theta}_{j^{\prime}}(x^{\nu+1/2};x^{\nu};{\cal L}) 
\, \right. \right\}
\]
are equal to a common index set
${\cal M}_{\infty}$ for all $\nu \in \kappa$, which must be a subset of 
${\cal M}_{\Theta}^{ \varepsilon_{\infty}}(x^{\infty})$.  In fact,
${\cal M}_{\Theta}^{\, \varepsilon_{\nu}}(x^{\nu})$ itself is a subset of
${\cal M}_{\Theta}^{ \varepsilon_{\infty}}(x^{\infty})$ for all $\nu$ sufficiently
large.  Therefore, the above inequality yields
\[ \begin{array}{lll}
\displaystyle{
\max_{j \in {\cal M}_{\Theta}^{\, \varepsilon_{\infty}}(x^{\infty})}
} \, \wh{\theta}_j(x;x^{\nu};{\cal L}) + \displaystyle{
\frac{\rho}{2}
} \, \| \, x - x^{\nu} \, \|_2^2 & \geq & \displaystyle{
\max_{j \in {\cal M}_{\Theta}^{\, \varepsilon_{\nu}}(x^{\nu})}
} \, \wh{\theta}_j(x;x^{\nu};{\cal L}) + \displaystyle{
\frac{\rho}{2}
} \, \| \, x - x^{\nu} \, \|_2^2 \\ [0.2in]
& \geq & \displaystyle{
\max_{j \in {\cal M}_{\Theta}^{\, \varepsilon_{\nu}}(x^{\nu})}
} \, \wh{\theta}_j(x^{\nu + 1/2};x^{\nu};{\cal L}_{\min}^{\infty}) + \displaystyle{
\frac{\rho}{2}
} \, \| \, x - x^{\nu} \, \|_2^2 \\ [0.2in]
& \geq & \displaystyle{
\max_{j \in {\cal M}_{\infty}}
} \, \wh{\theta}_j(x^{\nu+1/2};x^{\nu};{\cal L}_{\min}^{\infty}) + \displaystyle{
\frac{\rho}{2}
} \, \| \, x^{\nu+1/2} - x^{\nu} \, \|_2^2.
\end{array}
\]
Passing to the limit $\nu (\in \kappa) \to \infty$, we obtain:
\[
\displaystyle{
\max_{j \in {\cal M}_{\Theta}^{\, \varepsilon_{\infty}}(x^{\infty})}
} \, \wh{\theta}_j(x;x^{\infty};{\cal L}) + \displaystyle{
\frac{\rho}{2}
} \, \| \, x - x^{\infty} \, \|_2^2 \, \geq \, \displaystyle{
\max_{j \in {\cal M}_{\infty}}
} \, \underbrace{\wh{\theta}_j(x^{\infty};x^{\infty};{\cal L}_{\min}^{\infty})}_{\mbox{
$= \theta_j(x^{\infty})$}}.
\]
There are two cases: $\varepsilon_{\infty} = 0$ or $\varepsilon_{\infty} > 0$.  In the
former case, ${\cal M}_{\infty} \subseteq {\cal M}_{\Theta}(x^{\infty})$.  Thus,
\[ \begin{array}{l}
\displaystyle{
\max_{j \in {\cal M}_{\Theta}^{\, \varepsilon_{\infty}}(x^{\infty})}
} \, \wh{\theta}_j(x;x^{\nu};{\cal L}) + \displaystyle{
\frac{\rho}{2}
} \, \| \, x - x^{\nu} \, \|_2^2 \\ [0.2in]
\hspace{0.2in} \geq \, \displaystyle{
\max_{j \in {\cal M}_{\infty}}
} \, \theta_j(x^{\infty})\, = \, \displaystyle{
\max_{j \in {\cal M}_{\Theta}(x^{\infty})}
} \, \theta_j(x^{\infty}) \, = \, \displaystyle{
\max_{j \in {\cal M}_{\Theta}^{\varepsilon_{\infty}}(x^{\infty})}
} \, \theta_j(x^{\infty}) \\ [0.2in]
\hspace{0.2in} = \, \displaystyle{
\max_{j \in {\cal M}_{\Theta}^{\, \varepsilon_{\infty}}(x^{\infty})}
} \, \wh{\theta}_j(x^{\infty};x^{\infty};{\cal L}).
\end{array}
\]
This establishes (\ref{eq:intermediate step for directional}) by passing to the
limit $\nu (\in \kappa) \to \infty$.  In the latter case
where $\varepsilon_{\infty} > 0$, we have ${\cal M}_{\Theta}(x^{\infty}) \subseteq 
{\cal M}_{\Theta}^{\varepsilon_{\nu}}(x^{\nu})$ for all $\nu \in \kappa$ sufficiently
large.  Therefore, (\ref{eq:intermediate theta}) yields:
\[
\displaystyle{
\max_{j \in {\cal M}_{\Theta}^{\, \varepsilon_{\nu}}(x^{\nu})}
} \, \wh{\theta}_j(x;x^{\nu};{\cal L}) + \displaystyle{
\frac{\rho}{2}
} \, \| \, x - x^{\nu} \, \|_2^2 \, \geq \, \displaystyle{
\max_{j \in {\cal M}_{\Theta}(x^{\infty})}
} \, \wh{\theta}_j(x^{\nu+1/2};x^{\nu};{\cal L}_{\min}^{\infty}) + \displaystyle{
\frac{\rho}{2}
} \, \| \, x^{\nu+1/2} - x^{\nu} \, \|_2^2,
\]
which implies, since ${\cal M}_{\Theta}^{\, \varepsilon_{\nu}}(x^{\nu})$ 
is a subset of ${\cal M}_{\Theta}^{ \varepsilon_{\infty}}(x^{\infty})$ 
for all $\nu$ sufficiently large,
\[
\displaystyle{
\max_{j \in {\cal M}_{\Theta}^{\, \varepsilon_{\infty}}(x^{\infty})}
} \, \wh{\theta}_j(x;x^{\nu};{\cal L}) + \displaystyle{
\frac{\rho}{2}
} \, \| \, x - x^{\nu} \, \|_2^2 \, \geq \, \displaystyle{
\max_{j \in {\cal M}_{\Theta}(x^{\infty})}
} \, \wh{\theta}_j(x^{\nu+1/2};x^{\nu};{\cal L}_{\min}^{\infty}) + \displaystyle{
\frac{\rho}{2}
} \, \| \, x^{\nu+1/2} - x^{\nu} \, \|_2^2.
\]
At this point, the above proof can be applied to complete the proof of the desired
equality (\ref{eq:intermediate step for directional}).
\end{proof}

\section{Sequential Convergence Analysis} \label{sec:sequential convergence}

Our next goal in the convergence analysis of Algorithm~1 is to establish the
sequential convergence of the sequence $\{ x^{\nu} \}$ produced.  As expected,
specializations of the general setting are needed; these are captured by
the headings of Subsections~\ref{subsec:step sizes} and \ref{subsec:C2DC2 and KL}.
The overall development is based on the following general sequential convergence 
result with rates for an arbitrary sequence
without regard to its source.  This result is drawn 
from \cite[Theorem~8.4.3]{CuiPang2021} and is a summary of many results of this kind.   
Proof of the theorem can be found in the cited reference.

\begin{theorem} \label{th:sequence converges under KL} \rm
Let $\{ x^{\, \nu} \}$ be a bounded sequence in $\mathbb{R}^n$.  Suppose there 
exist a continuous function $f$ defined on an open set containing the sequence and 
a positive scalar $\gamma$ such that
\begin{equation} \label{eq:sufficient descent}
f(x^{\, \nu+1}) - f(x^{\, \nu}) \, \leq \, -\gamma \, 
\| \, x^{\, \nu+1} - x^{\, \nu} \, \|_2^2 \epc \forall \, \nu \, \geq \, 0.
\end{equation}
Then the sequence $\{ f(x^{\, \nu}) \}$ is monotonically nonincreasing and
converges to a limit $f_{\infty}$; moreover, either one of the following two statements
holds:

\gap

\noindent (a) $\{ x^{\nu} \}$ converges finitely, i.e.,
there exists $\bar{\nu}$ such that $x^{\, \nu} = x^{\bar{\nu}}$ for all
$\nu \geq \bar{\nu}$, or 

\gap

\noindent (b) $f(x^{\, \nu}) > f_{\infty}$ for all $\nu \geq 0$.  

\gap

\noindent In case (b), if 
additionally there exist a positive scalar $\eta$ and a nonnegative scalar $\zeta < 1$ 
such that
\begin{equation} \label{eq:KL with power}
\| \, x^{\, \nu+1} - x^{\, \nu} \, \|_2 \, \geq \, \displaystyle{
\frac{1}{\eta}
} \, \left[ \, f(x^{\, \nu+1}) 
- f_{\infty} \, \right]^{\zeta} \epc \forall \, \nu \mbox{ sufficiently large},
\end{equation}
then the following two statements hold:

\gap

\noindent $\bullet $ the sequence $\{ x^{\, \nu} \}$ converges to $x^{\, \infty}$, and

\gap

\noindent $\bullet $ the following convergence rates hold:

\gap

\noindent {\bf (i)} if $\zeta = 0$, the sequence $\{ x^{\, \nu} \, \}$ converges 
finitely; 

\gap

\noindent {\bf (ii)} if $\zeta \in (0, \frac{1}{2} \,]$, then there exist positive 
scalars $\wh{\eta}$ and $q \in [0,1)$ such that
\[
\| \, x^{\, \nu} - x^{\, \infty} \, \| \, \leq \, \wh{\eta} \, q^{\, \nu} \epc \forall \,
\nu \, \geq \, 0;
\] 
{\bf (iii)} if $\zeta\in (\,\frac{1}{2},1)$, then there exists a positive scalar 
$\wh{\eta}$ such that 
\[
\|\, x^\nu - x^{\, \infty} \, \| \, \leq \wh{\eta} \, \nu^{-(1-\zeta)/(2\zeta - 1)} 
\epc \forall \; \nu \, \geq \, 0.
\]
Thus the key in establishing the convergence of the sequence $\{ x^{\nu} \}$ lies
in the satisfaction of the two conditions (\ref{eq:sufficient descent}) and
(\ref{eq:KL with power}).  \hfill $\Box$
\end{theorem}

According to the theorem, there are two main verifications 
that are needed to establish the
sequential convergence of the sequence $\{ x^{\nu} \}$ produced by Algorithm~1;
namely, the inequality 
(\ref{eq:sufficient descent}), which is a sufficient descent condition with
a uniform constant $\gamma$, and (\ref{eq:KL with power}) which is closely
related to the subject of error bounds and the popular Kurdyka-Lojasiecwisz (KL)
inequality \cite[Subsection~8.4.3]{CuiPang2021}.  For an excellent exposition of
the KL inequality, see \cite{Lewis20} and the references therein, particularly
the original works \cite{AttouchBolte09,AttouchBolteSvaiter13}.  See also the
more recent study on the KL exponent and its implication on the linear convergence
of first-order methods for optimization problems with 
``KL functions'' \cite{LiPong18}. 

\subsection{Condition (\ref{eq:sufficient descent}) under LC$^{\, 1}$: 
Step sizes bounded away from zero} \label{subsec:step sizes}

In this subsection, we consider the problem (\ref{eq:unified framework}) where
the outer functions $\phi_j$ are LC$^{\, 1}$ (i.e., their gradients are Lipschitz
continuous) and the inner functions are the sum of a convex, a concave, and a 
differentiable function, all on an open convex set containing $X$:
\begin{equation} \label{eq:dc plus diff}
p_{jk}(x) \, = \, p_{jk}^{\rm cvx}(x) + p_{jk}^{\rm cve}(x) + p_{jk}^{\rm diff}(x),
\end{equation}
where $p_{jk}^{\rm diff}$ is also LC$^{\, 1}$ on the same domain as the
other functions, and $p_{jk}^{\rm cvx}$ and $p_{jk}^{\, \rm cve}$ 
are Lipschitz continuous on $\mbox{Lev}(x^0)$.  Thus the pairs $( \phi_j,P^{\, j} )$
are of type~I composition.  Moreover, the additional
Lipschitz conditions will then imply that
the sequence of step sizes $\{ \tau_{\nu} \}$
will be bounded away from zero.  With this sufficient positivity
property of the step sizes, the uniform descent condition (\ref{eq:sufficient descent})
can be proved readily.  We will subsequently use this preliminary result to
show that if the Lipschitz constants are available, they can be used to pre-select
the proximal regularization parameter $\rho$ so that a unit step size can be
achieved in the iterations, therefore eliminating the need for the line searches.

\begin{lemma} \label{lm:liminf step sizes} \rm
In the above setting 
suppose that the level set $\mbox{Lev}(x^0)$ is bounded.  It then holds that 
$\displaystyle{
\liminf_{\nu \to \infty}
} \, \tau_{\nu+1} > 0$.  Hence, for some $\gamma > 0$,
(\ref{eq:sufficient descent}) holds for $\Theta_{\max}$
for all $\nu$ sufficiently large. 
\end{lemma}

\begin{proof}  By way of contradiction, we assume that there is an infinite set
$\kappa$ such that $\displaystyle{
\lim_{\nu (\in \kappa) \to \infty}
} \, \tau_{\nu+1} = 0$.  Without loss of generality, we may assume that the
subsequence $\{ x^{\nu} \}_{\nu \in \kappa}$ converges to a limit $x^{\infty}$. 
We employ the fact that
for a LC$^{\, 1}$ function $f$, it holds that for some constant $\mbox{Lip}_{\nabla f} > 0$,
\begin{equation} \label{eq:LC1 inequality}
\left| \, f(u) - f(v) - \nabla f(v)^{\top}( u - v ) \, \right| \, \leq \, \displaystyle{
\frac{\mbox{Lip}_{\nabla f}}{2}
} \, \| \, u - v \, \|_2^2, \epc \forall \, u \mbox{ and } v,
\end{equation}
which yields, for some constant $\mbox{Lip}_f > 0$
\[
| \, f(u) - f(v) \, | \, \leq \, \mbox{Lip}_f \, \| \, u - v \, \|_2, \epc
\mbox{for all bounded $u$ and $v$}.
\]  
We proceed as in the proof of part (B) of
Theorem~\ref{th:subsequential convergence I}, following the steps in proving
the uniform upper approximation property in Proposition~\ref{pr:convergence of theta dd}. 
Using the same notation as in the former, let 
$\wh{x}^{\, \nu} = x^{\nu} + \displaystyle{
\frac{\tau_{\nu+1}}{\beta}
} \, ( x^{\nu + 1/2} - x^{\nu} )$ and let $j_{\nu}$ be
such that $\Theta_{\max}( \wh{x}^{\, \nu} ) = \theta_{j_{\nu}}( \wh{x}^{\, \nu} )$.
We apply the above Lipschitz inequalities 
to the functions $\phi_j$
and $p_{jk}^{\rm diff}$ and the bounded sequences $\{ \wh{x}^{\, \nu} \}_{\nu \in \kappa}$
and $\{ x^{\, \nu} \}_{\nu \in \kappa}$ with the goal of bounding
$\Theta_{\max}( \wh{x}^{\, \nu} ) - \Theta_{\max}(x^{\nu})$, which satisfy
the reverse Armijo inequality (\ref{eq:reverse Armijo}):
\[ 
\begin{array}{l}
-\displaystyle{
\frac{\sigma \, \rho}{2}
} \, \displaystyle{
\frac{\tau_{\nu+1}}{\beta}
} \, \| \, x^{\nu + 1/2} - x^{\, \nu} \, \|_2^2 \, < \, 
\Theta_{\max}( \wh{x}^{\, \nu} ) - \Theta_{\max}(x^{\nu}) \\ [0.15in]
\leq \, \theta_{j_{\nu}}( \wh{x}^{\, \nu} ) - \theta_{j_{\nu}}(x^{\nu})
\, = \, \phi_{j_{\nu}}( \wh{y}^{\, \nu} ) - \phi_{j_{\nu}}( y^{\nu} ), \epc \mbox{where } \  
\wh{y}^{\, \nu} = P^{j_{\nu}}(\wh{x}^{\, \nu}) \ \mbox{ and }  
y^{\nu} \triangleq P^{j_{\nu}}( x^{\nu}) \\ [0.15in]
= \, \nabla \phi_{j_{\nu}}(\wt{z}^{\, t;\nu})^{\top}
( \wh{y}^{\, \nu} - y^{\nu} ),
\epc \mbox{where } \ \wt{z}^{\, t;\nu} \, \triangleq \, y^{\nu} + 
t_{\nu} ( \wh{y}^{\, \nu} - y^{\nu} ) \ \mbox{ for some } 
t_{\nu} \in ( \, 0,1 \, ) \\ [0.2in]
= \, \underbrace{\left[ \, \nabla \phi_{j_{\nu}}(\wt{z}^{\, t;\nu}) - 
\nabla \phi_{j_{\nu}}(y^{\nu})^{\top} \, \right]^{\top}( \wh{y}^{\, \nu} - y^{\nu} )}_{
\mbox{denoted $T1$}} + 
\underbrace{\nabla \phi_{j_{\nu}}(y^{\nu})^{\top}( \wh{y}^{\, \nu} - y^{\nu} )}_{
\mbox{denoted $T2$}}. 
\end{array} \]
By the Lipschitz continuity of $\nabla \phi_{j_{\nu}}$ and $P^{j_{\nu}}$ with moduli
$\mbox{Lip}_{\nabla \phi}$ and $\mbox{Lip}_P$, respectively, we have
\[
T1 \, \leq \, \mbox{Lip}_{\nabla \phi} \, \| \, \wh{y}^{\, \nu} - y^{\nu} \, \|_2^2 
\, \leq \, \mbox{Lip}_{\nabla \phi} \, \mbox{Lip}_P^2 \, \left( \, \displaystyle{
\frac{\tau_{\nu+1}}{\beta}
} \, \right)^2 \, \| \, x^{\nu + 1/2} - x^{\nu} \, \|_2^2;
\]
and
\[ \begin{array}{l}
\mbox{T2} \, = \,	\displaystyle{
\sum_{k \in {\cal K}_{\phi_{j_{\nu}}}^{\, +}(y^{\nu})}
} \, \displaystyle{
\frac{\partial \phi_{j_{\nu}}(y^{\nu})}{\partial y_k}
} \, \left[ \, p_{j_{\nu}k}^{\rm cvx}(\wh{x}^{\, \nu}) - p_{j_{\nu}k}^{\rm cvx}(x^{\nu}) + 
p_{j_{\nu}k}^{\rm cve}(\wh{x}^{\, \nu}) - p_{j_{\nu}k}^{\rm cve}(x^{\nu}) +
p_{j_{\nu}k}^{\rm diff}(\wh{x}^{\, \nu}) - p_{j_{\nu}k}^{\rm diff}(x^{\nu}) 
\, \right] \\ [0.25in]
\hspace{0.2in} + \,  \displaystyle{
\sum_{k \in {\cal K}_{\phi_{j_{\nu}}}^{\, -}(y^{\nu})}
} \, \displaystyle{
\frac{\partial \phi_{j_{\nu}}(y^{\nu})}{\partial y_k}
} \, \left[ \, p_{j_{\nu}k}^{\rm cvx}(\wh{x}^{\, \nu}) - p_{j_{\nu}k}^{\rm cvx}(x^{\nu}) + 
p_{j_{\nu}k}^{\rm cve}(\wh{x}^{\, \nu}) - p_{j_{\nu}k}^{\rm cve}(x^{\nu}) +
p_{j_{\nu}k}^{\rm diff}(\wh{x}^{\, \nu}) - p_{j_{\nu}k}^{\rm diff}(x^{\nu}) 
\, \right] \\ [0.3in]
\leq \, \displaystyle{
\frac{\tau_{\nu+1}}{\beta}
} \, \left\{ \, \begin{array}{l}
\displaystyle{
\sum_{k \in {\cal K}_{\phi_{j_{\nu}}}^{\, +}(y^{\nu})}
} \, \displaystyle{
\frac{\partial \phi_{j_{\nu}}(y^{\nu})}{\partial y_k}
} \, \left[ \, \begin{array}{l}
p_{j_{\nu}k}^{\rm cvx}(x^{\nu+1/2}) - p_{j_{\nu}k}^{\rm cvx}(x^{\nu}) 
- ( b^{\, \nu;k} )^{\top}( x^{\nu+1/2} - \bar{x}^{\nu} ) \, + \\ [0.1in] 
\nabla p_{j_{\nu}k}^{\rm diff}(x^{\nu})^{\top}( x^{\nu+1/2} - \bar{x}^{\nu} ) 
\end{array} \right] \, + \\ [0.35in]
\displaystyle{
\sum_{k \in {\cal K}_{\phi_{j_{\nu}}}^{\, -}(y^{\nu})}
} \, \displaystyle{
\frac{\partial \phi_{j_{\nu}}(y^{\nu})}{\partial y_k}
} \, \left[ \, \begin{array}{l}
( a^{\nu;k} )^{\top}( x^{\nu+1/2} - \bar{x}^{\nu} ) +
p_{j_{\nu}k}^{\rm cve}(x^{\nu+1/2}) - p_{j_{\nu}k}^{\rm cve}(\bar{x}^{\nu}) \, + \\ [0.1in]
\nabla p_{j_{\nu}k}^{\rm diff}(\bar{x}^{\nu})^{\top}( x^{\nu+1/2} - x^{\nu} ) 
\end{array} \right] 
\end{array} \right\} \\ [0.65in]
\epc (\mbox{by convexity of $p_{j_{\nu}k}^{\rm cvx}$ and concavity of 
$p_{j_{\nu}k}^{\rm cve}$; where
$( a^{\nu;k},b^{\nu;k} ) \in \partial p_{j_{\nu}k}^{\rm cvx}(x^{\nu}) \times 
\partial (-p_{j_{\nu}k}^{\rm cve})(x^{\nu})$}) \\ [0.1in]
\hspace{0.4in} + \ \displaystyle{
\sum_{k =1}^K
} \, \underbrace{\displaystyle{
\frac{\partial \phi_{j_{\nu}}(y^{\nu})}{\partial y_k}
}}_{\mbox{bounded by $B_{\nabla \phi}$}}
\, \left[ \, p_{j_{\nu}k}^{\rm diff}(\wh{x}^{\, \nu}) - p_{j_{\nu}k}^{\rm diff}(x^{\nu}) -
\displaystyle{
\frac{\tau_{\nu+1}}{\beta}
} \, \nabla p_{j_{\nu}k}^{\rm diff}(x^{\nu})^{\top}( x^{\nu+1/2} - x^{\nu} ) \, \right]
\\ [0.2in]
\leq - \displaystyle{
\frac{\tau_{\nu+1}}{\beta}
} \, \displaystyle{
\frac{\rho}{2}
} \, \| \, x^{\nu+1/2} - x^{\nu} \, \|_2^2 + B_{\nabla \phi} \, \mbox{Lip}_{\nabla \phi} \, 
\left( \, \displaystyle{
\frac{\tau_{\nu+1}}{\beta}
} \, \right)^2 \, \| \, x^{\nu+1/2} - x^{\nu} \, \|_2^2.
\end{array}
\]
Putting the bounds for $T1$ and $T2$ together, we deduce
\[
-\displaystyle{
\frac{\sigma \, \rho}{2}
} \, \displaystyle{
\frac{\tau_{\nu+1}}{\beta}
} \, \| \, x^{\nu + 1/2} - x^{\, \nu} \, \|_2^2 \, < \,  - \displaystyle{
\frac{\tau_{\nu+1}}{\beta}
} \, \displaystyle{
\frac{\rho}{2}
} \, \| \, x^{\nu+1/2} - x^{\nu} \, \|_2^2 + C \, 
\left( \, \displaystyle{
\frac{\tau_{\nu+1}}{\beta}
} \, \right)^2 \, \| \, x^{\nu+1/2} - x^{\nu} \, \|_2^2,
\]
where $C > 0$ is a constant that depends on 
$B_{\nabla \phi}$, $\mbox{Lip}_{\nabla \phi}$, and $\mbox{Lip}_P$.  Hence it follows
that,
\[ 
0 \, \leq \, -\displaystyle{
\frac{\tau_{\nu+1}}{\beta}
} \, \left[ \, ( \, 1 - \sigma \, ) \, \displaystyle{
\frac{\rho}{2}
} \, + C \, \left( \, \displaystyle{
\frac{\tau_{\nu+1}}{\beta}
} \, \right) \, \right] \| \, x^{\, \nu + 1/2} - x^{\nu} \, \|_2^2,
\]
which yields 
\begin{equation} \label{eq:key inequality for convergence}
( \, 1 - \sigma \, ) \, \displaystyle{
\frac{\rho}{2}
} \, - C \, \left( \, \displaystyle{
\frac{\tau_{\nu+1}}{\beta}
} \, \right) \leq 0.
\end{equation}  
Passing to the limit $\nu (\in \kappa) \to 0$ yields a contradiction.  

\gap

The proof of the last statement of the lemma is as follows.  By the forward 
Armijo descent condition (\ref{eq:Armijo descent}), since $\displaystyle{
\liminf_{\nu \to \infty}
} \, \tau_{\nu+1} > 0$. we have, 
\[ \begin{array}{lll}
\Theta_{\max}(x^{\nu+1}) - \Theta_{\max}(x^{\nu}) 
& \leq & -\sigma \, \tau_{\nu+1} \, \displaystyle{
\frac{\rho}{2}
} \, \| \, x^{\nu+1/2} - x^{\nu} \, \|_2^2 \\ [0.15in]
& = & -\displaystyle{
\frac{\sigma}{\tau_{\nu+1}}
} \, \displaystyle{
\frac{\rho}{2}
} \, \| \, x^{\nu+1} - x^{\nu} \, \|_2^2,
\end{array} \]
where the last inequality is justified because $\tau_{\nu+1} > 0$
for all $\nu$ sufficiently large.   It therefore follows that,
 since $\tau_{\nu+1} \leq 1$,
\[
\Theta_{\max}(x^{\nu+1}) - \Theta_{\max}(x^{\nu}) \, \leq \, -\displaystyle{
\frac{\sigma \, \rho}{2}
} \, \| \, x^{\nu+1} - x^{\nu} \, \|_2^2.
\]
Thus (\ref{eq:sufficient descent}) holds with $\gamma = \sigma \, \rho/2$.
\end{proof}

In the above proof, the key inequality that leads to the 
final contradiction is (\ref{eq:key inequality for convergence}),
where the constant $C$ is most essential.  Provided that the Lipschitz constants
$\mbox{Lip}_{\nabla \phi}$, and $\mbox{Lip}_P$ and the
bounds $B_{\nabla \phi}$ are valid on the set $X$, it follows that the constant $C$ applies to 
all subsequences of $\{ x^{\nu} \}$.
Instead of passing $\nu$ to the limit subsequentially, the inequality
(\ref{eq:key inequality for convergence}) can be applied in an alternative way
to provide a choice on the proximal parameter $\rho$ so that a unit step size
can always be chosen.  

\begin{theorem} \label{th:unit step} \rm
Let $X$ be a closed convex set in $\mathbb{R}^n$.  Let each function $p_{jk}$ be
given by (\ref{eq:dc plus diff}) with each $\phi_j$ and
$p_{jk}^{\rm diff}$ being LC$^{\, 1}$ functions on 
$X$ where $p_{jk}^{\rm cvx}$ and $p _{jk}^{\, \rm cve}$ are Lipschitz 
continuous.  Suppose that the level set $\mbox{Lev}(x^0)$ is bounded.  
Then there exists a constant $C > 0$ such that if
$\{ x^{\nu} \}$ is a sequence for which $x^{\nu+1}$ is the unique global 
minimizer of the subproblem (\ref{eq:workhorse}) at $x^{\nu}$ with 
$\rho > 2 \, C$, then for every $\nu$, the following three statements hold:

\gap

{\bf (A)} $\Theta_{\max}(x^{\nu+1}) - \Theta_{\max}(x^{\nu}) \leq \
-\left( \, \displaystyle{
\frac{\rho}{2}
} - C \, \right) \, \| \, x^{\nu+1} - x^{\nu} \, \|_2^2$ for all $\nu$; and

\gap

{\bf (B)} the sequence $\{ \Theta_{\max}(x^{\nu}) \}$ converges monotonically
to a limit, say $\Theta_{\max}^{\, \infty}$.

\gap

{\bf (C)} every accumulation point of $\{ x^{\nu} \}$ 
is a weak directional stationary solution of (\ref{eq:unified framework}).
\end{theorem}

\begin{proof} Statement (A) follows from the proof 
of Lemma~\ref{lm:liminf step sizes}.  Statement (B) follows from the first
part of Theorem~\ref{th:sequence converges under KL}.  Statement (C) follows from
the proof of case (A) in Theorem ~\ref{th:subsequential convergence I}.
\end{proof}

\subsection{A detour: Subdifferential formulas}

To prepare for the continued analysis of Theorem~\ref{th:unit step} and complete
the proof of sequential convergence, we make a detour of the discussion and 
summarize some important properties of the (generalized) subdifferentials of the objective
function $\Theta_{\max}$ satisfying the assumptions of the theorem.  
For a locally Lipschitz continuous funciton $g : \mathbb{R}^n \to \mathbb{R}$,
the {\sl limiting subdifferential} of $g$ at $\bar{x}$ \cite{RockafellarWets98}
defined as
\[
\partial_{\, \rm L} g(\bar{x}) \triangleq 
\left\{ v \in \mathbb{R}^n \, \bigg| 
\, \exists \, \{ x^k \} \to \bar{x}
\mbox{ and } \{ v^k \} \to v \mbox{ such that } v^k\! \in \wh{\partial} g(x^k)
\mbox{ for all $k$ } \right\}
\]
is always nonempty and compact.
Between the two subdifferentials, $\wh{\partial} g$ versus $\partial_{\, \rm L} g$,
elements of the former are easier to characterize; we have 
$\wh{\partial} g(\bar{x}) \subseteq \partial_{\, \rm L} g(\bar{x})$, which implies 
\begin{equation} \label{eq:distance inequality}
\mbox{dist}_2(z,\wh{\partial} g(\bar{x})) \geq 
\mbox{dist}_2(z,\partial_{\, \rm L} g(\bar{x})), \epc \mbox{for all $z \in \mathbb{R}^n$},
\end{equation}
where $\mbox{dist}_2(z,S) \triangleq \displaystyle{ 
\inf_{x \in S}
} \, \| \, x - x \, \|_2$ is the Euclidean distance from the point $z$ to the
closed set $S$.  As mentioned before, the regular subdifferential 
$\wh{\partial} g(\bar{x})$ may be empty in general; nevertheless if $g$ is dd-convex, 
then $\wh{\partial} g(\bar{x})$
must be nonempty \cite[Proposition~4.3.3]{CuiPang2021}. 
The two references \cite[Section~2]{BotDaoLi22,BotLiTao24} have presented an excellent 
summary of the variational calculus of fractional functions.  The distance inequality
(\ref{eq:distance inequality}) provides a strong motivation to understand the
regular subdifferential $\wh{\partial} \Theta_{\max}(\bar{x})$ in place of the
limiting subdifferential $\partial_{\, \rm L} \Theta_{\max}(\bar{x})$; as a composite
function, \cite[Theorem~4.5]{Mordukhovich18} gives an inclusion relation for 
the latter function in terms of the outer and inner functions.  For our purpose,
we derive a chain-rule inclusion for $\wh{\partial} \Theta_{\max}(\bar{x})$ that 
can be directly applied to establish the sequential convergence of Algorithm~1.
In the rest of the paper, we assume that the pairs $( \phi_j,P^j )$ are of the
type~I composition, with each inner function $p_{jk}$ given by (\ref{eq:dc plus diff}).

\gap

With $\phi_j$ being differentiable, under either the (AC$^{\, 2}$) condition or the 
(C$^{\, 2}$DC$^{\, 2}$) assumption at a reference vector $\bar{x}$, each 
function $\theta_j = \phi_j \circ P^j$ is dd-convex at $\bar{x}$; thus so
is $\Theta_{\max} = \displaystyle{
\max_{1 \leq j \leq J}
} \, \theta_j$ at $\bar{x}$.  Hence it follows that
\begin{equation} \label{eq:subdifferential Theta}
\partial_{\, \rm L} \Theta_{\max}(\bar{x}) \, \supseteq \, 
\wh{\partial} \Theta_{\max}(\bar{x}) \, \supseteq \, \mbox{conv} \, \left[ \, \displaystyle{
\bigcup_{j \in {\cal M}_{\Theta}(\bar{x})}
} \, \wh{\partial} \theta_j(\bar{x}) \right],
\end{equation}
where ``conv'' denotes the convex hull of; the above inclusion holds because the
right-hand union without the convex hull can easily be seen to be a subset of the
regular subdifferential and also because the latter subdifferential is a convex set.
The key in the remaining analysis is to provide an expression for each
$\wh{\partial} \theta_j(\bar{x})$.  To simplify the discussion, we complete the analysis
assuming the 
(C$^{\, 2}$DC$^{\, 2}$) condition holds at all reference vectors $\bar{x} \in X$
and omit the analysis under the (AC$^{\, 2}$) condition which should be similar.  

\gap

For the functions $p_{jk}$ given by (\ref{eq:dc plus diff}), the
(C$^{\, 2}$DC$^{\, 2}$) condition at $\bar{x}$ postulates that, 
with $\bar{y}^j \triangleq P^j(\bar{x})$,

\gap

$\bullet $ $\displaystyle{
\frac{\partial \phi_j(\bar{y}^j)}{\partial y_k}
} > 0$ implies that $p_{jk}^{\rm cve}$ is strictly differentiable at $\bar{x}$; and 

\gap

$\bullet $ $\displaystyle{
\frac{\partial \phi_j(\bar{y}^j)}{\partial y_k}
} < 0$ implies that $p_{jk}^{\rm cvx}$ is strictly differentiable at $\bar{x}$.

\gap

Under this condition, we have
\[ \begin{array}{lll}
\theta_j^{\, \prime}(\bar{x};v) & = & \left\{ \begin{array}{l}
\displaystyle{
\sum_{k \in {\cal K}_{\phi_j}^+(\bar{y}^j)}
} \, \displaystyle{
\frac{\partial \phi_j(\bar{y}^j)}{\partial y_k}
} \, \left[ \, ( p_{jk}^{\rm cvx} )^{\, \prime}(\bar{x};v) + 
\nabla p_{jk}^{\rm cve}(\bar{x})^{\top}v + \nabla p_k^{\rm diff}(\bar{x})^{\top}v \, \right]
\, + \\ [0.2in]
\displaystyle{
\sum_{k \in {\cal K}_{\phi_j}^-(\bar{y}^j)}
} \, \displaystyle{
\frac{\partial \phi_j(\bar{y}^j)}{\partial y_k}
} \, \left[ \, \nabla p_{jk}^{\rm cvx}(\bar{x})^{\top}v + 
( p_{jk}^{\rm cve} )^{\, \prime}(\bar{x};v) + 
\nabla p_{jk}^{\rm diff}(\bar{x})^{\top}v \, \right]
\end{array} \right\} \\ [0.5in]
& = &  \left\{ \begin{array}{l}
\displaystyle{
\sum_{k \in {\cal K}_{\phi_j}^+(\bar{y}^j)}
} \, \displaystyle{
\frac{\partial \phi_j(\bar{y}^j)}{\partial y_k}
} \, \left[ \, \displaystyle{
\max_{a \in \partial (p_{jk}^{\rm cvx})(\bar{x})}
} \, a^{\top}v + \nabla p_{jk}^{\rm cve}(\bar{x})^{\top}v + 
\nabla p_{jk}^{\rm diff}(\bar{x})^{\top}v \, \right] \, + \\ [0.2in]
\displaystyle{
\sum_{k \in {\cal K}_{\phi_j}^-(\bar{y}^j)}
} \, \displaystyle{
\frac{\partial \phi_j(\bar{y}^j)}{\partial y_k}
} \, \left[ \, \nabla p_{jk}^{\rm cvx} )(\bar{x})^{\top}v - \displaystyle{
\max_{b \in \partial (-p_{jk}^{\rm cve})(\bar{x})}
} \, b^{\top}v + \nabla p_{jk}^{\rm diff}(\bar{x})^{\top}v \, \right]
\end{array} \right\},
\end{array} \]
For each index $j$, define the convex function 
$\Gamma_j(\bullet;\bar{x})$ by
\[ \begin{array}{lll}
\Gamma_j(x;\bar{x}) & \triangleq & \displaystyle{
\sum_{k \in {\cal K}_{\phi_j}^+(\bar{y}^j)}
} \, \displaystyle{
\frac{\partial \phi_j(\bar{y}^j)}{\partial y_k}
} \, \left[ \, p_{jk}^{\rm cvx}(x) - p_{jk}^{\rm cvx}(\bar{x}) + 
\left( \, \nabla p_{jk}^{\rm cve}(\bar{x}) +  \nabla p_{jk}^{\rm diff}(\bar{x}) \, \right)^{\top}
( x - \bar{x}) \, \right] \, + \\ [0.2in]
& & \displaystyle{
\sum_{k \in {\cal K}_{\phi_j}^-(\bar{y}^j)}
} \, \displaystyle{
\frac{\partial \phi_j(\bar{y}^j)}{\partial y_k}
} \, \left[ \, \left( \, 
\nabla p_{jk}^{\rm cvx}(\bar{x}) +  \nabla p_{jk}^{\rm diff}(\bar{x}) \, \right)^{\top}
( x - \bar{x}) + p_{jk}^{\rm cve}(x) - p_{jk}^{\rm cve}(\bar{x}) \, \right],
\end{array} \]
which is the sum of the following family of convex functions:
%
%
\[
\left\{ \, p_{jk}^{\rm cvx} \, : \, k \in {\cal K}_{\phi_j}^+(\bar{y}^j) \, \right\} 
\, \bigcup \,
\left\{ \, -p_{jk}^{\rm cve} \, : \, k \in {\cal K}_{\phi_j}^-(\bar{y}^j) \, \right\}
\]
plus some affine functions.  Thus, we have
for an arbitrary vector $\wh{x}$, by \cite[Theorem~23.8]{Rockafellar70},
\begin{equation} \label{eq:partial of C} 
\begin{array}{lll}
[ \, \partial \Gamma_j(\bullet;\bar{x}) \, ]( \wh{x} ) & = & 
\displaystyle{
\sum_{k \in {\cal K}_{\phi_j}^+(\bar{y}^j)}
} \, \displaystyle{
\frac{\partial \phi_j(\bar{y}^j)}{\partial y_k}
} \, \left[ \, \partial p_{jk}^{\rm cvx}( \wh{x} ) + \nabla p_{jk}^{\rm cve}(\bar{x}) +  
\nabla p_{jk}^{\rm diff}(\bar{x}) \, \right] \, + \\ [0.2in]
& & \displaystyle{
\sum_{k \in {\cal K}_{\phi_j}^-(\bar{y}^j)}
} \, \displaystyle{
\frac{\partial \phi_j(\bar{y}^j)}{\partial y_k}
} \, \left[ \, \nabla p_{jk}^{\rm cvx}(\bar{x}) +  \nabla p_{jk}^{\rm diff}(\bar{x}) - 
\partial (-p_{jk}^{\rm cve}) ( \wh{x} ) \, \right].
\end{array} \end{equation}
Hence, since $\theta_j$ is dd-convex, it follows that
\[
\theta_j^{\, \prime}(\bar{x};v) \, = \, \displaystyle{
\max_{a \in \wh{\partial} \theta_j(\bar{x})}
} \, a^{\top} v \, = \, \displaystyle{
\max_{a \in [ \, \partial \Gamma_j(\bullet;\bar{x}) \, ]( \bar{x} )}
} \, a^{\, \top}v.
\]
By \cite[Proposition~2.1.4]{Clarke90}, we deduce
\begin{equation} \label{eq:equality of two subdiff}
\wh{\partial} \theta_j(\bar{x}) = 
\left[ \, \partial \Gamma_j(\bullet;\bar{x}) \, \right]( \bar{x} ).
\end{equation}

\subsection{Condition (\ref{eq:KL with power}) under (C$^{\, 2}$DC$^{\, 2}$):
Application of KL} \label{subsec:C2DC2 and KL}

Let $\{ x^{\nu} \}$ be a sequence satisfying the assumptions of Theorem~\ref{th:unit step}. 
We wish to show that this sequence satisfies the inequality (\ref{eq:KL with power})
for the function $\Theta_{\max}$ under the uniform KL postulate.  To introduce this
postulate, we let $\Omega$ be the (nonempty) set of accumulation points of the 
sequence $\{ x^{\nu} \}$.
By part (B) of Theorem~\ref{th:unit step}, it follows that $\Theta_{\max}$ is a constant
on $\Omega$, which we denote $\Theta_{\infty}$.
  
\gap

$\bullet $ {\bf (Uniform KL property)} the extended-valued function 
$\Theta_{\max} + \delta_X$, where $\delta_X$ is the indicator function 
of the closed convex set $X$, has a
{\sl power desingularization function} (a terminology due 
to \cite{AttouchBolte09,AttouchBolteSvaiter13}) $\psi$ that is applicable to 
every accumulation point in $\Omega$; i.e., with $\psi(s) = \eta \, s^{1 - \alpha}$ 
defined for for $s \geq 0$, 
where $\eta > 0$ and $\alpha \in [ \, 0,1 \, )$ are some scalars, it holds that
for every accumulation point $x^{\infty}$ of the sequence $\{ x^{\nu} \}$,
there exists a neighborhood ${\cal N}$ of $x^{\infty}$ and a positive scalar
$\zeta > 0$ such that for all $x \in {\cal N}$ with 
$\Theta_{\max}(x^{\infty}) < \Theta_{\infty} < 
\Theta_{\max}(x^{\infty}) + \zeta$, the following inequality holds:
\[
\psi^{\, \prime}\left( \Theta_{\max}(x) - \Theta_{\infty} \right) \, 
\mbox{dist}\left( \, 0, \partial_{\, \rm L} \Theta_{\max}(x) + {\cal N}(x;X) \, \right)
\, \geq \, 1.
\]
It is shown in \cite[Lemma~2.3]{BotDaoLi22} that this uniform KL property is 
implied by the pointwise version.  
With $\psi$ being the given power function, the above inequality yields
\[
\left[ \, \Theta_{\max}(x) - \Theta_{\infty}) \, \right]^{\alpha} 
\, \leq \, \eta^{\, \prime} \, 
\mbox{dist}\left( \, 0, \partial_{\, \rm L} \Theta_{\max}(x) + {\cal N}(x;X) \, \right),
\epc \mbox{where } \ \eta^{\, \prime} \, \triangleq \, ( \, 1 - \alpha \, ) \, \eta.
\]
The key is to upper bound the distance on the right-hand side for $x = x^{\nu+1}$.
By (\ref{eq:subdifferential Theta}), we have
\begin{equation} \label{eq:distance bound}
\mbox{dist}\left( \, 0, \, \partial_{\, \rm L} \Theta_{\max}(x^{\nu+1}) 
+ {\cal N}(x^{\nu+1};X) \, \right) \, \leq \, \mbox{dist}\left( \, 0, \,
\mbox{conv} \, \left[ \, \displaystyle{
\bigcup_{j \in {\cal M}_{\Theta}(x^{\nu+1})}
} \, \wh{\partial} \theta_j(x^{\nu+1}) \right] \, \right).
\end{equation}
Under the (C$^{\, 2}$DC$^{\, 2}$) and the all-but-one-differentiability assumptions,
we may combine (\ref{eq:partial of C}) and (\ref{eq:equality of two subdiff}) to obtain
(where $y^{j;\nu+1} \triangleq P^j(x^{\nu+1})$):
\[ \begin{array}{lll}
\wh{\partial} \theta_j(x^{\nu+1})  & = & 
\displaystyle{
\sum_{k \in {\cal K}_{\phi_j}^+(y^{j;\nu+1})}
} \, \displaystyle{
\frac{\partial \phi_j(y^{j;\nu+1})}{\partial y_k}
} \, \left[ \, \partial p_{jk}^{\rm cvx}(x^{\nu+1}) + 
\nabla p_{jk}^{\rm cve}(x^{\nu+1}) +  
\nabla p_{jk}^{\rm diff}(x^{\nu+1}) \, \right] \, + \\ [0.2in]
& & \displaystyle{
\sum_{k \in {\cal K}_{\phi_j}^-(y^{j;\nu+1})}
} \, \displaystyle{
\frac{\partial \phi_j(y^{j;\nu+1})}{\partial y_k}
} \, \left[ \, \nabla p_{jk}^{\rm cvx}(x^{\nu+1}) +  
\nabla p_{jk}^{\rm diff}(x^{\nu+1}) - 
\partial (-p_{jk}^{\rm cve}) (x^{\nu+1}) \, \right]
\end{array} \]
By definition, $x^{\nu+1}$ is the unique optimal solution of the convex
program 
\begin{equation} \label{eq:final subproblem}
\displaystyle{
\operatornamewithlimits{\mbox{\bf minimize}}_{x \in X}
} \ \displaystyle{
\max_{j \in {\cal M}_{\Theta}^{\, \varepsilon_{\nu}}(x^{\nu})}
} \, \wh{\theta}_j(x;x^{\nu}) + \displaystyle{
\frac{\rho}{2}
} \, \| \, x - x^{\nu} \, \|_2^2, 
\end{equation}
where, with $y^{j;\nu} \triangleq P^j(x^{\nu})$,
\begin{equation} \label{eq:final subproblem theta} 
\begin{array}{l}
\wh{\theta}_j(x;x^{\nu}) \, = \, \theta_j(x^{\nu}) \, + \\ [0.1in]
\epc \displaystyle{
\sum_{k \in {\cal K}_{\phi_j}^{\, +}(y^{j;\nu})}
} \, \displaystyle{
\frac{\partial \phi_j(y^{j;\nu})}{\partial y_k}
} \, \left[ \,
p_{jk}^{\rm cvx}(x) - p_{jk}^{\rm cvx}(x^{\nu}) + 
\left[ \, \nabla p_{jk}^{\rm cve}(x^{\nu}) + 
\nabla p_{jk}^{\rm diff}(x^{\nu}) \, \right]^{\top}(x - x^{\nu}) \, \right]
\, + \\ [0.3in]
\epc \displaystyle{
\sum_{k \in {\cal K}_{\phi_j}^{\, -}(y^{j;\nu})}
} \, \displaystyle{
\frac{\partial \phi_j(y^{j;\nu})}{\partial y_k}
} \, \left[ \, p_{jk}^{\rm cve}(x) - p_{jk}^{\rm cve}(\bar{x} ) + 
\left[ \, \nabla p_{jk}^{\rm cvx}(x^{\nu} +
\nabla p_{jk}^{\rm diff}(x^{\nu}) \, \right]^{\top}( x - x^{\nu}) \, \right].
\end{array} \end{equation}
Thus there exist a subgradient 
$c^{\, \nu} \in \partial \left( \displaystyle{
\max_{j \in {\cal M}_{\Theta}^{\varepsilon_{\nu}}(x^{\nu})}
} \, \wh{\theta}_j(\bullet;x^{\nu}) \right)(x^{\nu+1})$ and
a normal vector $w^{\nu}$ in ${\cal N}(x^{\nu+1};X)$ such that 
\begin{equation} \label{eq:stationarity of subproblem}
0 \, = \, c^{\, \nu} + \rho \, ( x^{\nu+1} - x^{\nu} ) + w^{\nu}.
\end{equation}
As in the proof of Theorem~\ref{th:subsequential convergence I}, there exists a subset
${\cal M}_{\nu+1/2} \subseteq {\cal M}_{\Theta}^{\varepsilon_{\nu}}(x^{\nu})$ such that
$c^{\nu}$ belongs to the convex hull of the union $\displaystyle{
\bigcup_{j \in {\cal M}_{\nu+1/2}}
} \, \left[ \, \partial \wh{\theta}_j(\bullet;x^{\nu}) \, \right](x^{\nu + 1})$.
We have
\[ \begin{array}{lll}
\left[ \, \partial \wh{\theta}_j(\bullet;x^{\nu}) \, \right](x^{\nu + 1}) & = & 
\displaystyle{
\sum_{k \in {\cal K}_{\phi_j}^+(y^{j;\nu})}
} \, \displaystyle{
\frac{\partial \phi_j(y^{j;\nu})}{\partial y_k}
} \, \left[ \, \partial p_{jk}^{\rm cvx}( x^{\nu+1} ) + \nabla p_{jk}^{\rm cve}(x^{\nu}) +  
\nabla p_{jk}^{\rm diff}(x^{\nu}) \, \right] \, + \\ [0.25in]
& & \displaystyle{
\sum_{k \in {\cal K}_{\phi_j}^-(y^{j;\nu})}
} \, \displaystyle{
\frac{\partial \phi_j(y^{j;\nu})}{\partial y_k}
} \, \left[ \, \nabla p_{jk}^{\rm cvx}(x^{\nu}) +  \nabla p_{jk}^{\rm diff}(x^{\nu}) - 
\partial (-p_{jk}^{\rm cve}) (x^{\nu+1} ) \, \right] 
\end{array} \]
We compare the members of 
$\left[ \, \partial \wh{\theta}_j(\bullet;x^{\nu}) \, \right](x^{\nu + 1})$
with those in
$\wh{\partial} \theta_j(x^{\nu+1}) = 
\left[ \, \partial \Gamma_j(\bullet;x^{\nu+1}) \, \right]( x^{\nu+1} )$;
cf.\ (\ref{eq:partial of C}).  For this purpose, we introduce the following

\gap

$\bullet $ {\bf (locally uni-signed partial derivatives)} at every 
$x^{\infty} \in \Omega$: with $y^{j;\infty} \triangleq P^j(x^{\infty})$ for all 
$j \in [ J ]$, the following implication holds for all pairs 
$(j,k) \in {\cal M}_{\Theta}^{\, \varepsilon_{\infty}}(x^{\infty}) \times [ K ]$,
\[ \displaystyle{
\frac{\partial \phi_j(y^{j;\infty})}{\partial y_k}
} \, = \, 0 \ \Rightarrow \ \left\{ \begin{array}{l}
\exists \ \mbox{a neighborhood ${\cal N}$ of $x^{\infty}$ such 
that} \\ [0.1in]
\displaystyle{
\frac{\partial \phi_j(P^j(x))}{\partial y_k}
} \mbox{ is of one sign (possibly zero) for all $x \in {\cal N}$}.
\end{array} \right.
\] 
This assumption implies that for every $x^{\infty} \in \Omega$, there exists a 
neighborhood ${\cal N}$ of $x^{\infty}$, such that for every $j$, there exist 
two index sets
${\cal K}_j^{\oplus}(x^{\infty})$ and ${\cal K}_j^{\ominus}(x^{\infty})$
such that for any $x \in {\cal N}$ with 
$y^j \triangleq P^j(x)$, we have
\[
{\cal K}_{\phi_j}^+(y) \, \cup \, \left\{ \, k \, \mid \, \displaystyle{
\frac{\partial \phi_j(y)}{\partial y_k}
} \, = \, 0 \, \right\} \, = \, {\cal K}_j^{\oplus}(x^{\infty}) \epc \mbox{and} \epc
{\cal K}_{\phi_j}^-(y) \, \cup \, \left\{ \, k \, \mid \, \displaystyle{
\frac{\partial \phi_j(y)}{\partial y_k}
} \, = \, 0 \, \right\} \, = \, {\cal K}_j^{\ominus}(x^{\infty}) 
\]  
For the sequence $\{ x^{\nu} \}$, it was shown in \cite[Lemma!~2.3]{BotDaoLi22}
that $\displaystyle{
\lim_{\nu \to \infty}
} \, \mbox{dist}(x^{\nu},\Omega) = 0$.  Moreover, since $\displaystyle{
\lim_{\nu \to \infty}
} \, \| \, x^{\nu+1} - x^{\nu} \, \| = 0$, it follows that for all $\nu$
sufficiently large, $x^{\nu}$ and $x^{\nu+1}$ belong to the same neighborhood
of some accumulation point.  Hence, we have 
\[
{\cal K}_{\phi_j}^{\pm}(y^{j;\nu+1}) \, \cup \, \left\{ \, k \, \mid \, \displaystyle{
\frac{\partial \phi_j(y^{j;\nu+1})}{\partial y_k}
} \, = \, 0 \, \right\} \, = \, 
{\cal K}_{\phi_j}^{\pm}(y^{j;\nu}) \, \cup \, \left\{ \, k \, \mid \, \displaystyle{
\frac{\partial \phi_j(y^{j;\nu})}{\partial y_k}
} \, = \, 0 \, \right\} 
\]
respectively.  Therefore, with an abuse of notation, we may write
\[ \begin{array}{l}
\displaystyle{
\sum_{k \in {\cal K}_{\phi_j}^+(y^{j;\nu+1})}
} \, \displaystyle{
\frac{\partial \phi_j(y^{j;\nu+1})}{\partial y_k}
} \, \left[ \, \partial p_{jk}^{\rm cvx}( x^{\nu+1} ) + \nabla p_{jk}^{\rm cve}(x^{\nu}) +  
\nabla p_{jk}^{\rm diff}(x^{\nu}) \, \right] \\ [0.2in]
\epc + \, \displaystyle{
\sum_{k \in {\cal K}_{\phi_j}^-(y^{\nu+1})}
} \, \displaystyle{
\frac{\partial \phi_j(y^{j;\nu+1})}{\partial y_k}
} \, \left[ \, \nabla p_{jk}^{\rm cvx}(x^{\nu}) +  
\nabla p_{jk}^{\rm diff}(x^{\nu}) - \partial (-p_{jk}^{\rm cve})( x^{\nu+1} ) \, \right] 
\\ [0.2in]
= \, \displaystyle{
\sum_{k \in {\cal K}_{\phi_j}^+(y^{j;\nu})}
} \, \displaystyle{
\frac{\partial \phi_j(y^{j;\nu})}{\partial y_k}
} \, \left[ \, \partial p_{jk}^{\rm cvx}( x^{\nu+1} ) + \nabla p_{jk}^{\rm cve}(x^{\nu}) +  
\nabla p_{jk}^{\rm diff}(x^{\nu}) \, \right]  \\ [0.2in]
\epc + \displaystyle{
\sum_{k \in {\cal K}_{\phi_j}^-(y^{\nu})}
} \, \displaystyle{
\frac{\partial \phi_j(y^{j;\nu})}{\partial y_k}
} \, \left[ \, \nabla p_{jk}^{\rm cvx}(x^{\nu}) +  
\nabla p_{jk}^{\rm diff}(x^{\nu}) - \partial (-p_{jk}^{\rm cve})( x^{\nu+1} )
\, \right] \\ [0.2in] 
\epc + \displaystyle{
\sum_{k \in {\cal K}_{\phi_j}^+(y^{j;\nu+1})}
} \, \left( \, \displaystyle{
\frac{\partial \phi_j(y^{j;\nu+1})}{\partial y_k}
} - \displaystyle{
\frac{\partial \phi_j(y^{j;\nu})}{\partial y_k}
} \, \right)
\, \left[ \, \partial p_{jk}^{\rm cvx}( x^{\nu+1} ) + \nabla p_{jk}^{\rm cve}(x^{\nu}) +  
\nabla p_{jk}^{\rm diff}(x^{\nu}) \, \right] \\ [0.2in]
\epc + \, \displaystyle{
\sum_{k \in {\cal K}_{\phi_j}^-(y^{j;\nu+1})}
} \, \left( \, \displaystyle{
\frac{\partial \phi_j(y^{j;\nu+1})}{\partial y_k}
} - \displaystyle{
\frac{\partial \phi_j(y^{j;\nu})}{\partial y_k}
} \, \right)
\, \left[ \, \nabla p_{jk}^{\rm cvx}(x^{\nu}) + \nabla p_{jk}^{\rm diff}(x^{\nu}) 
- \partial (-p_{jk}^{\rm cve})( x^{\nu+1} ) \, \right] 
\end{array} \]
By the Lipschitz continuity of $\phi_j$ and $P^j$, we deduce
\[
\left| \,  \displaystyle{
\frac{\partial \phi_j(y^{j;\nu+1})}{\partial y_k}
} - \displaystyle{
\frac{\partial \phi_j(y^{j;\nu})}{\partial y_k}
} \, \right| \, \leq \, \mbox{Lip}_{\nabla \phi} \, \mbox{Lip}_P \, 
\| \, x^{\nu+1} - x^{\, \nu} \, \|_2.
\]
Hence it follows that
\[
\partial \left[ \, \wh{\theta}_j(\bullet;x^{\nu}) \, \right](x^{\nu+1})
\, \subseteq \, \wh{\partial} \theta_j(x^{\nu+1}) +
\wh{\mbox{Lip}} \, \| \, x^{\nu+1} - x^{\nu} \, \|_2 \, {\cal B}(0,1)
\]
for some aggregate constant $\wh{\mbox{Lip}} > 0$.  With the vector $c^{\, \nu}$ satisfying
(\ref{eq:stationarity of subproblem}), it follows that there exists a vector 
$\wh{c}^{\, \nu}$ belonging to the convex hull of $\displaystyle{
\bigcup_{j \in {\cal M}_{\nu+1/2}}
} \, \wh{\partial} \theta_j(x^{\nu+1})$ such that 
\[
\| \, c^{\, \nu} - \wh{c}^{\, \nu} \, \|_2 \, \leq \, 
\wh{\mbox{Lip}} \, \| \, x^{\nu+1} - x^{\nu} \, \|_2.
\] 
From (\ref{eq:stationarity of subproblem}), we deduce, 
with $\wh{\mbox{Lip}}^{\, \prime} \, \triangleq \, \wh{\mbox{Lip}} + \rho$,
\[
0 \, = \, \wh{c}^{\, \nu} + w^{\, \nu} + p^{\, \nu}, \epc \mbox{where } \
\| \, p^{\, \nu} \, \|_2 \, \leq \, \wh{\mbox{Lip}}^{\, \prime} \, 
\| \, x^{\nu+1} - x^{\nu} \, \|_2,
\]
which implies, 
\begin{equation} \label{eq:bounding distances}
\mbox{dist}\left( 0, \, \mbox{conv} \left[ \, \displaystyle{
\bigcup_{j \in {\cal M}_{\nu+1/2}}
} \, \wh{\partial} \theta_j(x^{\nu+1}) \, \right] + {\cal N}(x^{\nu+1};X) \right) 
\, \leq \, \wh{\mbox{Lip}}^{\, \prime} \, \| \, x^{\nu+1} - x^{\nu} \, \|_2.
\end{equation}
Note that there is a difference between the left-hand distance 
in (\ref{eq:bounding distances})
and the right-hand distance in (\ref{eq:distance bound}); namely, the former involves
the union over $j \in {\cal M}_{\nu+1/2} \subseteq 
{\cal M}_{\Theta}^{\, \varepsilon_{\nu}}(x^{\nu})$ 
while the latter involves the union over $j \in {\cal M}_{\Theta}(x^{\nu+1})$.
One situation in which this difference can be reconciled is when 
${\cal M}_{\Theta}(x^{\infty})$
is a singleton for all $x^{\infty} \in \Omega$.  This assumption implies that for 
all $\nu \in \kappa$ sufficiently
large, it holds that ${\cal M}_{\Theta}(x^{\nu+1}) = M_{\nu+1/2} = 
{\cal M}_{\Theta}(x^{\infty})$.  

\subsection{A final sequential convergence result with rates}

With all the preparations, we are now ready to state the following sequential
convergence result.  No additional proof is needed.
We should point out that the result requires several additional
assumptions beyond those for the subsequence convergence 
Theorem~\ref{th:subsequential convergence I}, it may be possible to relax
some of these, particularly those that involve all the accumulation points.
At this time, we have not attempted to establish such a refined analysis.

\begin{theorem} \label{th:final sequential differentiable} \rm
Assume the setting of Theorem~\ref{th:unit step}.  Suppose that the
(C$^{\, 2}$DC$^{\, 2}$) condition holds at all $\bar{x} \in X$. 
Let $\{ x^{\nu} \}$ be a sequence such that each $x^{\nu+1}$ is the unique 
optimal solution of (\ref{eq:final subproblem}) with $\wh{\theta}_j(\bullet;x^{\nu})$
given by (\ref{eq:final subproblem theta}).  Let 
$C > 0$ be a constant such that with $\rho > 2C$ this sequence $\{ x^{\nu} \}$
satisfies:

\gap

$\bullet $ $\Theta_{\max}(x^{\nu+1}) - \Theta_{\max}(x^{\nu}) \leq \
-\left( \, \displaystyle{
\frac{\rho}{2}
} - C \, \right) \, \| \, x^{\nu+1} - x^{\nu} \, \|_2^2$ for all $\nu$.

\gap

Then the following two statements hold:

\gap

{\bf (A) (Subsequential convergence})
the sequence $\{ x^{\nu} \}$ is bounded and every one of its accumulation
points is a directional stationary solution of (\ref{eq:unified framework});
moreover, either the sequence of iterates $\{ x^{\nu} \}$ converges finitely, or the 
sequence of objective values $\{ \Theta_{\max}(x^{\nu}) \}$ converges monotonically
to some value $\Theta_{\infty}$ with $\Theta_{\max}(x^{\nu}) > \Theta_{\infty}$
for all $\nu$.

\gap

{\bf (B) (Sequential convergence}) Under the following additional assumptions:
%
%
%
\gap

$\bullet $ the local uni-signed partial derivatives property holds for the outer
functions $\phi_j$ at all $x^{\infty}$ in the set $\Omega$
of accumulation points of the sequence $\{ x^{\nu} \}$;

\gap

$\bullet $ the uniform KL property holds at all $x^{\infty} \in \Omega$;

\gap

$\bullet $ the index set ${\cal M}_{\Theta}(x^{\infty})$ is a singleton for all
$x^{\infty} \in \Omega$;

\gap 

$\bullet $ the extended-valued function $\Theta_{\max} + \delta_X$  has a power
desingularization function at every $x^{\infty} \in \Omega$,

\gap

it holds that
there exist a positive scalar $\eta$ and a nonnegative scalar $\zeta < 1$ such that
\begin{equation} \label{eq:KL with power for Theta}
\| \, x^{\, \nu+1} - x^{\, \nu} \, \|_2 \, \geq \, \displaystyle{
\frac{1}{\eta}
} \, \left[ \, \Theta_{\max}(x^{\, \nu+1}) 
- \Theta_{\infty} \, \right]^{\zeta} \epc \forall \, \nu \mbox{ sufficiently large}.
\end{equation}
Consequently, the conclusions of Theorem~\ref{th:sequence converges under KL} hold
for the sequence $\{ x^{\nu} \}$; moreover, the limit of this sequence is a 
directional stationary point of (\ref{eq:unified framework}).  \hfill $\Box$
\end{theorem}

\section{Concluding Remarks} 

After a review of some old results and the derivation of some new ones for the class of
quasi-dc functions, we have presented a unified iterative descent convex-programming
based algorithm and established 
its subsequential convergence to a weak directional stationary point for a broad class of 
composite quasi-dc functions.  An enhancement of the algorithm that is computationally
more demanding is also presented that can obtain a directional stationary solution.
The construction of the descent subproblems in the algorithms depends on the matching 
properties of the outer and inner functions defining the compositions.  Sequential
convergence and rates of convergence of a simplified version of the main algorithm 
are also established under additional assumptions, some of which may be difficult 
to relax in general and yet may be
simplified for problems with favorable properties.  Due to the length of the paper,
we have omitted further analysis of the algorithm for all the special cases, of 
which there are plenty.  Finally,
numerical results are presently being planned, important applications will be
reported separately, and extensions to composite Heaviside
optimization are also part of the future research.

\end{document}